\documentclass[sn-mathphys,Numbered]{sn-jnl}

\usepackage{graphicx}
\usepackage{comment}
\usepackage{subcaption}
\usepackage{amsmath,amssymb,amsfonts}
\usepackage{amsthm}
\usepackage{mathrsfs}
\usepackage[title]{appendix}
\usepackage{xcolor}
\usepackage{textcomp}
\usepackage{manyfoot}
\usepackage{booktabs}
\usepackage{listings}
\usepackage{setspace}
\usepackage{pdfpages}
\usepackage{float}
\usepackage[normalem]{ulem}

\usepackage{multirow}
\usepackage{array}
\usepackage{booktabs}
\usepackage{makecell}
\usepackage{tabularx}

\usepackage{algorithm}
\usepackage{algorithmicx}
\usepackage{algpseudocode}

\usepackage{cleveref}

\newtheorem{theorem}{Theorem}

\newtheorem{lemma}{Lemma}
\newtheorem{assumption}{Assumption}
\newtheorem{corollary}{Corollary}

\theoremstyle{definition}

\theoremstyle{remark}

\renewcommand\proofname{\textbf{Proof}}

\newcommand{\N}{\mathcal{N}}

\newcommand{\F}{\mathcal{F}}
\newcommand{\E}{\mathbb{E}}
\newcommand{\R}{\mathbb{R}}

\newcommand{\kg}{\kappa_{\nabla}}
\newcommand{\kuno}{\kappa_1}
\newcommand{\kdue}{\kappa_2}
\newcommand{\ktre}{\kappa_3}
\newcommand{\kfour}{\kappa_4}
\newcommand{\kfive}{\kappa_5}
\newcommand{\ksei}{\kappa_6}

\usepackage{textcase}
\newcommand{\name}{\textsc{psg\_leco}}
\newcommand{\names}{\textsc{psg\_leco }}
\newcommand{\ipas}{\textsc{ipas}}
\newcommand{\sqp}{\textsc{sqp\_adam}}

\begin{document}


\title{A Projected Stochastic Gradient Method for Finite-Sum Problems with Linear Equality Constraints}

\author[1]{\orcid{https://orcid.org/0000-0001-5195-9295}\fnm{Nata\v sa} \sur{Krklec Jerinki\'c}}
\email{natasa.krklec@dmi.uns.ac.rs}

\author[2]{\orcid{https://orcid.org/0000-0003-4826-1114}\fnm{Benedetta} \sur{Morini}}
\email{benedetta.morini@unifi.it}

\author[2]{\orcid{https://orcid.org/0000-0002-9213-3622}\fnm{Mahsa} \sur{Yousefi}}
\email{mahsa.yousefi@unifi.it}

\affil[1]{
\orgdiv{Department of Mathematics and Informatics, Faculty of Sciences},
\orgname{University of Novi Sad},
\orgaddress{
\street{Trg D. Obradovi\' ca 4},
\city{Novi Sad},
\postcode{21000},
\country{Serbia}}}

\affil[2]{
\orgdiv{Department of Industrial Engineering},
\orgname{University of Florence},
\orgaddress{
\street{Viale G.B. Morgagni 40},
\city{Florence},
\postcode{50134},
\country{Italy}}}


\abstract{
A stochastic gradient method for finite-sum minimization subject to deterministic linear constraints is proposed and analyzed. The procedure presented adapts the projected gradient method on a convex set to the use of both a stochastic gradient and a possibly inexact projection map. Under standard assumptions in the field of stochastic gradient methods, we provide theoretical results in agreement with the theory for unconstrained problems. Numerical results are presented to show the practical behavior of the procedure.
}



\maketitle

\section{Introduction}\label{Sec.intro}
We consider the finite-sum optimization problem with linear equality constraints 
 \begin{equation} \label{problem}
\min_{x \in S} f(x)=\frac{1}{N}  \sum_{i=1}^{N}  f_i(x), \quad S=\{x \in \mathbb{R}^n \; | \; Ax=b\},
\end{equation}
where the functions  $ f_i:\mathbb{R}^n \to \mathbb{R},\, i=1,...,N$, are continuously-differentiable, $A \in \mathbb{R}^{m \times n}$, $m<n$,  is a full row-rank matrix  and $b$ is a vector in $\mathbb{R}^m$.

Problems of minimizing finite-sums are often encountered in applications, such as least-squares approximation and Machine Learning within training phase where parameters of a model function are optimized. This usually assumes a large number of data which corresponds to a large $N$ in \eqref{problem}. The so-called Big Data setup motivates stochastic optimization approach since evaluating the whole function and/or its derivatives is too expensive to be treated by classical, deterministic methods. This raised a number of  first-order stochastic strategies proposing different function and/or gradient approximations \cite{BOTTOU}. 
Moreover, analogously to deterministic optimization,  second-order information can speed up convergence, and this yielded a new direction in stochastic optimization research based on appropriate Hessian. A special case of this approach leans on spectral methods and stochastic  Barzilai-Borwein step-sizes, see e.g., \cite{SVRGBB,VRIT,Slises,BMY_JCAM}.

Subsampling is a main-stream procedure for obtaining approximations to function and/or gradient evaluations; subsampling strategies range from the mini-batch approach where the sample size is usually small and fixed (e.g., \cite{SGD}),
to increasing sample size strategies where the full sample is eventually reached (e.g., \cite{FriedSch}),
with many variations in between (\cite{LISA,Bastin,VRIT} to name just a few). 
Determining a suitable step-size sequence is also a key component in stochastic optimization field.
In addition to prefixed constant step-size sequences and diminishing step-sizes, globalization strategies such as line search and trust region can be adapted to the stochastic framework (see e.g. \cite{sirtr, PS,TRish})
but the majority of such approaches require at least approximate function evaluations. 

Finite-sum problems can also come with constraints incorporating prior to knowledge and physical meaning \cite{curtis_2026,CurtRobiZhou24, curtis_siopt, mahoney_siopt}. 
In this work, we propose a stochastic projected gradient method for problem (\ref{problem}) based on the projected gradient method. Function evaluations are not required while a mini-batch approach is employed to compute  stochastic gradients. Regarding the mini-batch strategy, the set of indices in the sum (\ref{problem}) is divided into $r$ mini-batches which can be redefined possibly at each iteration. At any iteration $k$, each mini-batch can be selected with the same probability; a gradient estimate is calculated on the sampled mini-batch and by using an appropriate scaling that provides an unbiased estimate of the full gradient. 
In order to handle constraints,   we allow the use of inexact projections, especially suited in the presence of a large number of constraints (e.g., \cite{krejic2025ipas}); specifically, inexact but controlled projections  provide a nonmonotone decay of the infeasibility measure. 
The proposed algorithm is accompanied by a theoretical analysis that establishes convergence results, distinguishing constant and diminishing  step-size sequences, in line with the existing literature. Convexity of the objective function is not required. 
The numerical behavior of the method is shown considering some step-size selections in agreement with the theory.

{\textbf{Outline of the paper.}} 
\S \ref{Sec.method} provides the new method while \S \ref{Sec.analysis} is devoted to its theoretical analysis. \S  \ref{sec:offosqp} provides a comparison with related contributions in the literature. \S \ref{Sec.NumericalResults} is dedicated to  the experimental configuration while numerical results are presented in \S \ref{Sec.Num}.
\S \ref{sec.con} draws the main  conclusions.

{\textbf{Notations.}} The symbol $\|\cdot\|$ indicates the Euclidean norm.  $Pr(\cdot)$ and   $\E[\cdot]$ represent the probability function and  expected value, respectively.
\section{Description of the method}\label{Sec.method}

In this section,  we introduce our Projected Stochastic Gradient method for Linear Equality COnstrained problems, named the \names  method, to solve the problem  \eqref{problem}. We start by describing two main tasks in the algorithm: the construction of a stochastic gradient and the computation of the projection of a point in $\mathbb{R}^n$ onto $S$.

At the $k$-th iteration, given $x_k\in \mathbb{R}^n$, a stochastic gradient $g_k$ is computed as a mini-batch gradient using the following strategy.  
Let us define a partition $\{\N_k^i\}_{i=1}^r$ of $\N=\{1, \ldots,n\}$ into $r$ disjoint mini-batches at iteration $k$, namely
  \begin{equation} \label{partition}
\N=\bigcup_{i=1}^{r}\N_k^i, \quad \mbox{where} \quad \N_k^i \cap \N_k^j = \emptyset, \quad \mbox{for all } i \neq j. 
\end{equation}
This partition can be fixed or varying along the iterations. 
Then we let  $g_k^{(i)}$ be the eligible mini-batch gradients associated with the partition in (\ref{partition}), i.e.,  
\begin{equation} \label{gki}
g_k^{(i)}=\frac{r}{N}  \sum_{j\in \N_k^i}  \nabla f_j(x_k), \quad i=1,\ldots,r,
\end{equation}
and $g_k$ be our stochastic gradient that corresponds to a uniformly and randomly selected mini-batch  $\N_k^i$ from the partition. By 
\begin{equation} \label{gkdistribution}
Pr\left( g_k=g_k^{(i)} \mid \F_k\right)=\frac{1}{r}, \quad i=1,\ldots,r, 
\end{equation}
where $Pr(\cdot)$ represents the probability of the outcomes, and $\F_k$ is a $\sigma$-algebra generated by  $x_0, \ldots, x_k$, i.e., by $g_0,\ldots, g_{k-1}$, we have  
\begin{equation}\label{unbiased_dk}
    \E[g_k\, |\F_k] = \nabla f(x_k).
\end{equation}

Regarding the projection map onto $S$, since  $A \in \mathbb{R}^{m \times n}$ has full rank, the orthogonal projection $\pi_S(y)$ of any given point $y\in \R^{n}$ onto \( S \)  takes the form
\begin{equation} \label{exactprojection}
\begin{aligned}
        \pi_{S}(y)&=y-A^T\lambda(y),     \\
    \lambda(y)&=(AA^T)^{-1} (Ay-b). 
\end{aligned}
\end{equation}
The computation of $\pi_{S}(y)$ is viable for moderate values of $m$ as it depends on the solution of the linear system 
\begin{equation}\label{res}  
      AA^T \, \widetilde{\lambda}(y) = Ay-b.  
\end{equation} 
More generally, it is advisable to allow for inexact projections of the form
\begin{equation}\label{inexactprojection} 
\begin{aligned}
  \widetilde{\pi}_{S}(y)&=y-A^T\widetilde{\lambda}(y), \\
     \widetilde{\lambda}(y) &=(AA^T)^{-1}( Ay-b+r(y)),   
\end{aligned}
\end{equation}
and $r(y)$ denotes the residual vector in the solution of (\ref{res}).

Algorithm \ref{alg1} sketches the $k$-th iteration of our procedure. Step 1 indicates the hyper-parameters required for execution: two nonnegative sequences $\{\eta_k\}$ and $\{\mu_k\}$ that control inexactness of the projection onto $S$, a positive step-size related sequence $\{\alpha_k\}$, two positive scalars $\delta_{\ell},\, \delta_u$ that define the projection $\pi_{\delta}(\cdot)=\max\left\{ \delta_{\ell}, \min\left\{\cdot, \delta_u\right\}\right\}$ employed in the computation of the step-size.

Step 3 refers to the construction of the stochastic gradient $g_k$ as described above. Steps 4 and 5 concern the computation of the  iterate  $x_{k+1}$. Specifically, first, the vector $y_k$ is formed using the step-length $\Delta_k$ fixed in the previous iteration, then  $y_k$ is projected onto $S$ and this gives rise to the new iterate $x_{k+1}$. Inequality
(\ref{resb}) controls the accuracy in the calculation of the projection of $y_k$ by means of scalars $\eta_k$ and $\mu_k$. We observe that if $\eta_k=\mu_k=0$, $\forall k$, then the iterates $x_k$, $k\ge 1$, are feasible irrespective of $x_0$.

Steps 6 and 7 are devoted to the  computation of the step-length to be used at the subsequent iteration. The choice of  a  (positive)   scalar $\delta_{k}$ offers a variety of options and we discuss some possible adaptive  choices in Section~\ref{Sec.NumericalResults}.

At the end of iteration $k$, $\Delta_{k+1}$ is formed by means of $\alpha_{k+1} $ and $ \pi_{\delta}(\delta_{k})$. Consequently, $\Delta_{k+1}$ is deterministic conditioning on $x_{k+1}$ and this feature is crucial in the analysis of the procedure. The predetermined  sequences $\{\alpha_k\}$ and $\{\mu_k\}$ affect the convergence properties, as shown in the subsequent theoretical analysis.

We conclude this section by giving more insight into  the rule (\ref{resb}) in the \name\ Algorithm. Inexact projections are convenient in case the dimension $m$ is large as  $\widetilde \pi_S(\cdot)$ can be computed using linear iterative solvers, such as the Conjugate Gradient method \cite{CG}. Since we accept every iterate without performing an acceptance test, it holds (see the next  Lemma \ref{feasibility})
$$\|Ax_{k+1}-b\|=\|r(y_k)\|,$$ 
where $r(y_k)$ is given in (\ref{inexactprojection}). Thus, the inequality (\ref{resb}) implies an explicit relation between the infeasibility measure at $x_k$ and $x_{k+1}$, i.e., $\|Ax_{k}-b\|$  and $\|Ax_{k+1}-b\|$, respectively.
In addition to such characteristic, the inequality  (\ref{resb}) is inspired by the standard control $\|r(y_k)\|\le \chi_k \|Ay_k-b\|$, $ \chi_k\in (0,1)$, for the approximate solution of the linear system (\ref{res}).
In fact, by the definition of $y_k$, i.e. $y_k=x_k - \Delta_k g_k$, it holds
$$
\|r(y_k)\|\le \chi_k \|Ay_k-b\|\le \chi_k \|Ax_k-b\|+\chi_k\Delta_k\|Ag_k\|.
$$
Now,  (\ref{resb}) rephrases the above inequality using the prescribed term $\mu_k$ instead of the random quantity $\chi_k\Delta_k\|Ag_k\|$.

Finally, we note that inexact projections are   allowed in \cite{krejic2025ipas} by imposing a control of the form $\|r(y_k)\| \leq \eta_k^{{\rm IPAS}}$,  where $r(y_k)$ is given in (\ref{inexactprojection}), the sequence $\{\eta_k^{{\rm IPAS}}\}$ is prefixed and  $\sum_{k=0}^\infty (\eta_k^{{\rm IPAS}})^2 < \infty$. Contrary to \ipas,  our upper bound on $\|r(y_k)\|$ is adaptive due to the presence of $x_k$. The assumptions on the sequences $\{\eta_k\}$ and $\{\mu_k\}$ required to ensure convergence to a stationary point of problem \eqref{problem} are introduced in the next section.

\smallskip
\begin{algorithm}[t]
\small
\begin{algorithmic}[1]
\State{Choose an initial iterate $x_0 \in \mathbb{R}^n$,    a positive sequence $\{\alpha_k\}$, nonnegative sequences  $\{\eta_k\}$ and $\{\mu_k\}$ with $\eta_k\in [0, \eta)$,  $  \eta<1$, positive scalars $\delta_{\ell}<\delta_u<\infty$, and an initial step-size  $\Delta_0\in [\alpha_0\delta_{\ell},\, \alpha_0 \delta_u ]$.}

\For{$k = 0,1,\ldots$}
\State{Select uniformly at random a mini-batch $\mathcal{N}_k^i$ from the partition 
$\{\mathcal{N}_k^1, \ldots, \mathcal{N}_k^r\}$ as in 
\hspace*{10pt} \eqref{partition}, and set $g_k = g_k^{(i)}$ as in \eqref{gki}.}
\State{Set $y_k=x_k - \Delta_k g_k$.}
\State{Compute $x_{k+1}=\widetilde \pi_S(y_k)$ as in \eqref{res}, where $r(y_k)$ satisfies
\begin{equation}\label{resb}
\|r(y_k)\|\leq \eta_k \|Ax_k-b\| +\mu_k.
\end{equation}
}
\State{Choose a  scalar $\delta_{k}$.} 
\State{Compute $ \Delta_{k+1} = \alpha_{k+1} \pi_{\delta}(\delta_{k})=\alpha_{k+1}\max\left\{ \delta_{\ell}, \min\left\{\delta_k, \delta_u\right\}\right\}$.} 
\State{Set $k = k+1$.}
\EndFor
\end{algorithmic}
\caption{{\sc {\bf  PSG\_LECO}} \label{alg1}} 
\end{algorithm}


\section{Theoretical Analysis}\label{Sec.analysis}

In this section, we analyze the theoretical properties of the \names method. We make the following standard assumptions on the objective function.
\vskip 5pt
\begin{assumption}\label{A1}  
 The objective function $f:\mathbb{R}^{n} \to \mathbb{R}$ is continuously differentiable on $\mathbb{R}^n$. The gradient of $f$ is Lipschitz continuous with constant $L>0$. 
\end{assumption}
\vskip 5pt
\begin{assumption}\label{A2}
 The objective function $f:\mathbb{R}^{n} \to \mathbb{R}$ is bounded from below by $f^*$,
\begin{equation} \label{boundedf}
f(x) \geq f^*, \qquad \forall\   x \in \mathbb{R}^n.
\end{equation}

\end{assumption}
\medskip
The following result introduces an optimality measure $d(x)$ for problem (\ref{problem})  which will be used in our subsequent analysis.
\medskip

\begin{theorem} \label{th.theorem1v2} 
Assume that $f$ is continuously differentiable in an open set containing $S$ and let 
\begin{equation}\label{d_def}
d(x)=\pi_{S}(x-\nabla f(x))-x.
\end{equation}
Then, it holds $d(\bar x) = 0$, $\bar x\in S$, if and only if $\bar x$ is a stationary point for \eqref{problem}.
\end{theorem}
\begin{proof} 
See  \cite[Lemma 2.1]{birgin2000nonmonotone}.
\end{proof}

Under suitable assumptions, our analysis will provide the following results which are in accordance with the corresponding algorithms for unconstrained optimization:
each limit point of the sequence $\{x_k\}$ is feasible;
if $\alpha_k=\alpha$, for all $k$, then  the expected sum of average-squared norms of $d(x_k)$ is bounded and decreases with $\alpha$;
if $\{\alpha_k\}$ is a diminishing sequence, then  $\|d(x_k)\|$ cannot stay bounded away from zero almost surely.

The first step in our analysis is to characterize the infeasibility of the iterates generated by Algorithm~\ref{alg1}. To this end, we introduce the measures 
\begin{equation}
    \label{e}
  \widetilde e(y)=Ay-b, \qquad    e(y)=\|\widetilde e(y)\|,
\end{equation}
and make the following assumption on the sequences $\{\eta_k\}$ and $\{\mu_k\}$.
\smallskip
\begin{assumption}\label{Ainf}
The nonnegative sequence $ \{\eta_k\}$   satisfies  $ \eta_k\le  \eta<1$, $\forall k\ge0$. The nonnegative sequence $\{\mu_k\}$  converges to zero R-linearly. 
\end{assumption}
\smallskip
\begin{lemma} \label{feasibility}
Suppose that Assumptions  \ref{A1} and  \ref{Ainf} hold. Then,
  \begin{description}
 \item{a)} For $y\in \mathbb{R}^n$, it holds
  \begin{equation}\label{infeas_r}
   \widetilde e(\widetilde \pi_S(y)) =  -r(y).
 \end{equation}
\item{b)}  The iterate $x_{k+1}$ satisfies  
  \begin{equation}\label{infeas_x}
   \widetilde e(x_{k+1}) =  -r(y_k), \quad  \text{for all} \quad k\ge 0.
 \end{equation}
\item{c)}  It holds
\begin{equation}\label{infeas_rid}
   e(x_{k})\le  \kappa_e q^{k}, \quad  \text{for all} \quad k\ge 0,
\end{equation}
for some constant $\kappa_e>0$ and $q \in ( \eta,1)$ and every limit point of $\{x_k\}$ is feasible.   
\item{d)}  It holds 
\begin{equation}\label{infeas_4}
   e(x_k)\le \kappa_e, \quad  \text{for all} \quad k\ge 0.
\end{equation}
\end{description}
\end{lemma}
\begin{proof} 
(a) From (\ref{inexactprojection}), (\ref{res}), and (\ref{e}), we get the expression in (\ref{infeas_r}) as follows:
$$
\widetilde e(\widetilde \pi_S(y)) =A\widetilde \pi_S(y)- b = A\left (y-A^T\tilde{\lambda}(y)\right )-b=-(AA^T\tilde{\lambda}(y)-Ay+b).
$$
\vskip 5pt \noindent
b) Equality (\ref{infeas_x}) follows from (\ref{infeas_r}) and Step 4 of Algorithm \ref{alg1}.
\vskip 5pt \noindent
c)
Equations 
(\ref{infeas_x}), (\ref{resb}), (\ref{res}) and Assumption \ref{Ainf} give
$$e(x_{k+1})\le  \eta e(x_k)+\mu_k.$$ Therefore, for all $k$, there holds   
\begin{equation}\label{new1} e(x_k)\le  \eta ^k e(x_0)+v_k,\qquad  v_k=\sum_{i=1}^{k}  \eta ^{i-1} \mu_{k-i}.\end{equation}
Since $\{\mu_k\}$ converges to zero R-linearly and $ \eta <1$,  $\{v_k\}$ converges to zero R-linearly which implies 
\eqref{infeas_rid}, see e.g., \cite[Lemma 4.2]{NKNKJ-NM}.
Thus, $\lim_{k\rightarrow \infty}   e(x_{k})=0$.
\vskip 5pt \noindent
d) The claim trivially follows from (\ref{infeas_rid}).
\end{proof}

We note that feasibility is eventually enforced, i.e., $\lim_{k \to \infty} e(x_k)=0$, whenever $\{\mu_k\}$ tends to zero. On the other hand, the summability of $\sum_{k=0}^{\infty}\mu_k$ is required to prove the main Theorem \ref{th.theorem2}; this request motivates our Assumption \ref{Ainf} on $\{\mu_k\}$.

 Now  we proceed with an intermediate result on  
the step $d_k$  taken at iteration $k$
\begin{equation}\label{step} 
    d_k=x_{k+1}-x_k.
\end{equation}
We denote   
\begin{equation}\label{e.matrices}
    D=A^T(AA^T)^{-1}, \qquad  P_A = I - DA, \qquad \kappa_D=\|D\|,
\end{equation}
and introduce the following assumptions.
\medskip
\begin{assumption}\label{Anabla}
For some positive constant $\kg$, the sequence  $\{x_k\}$ is either feasible or such that $$\|\nabla f(x_k)\| \leq \kg, \qquad \text{for all} \quad k\ge 0.$$  
\end{assumption}
\smallskip
Note that bounded gradients are assumed only when the iterates are infeasible. We now consider the following assumption, which states that the conditional variance of the sampled gradient is uniformly bounded. 

\medskip
 
\begin{assumption}\label{A3}
There exists a positive constant  \( \nu > 0 \)  such that 
\begin{equation}\label{var_bound1}
    \E\big[\|g_k-\nabla f (x_k)\|^2\, |\F_k\big] \leq \nu.
\end{equation}
\end{assumption}

\medskip

\begin{lemma}\label{lem.lemma4IPSG} 
Let $\{x_k\}$ be generated by Algorithm \name. Suppose that Assumption  \ref{A1} holds.
\begin{description}
\item{a)} The  step $d_k$ defined in (\ref{step}) satisfies
\begin{equation} 
\label{dkEnew}
    \E{\big[d_k\, | \F_k\big]}
    =  \Delta_{k} d(x_k)-(1-\Delta_k)D \widetilde{e}(x_k)  - D \E{\big[r(y_k)\, | \F_k\big]},
\end{equation}
where $D$ is defined in (\ref{e.matrices}).
\item{b)} Suppose further that Assumptions \ref{Ainf} and \ref{Anabla}   hold. Then, $d(x_k)$ defined in (\ref{d_def}) satisfies
\begin{equation}\label{descentd(xk)ines}
    d(x_k)^T \nabla f(x_k) \leq -\|d(x_k)\|^2 + \kuno e(x_k),   
\end{equation}
for some positive  constant  $\kuno$.
\item{c)} Suppose further that Assumptions \ref{Ainf}, \ref{Anabla} and \ref{A3} hold, and that $(1-\Delta_k)^2\le \Delta^*,$ for all $k$ and some positive $\Delta^*$. Then,  
\begin{equation}\label{var_bound2}
    \E\big[\|d_k\|^2\, |\F_k\big] \leq \kdue \Delta_k^2+ \ktre   q^{2k} + \kfour \mu_k +  2 \Delta_k^2\|d(x_k)\|^2,
\end{equation}
for some positive constants $\kdue, \, \ktre,\kfour$.
\end{description}
\end{lemma}
\begin{proof} 
a) First, we note that \eqref{step}, \eqref{e.matrices}, \eqref{inexactprojection} and \eqref{res} give
\begin{eqnarray}
d_k &=&\widetilde{\pi}_{S}(x_k-\Delta_k g_k)-x_k \nonumber\\
&=& x_k-\Delta_k g_k-D(Ax_k-\Delta_kA g_k-b+r(y_k))-x_k  \nonumber\\
& =& -\Delta_kP_Ag_k -D\widetilde e(x_k)- Dr(y_k),\label{step1}
\end{eqnarray}
and that (\ref{d_def}), (\ref{exactprojection})
give
\begin{equation}\label{d2}
d(x_k)
=-P_A\nabla f(x_k)-D\widetilde e(x_k).
\end{equation}
Hence, using \eqref{unbiased_dk} and that $\Delta_k$ is $\mathcal{F}_k$ -measurable, we have 
\begin{eqnarray*} 
\label{dkE}
    \E{\big[d_k\, | \F_k\big]} &=& -\Delta_k P_A \E{\big[g_k\, | \F_k\big]}-D \widetilde{e}(x_k)  -D \E{\big[r(y_k)\, | \F_k\big]} \\
    & = & -\Delta_kP_A \nabla f(x_k)-D \widetilde{e}(x_k) -D \E{\big[r(y_k)\, | \F_k\big]},
\end{eqnarray*}
and (\ref{d2}) concludes the proof.
\vskip 5pt \noindent
b) Using (\ref{d2}), $P_A D=0$, $P_A^T=P_A$ and $P_A^2=P_A$, it follows
\begin{eqnarray*}
    d(x_k)^T d(x_k)&=&\nabla f(x_k)^T P_A^2 \nabla f(x_k)+\|D\widetilde e(x_k)\|^2 \\ 
    &=&-\nabla f(x_k)^T d(x_k) -\nabla f(x_k)^TD\widetilde e(x_k) +\|D\widetilde e(x_k)\|^2.
\end{eqnarray*}
Hence, using \eqref{e}, \eqref{infeas_4} we obtain 
\begin{eqnarray*}
    \nabla f(x_k)^T d(x_k)&\le& -\|d(x_k)\|^2 + \|\nabla f(x_k)\|\,\kappa_De(x_k)
    +\kappa_D^2 e(x_k)^2 \\
    &\le& -\|d(x_k)\|^2 + \big(\kg\kappa_D+\kappa_D^2    \kappa_e \big)e(x_k).
\end{eqnarray*}
Thus, \eqref{descentd(xk)ines} holds with $\kuno=\kg\kappa_D+\kappa_D^2  \kappa_e$. 
\vskip 5pt \noindent
c) Using (\ref{step1}), (\ref{d2}), $\|P_A\|=1$, and (\ref{infeas_rid}) we obtain
\begin{eqnarray*}
    \|d_k-\Delta_kd(x_k)\|^2
    &=&\|-\Delta_k P_A(g_k-\nabla f(x_k))-(1-\Delta_k)D \widetilde e(x_k)-Dr(y_k)\|^2\nonumber \\
		   &\le&2\Delta_k^2 \|g_k-\nabla f(x_k)\|^2 +4 \|(1-\Delta_k)D \widetilde e(x_k)\|^2+ 4\|Dr(y_k) \|^2\nonumber \\
    & \le&    2\Delta_k^2 \|g_k-\nabla f(x_k)\|^2  +4\kappa_D^2  
   \Delta^* e(x_k)^2 +4\kappa_D^2 (\eta_ke(x_k)+\mu_k)^2  \nonumber\\
    &\le&     2 \Delta_k^2\|g_k-\nabla f(x_k)\|^2  +4\kappa_D^2 (\Delta^*+ \eta^2 ) \kappa_e^2 q^{2k} + 4\kappa_D^2 (2 \eta \kappa_e  q^k+\mu_k)\mu_k. 
\end{eqnarray*}
Now, using this last inequality, Assumption  \ref{A3} and the equation
\begin{eqnarray*}
\|d_k\|^2&=&\|d_k-\Delta_k d(x_k)+\Delta_k d(x_k)\|^2\le 2\|d_k-\Delta_k d(x_k)\|^2 +  2 \Delta_k^2\|d(x_k)\|^2 ,
\end{eqnarray*}
we obtain 
\begin{align}
    \E\big[\|d_k\|^2\, |\F_k\big] 
    &\leq 
    2 \E\big[\|d_k-\Delta_k d(x_k)\|^2\, |\F_k\big]
    +2\Delta_k^2 \E\big[\|d(x_k)\|^2\, |\F_k\big] \nonumber \\
    &\leq
    4\Delta_k^2 \E\big[\|g_k-\nabla f(x_k)\|^2\, |\F_k\big] +8\kappa_D^2 (\Delta^*+ \eta^2) \kappa_e^2 q^{2k} 
	\nonumber \\ 
    & \; + 8\kappa_D^2 (2 \eta \kappa_e q^k+\mu_k)\mu_k + 2\Delta_k^2\|d(x_k)\|^2. 
    \end{align}
Thus, the claim holds with $\kdue=  4\nu$, $\ktre= 8\kappa_D^2 (\Delta^*+ \eta^2 ) \kappa_e^2$,   $\kfour= 8\kappa_D^2 (2 \eta \kappa_e +\mu_{\max})$ where $\mu_{\max}=\max_k\, \{\mu_k\}$.
\end{proof}
\medskip

For sake of completeness, the following lemma rephrases the results above when  $x_0$ is feasible and the exact projection $\pi_S$ is used.
\medskip

\begin{corollary}
Let $\{x_k\}$ be generated by Algorithm \name. Suppose that Assumption \ref{A1} holds, $x_0\in S$, $\eta_k= \mu_k=0$ for all $k\ge 0$.
 \item{a)} The sequence $\{x_k\}$ is feasible.
 \item {b)} The  step $d_k$ defined in (\ref{step}) satisfies
\begin{equation} \label{dkE_feas}
    \E{\big[d_k\, | \F_k\big]}
    =  \Delta_{k} d(x_k).
\end{equation}
\item{c)} The direction $d(x_k)$ defined in (\ref{d_def}) satisfies
\begin{equation}\label{descentd(xk)ines_feas}
    d(x_k)^T \nabla f(x_k) \leq -\|d(x_k)\|^2.  
\end{equation}
\item{d)} Suppose further that Assumption \ref{A3} holds and that $(1-\Delta_k)^2\le \Delta^*$ for all $k$ and some positive $\Delta^ *$. Then,  
\begin{equation}\label{var_bound2_fis}
    \E\big[\|d_k\|^2\, |\F_k\big] \leq \kdue \Delta_k^2+  2 \Delta_k^2\|d(x_k)\|^2, \qquad \text{for some} \;\;\; \kdue>0.
\end{equation}
\end{corollary}

\begin{proof}
Item (a) follows from (\ref{new1}), while items (b)--(d) follow from Lemma~\ref{lem.lemma4IPSG}.
\end{proof}
\medskip
By the following theorem, we now analyze the behavior of the \names method under the assumption that the sequence $\{\alpha_k\}$ is diminishing.
\medskip

\begin{theorem} \label{th.theorem2} 
Suppose that Assumptions~\ref{A1}--\ref{A3} hold,    and that 
\begin{eqnarray}\label{alphakalpha0}
    \sum_{k=0}^\infty\alpha_k =\infty, \quad \sum_{k=0}^\infty \alpha_k^2 < \infty. 
\end{eqnarray}
Then, almost surely,
\begin{equation}\label{limTheorem0}
    \liminf_{k \rightarrow\infty} \|d(x_k)\| =0, 
\end{equation}
and the sequence of function values $\{f(x_k)\}$ converges. 
\end{theorem}
\begin{proof} 
Assumption \ref{A1} implies that 
\begin{equation}\label{eq.g_Lip}
    f(x) \leq f(y) + \nabla f(y)^T(x - y) +  \frac{L}{2} \|x - y\|^2, \qquad \forall \ x, y  \in \mathbb{R}^{n}.
\end{equation}
Hence, by (\ref{step}) we have
$$
f(x_{k+1}) \le f(x_k ) +  \nabla f(x_k)^Td_k + \frac{L}{2}\|  d_k\|^2.
$$
Taking the conditional expectation with respect to the $\sigma$-algebra $\mathcal{F}_k$, we obtain
$$
    \E{\big[f(x_{k+1})\, | \F_k\big]}\leq f(x_{k}) +  \nabla f(x_k)^T\E{\big[ d_k\, | \F_k\big]} + \frac{L}{2}   \,\E{\big[\|d_k\|^2\,|\F_k\big]}.
$$
Conditions (\ref{alphakalpha0}) imply $(1-\Delta_k)^2\le \Delta^*$ for some $\Delta^*>0$. Now, by  \eqref{dkEnew}--\eqref{var_bound2}, we have 
\begin{equation*}
\begin{aligned}
    \E{\big[f(x_{k+1})\, | \F_k\big]}
    &\leq f(x_{k}) - \Delta_k \|d(x_k)\|^2 +\kuno\Delta_k e(x_k) -  (1-\Delta_k) \nabla f(x_k)^T D \widetilde e(x_k)
    \nonumber \\ 
    & \;-\E{\big[\nabla f(x_k)^T D r(y_k)\, | \F_k\big]} + \frac{L}{2}  (\kdue \Delta_k^2+ \ktre {q^{2k}}+\kfour \mu_k+  2 \Delta_k^2\|d(x_k)\|^2) \nonumber \\ 
    &\leq f(x_{k}) - \Delta_k \|d(x_k)\|^2 +\kuno\Delta_k e(x_k) +  \sqrt{\Delta^*} \|\nabla f(x_k)\| \kappa_D e(x_k)
    \nonumber \\ 
    & \;+\|\nabla f(x_k)\| \kappa_D \E{\big[\|r(y_k)\|\, | \F_k\big]} + \frac{L}{2}  (\kdue \Delta_k^2+ \ktre {q^{2k}}+\kfour \mu_k+ 2\Delta_k^2\|d(x_k)\|^2).
    \label{inequality0}
\end{aligned}  
\end{equation*}

If $\{x_k\}$ is feasible, then $r(y_k)=0$. Otherwise, the upper bound ${\eta} e(x_k)+\mu_k$ on  $\|r(y_k)\|$ given  in \eqref{resb}  is $\F_k$-measurable,
and we obtain the following inequality that includes exact and inexact projections,
\begin{align}
    \E{\big[f(x_{k+1})\, | \F_k\big]}
    &\leq  f(x_{k}) - \Delta_k \|d(x_k)\|^2 +\big((1+\sqrt{\Delta^*})\kuno+ 
    \kg\kappa_D( \sqrt{\Delta^*}+ \eta) \big) e(x_k) \nonumber\\ 
    & \quad + \kg\kappa_D\mu_k + \frac{L}{2}  (\kdue \Delta_k^2+\ktre {q^{2k}}+\kfour\mu_k+  2\Delta_k^2 \|d(x_k)\|^2)\nonumber \\
    &\leq f(x_{k})  - \Delta_k (1-L\Delta_k) \|d(x_k)\|^2  +\kfive  q^{k}+\ksei \mu_k+\frac{L}{2}  \kdue \Delta_k^2,  
    \label{inequality}
\end{align}
with $\kfive=\Big(\kuno(1+\sqrt{\Delta^*})+     \kg\kappa_D( \sqrt{\Delta^*}+ \eta)\Big) \kappa_e+\frac{L}{2}\ktre$, $\ksei=\Big(\kg\kappa_D+\frac{L}{2}\kfour\Big)$ and 
the last inequality obtained using \eqref{infeas_rid}.

By construction $\Delta_k$ satisfies 
\begin{equation}\label{Deltak_bound}
    \alpha_{k}\delta_{\ell} \leq \Delta_k \leq \alpha_{k} \delta_u,
\end{equation}
hence
\begin{align}\label{positiveCoef}
    \E{\big[u_{k+1}\, | \F_k\big]}
    \leq u_k - \alpha_k\big( \delta_{\ell} - L \alpha_k\delta_u^2\ \big)\|d(x_k)\|^2  + \kfive q^{k} +\ksei \mu_k + \frac{L}{2}\kdue \alpha_k^2\delta_u^2   ,
\end{align}
where $u_k = f(x_k) - f^* $ from \eqref{boundedf}.
Since $\{\alpha_k\}$ is diminishing, let  $\bar{k}$ be the index such that  $\big( \delta_{\ell} -L \alpha_k\delta_u^2 \big) \geq  \frac{\delta_{\ell}}{2}$ for all $k\geq \bar{k}$. 
Hence, for $k\ge \bar k$
$$
\E\big[u_{k+1} \mid \mathcal{F}_k\big] \leq u_k - \alpha_k\frac{\delta_{\ell}}{2}\|d(x_k)\|^2 + \kfive  q^{k} +\ksei \mu_k  + \frac{L}{2}\kdue \alpha_k^2\delta_u^2   . 
$$
Since  $\sum_{k=0}^\infty \alpha_k^2$, $\sum_{k=0}^\infty  q^{k}$ and $\sum_{k=0}^\infty  \mu_k$ are summable, the Robbins–Siegmund supermartingale convergence Theorem \cite{robbins1971convergence} gives that 
$$\sum_{k=\bar{k}+1}^\infty \alpha_k \|d(x_k)\|^2 < \infty,$$
almost surely.
Now, assume by contradiction that $\|d(x_k)\| \geq c > 0$ for all $k > \bar{k}$. Thus
$$\sum_{k=\bar{k}+1}^\infty \alpha_k \|d(x_k)\|^2 \geq c^2 \sum_{k=\bar{k}+1}^\infty \alpha_k,
$$
which contradicts  \eqref{alphakalpha0}.  
Therefore, \eqref{limTheorem0} holds almost surely. Finally, from the Robbins–Siegmund supermartingale convergence theorem, we also conclude that $\lim_{k\rightarrow \infty} u_k$ exists and is finite almost surely.
\end{proof}
\medskip
The theorem above implies that almost surely there exists a subsequence $\{\|d(x_k)\|\}_{k\in K}$ of $\{\|d(x_k)\|\}$ convergent to zero; if $\{x_k\}_{k\in K}$ admits limit points, then such points are stationary. Now we analyze  the convergence of the \name\ using constant step-lengths and show that the expected sum of average-squared norms of $d(x_k)$ is bounded and decreases with $\alpha$.  
\medskip

\begin{theorem} \label{th.theorem3} 
Suppose that  Assumptions~\ref{A1}--\ref{A3} hold,  $\eta_k\le  \eta<1, \forall k$.
If  $\alpha_k = \alpha,\, \forall k\ge 0$, with $\alpha$   such that
\begin{equation}\label{alphak}
0<\alpha\le \frac{\delta_{\ell}}{2L \delta_u^2},
\end{equation}
then the iterates generated by \name
method satisfy 
\begin{equation}
    \lim_{K\rightarrow \infty} \E\left [\frac{1}{K} \sum_{k=0}^K\|d(x_k)\|^2\right]  \leq 
		\frac{\alpha L\kdue\delta_u^2 }{\delta_{\ell}},
\end{equation}
where $\kdue$ is a scalar in (\ref{var_bound2}). 
\end{theorem}
\begin{proof}
Condition  (\ref{alphak}) implies $(1-\Delta_k)^2\le \Delta^*,\, \forall  k$, and some positive $\Delta^*$.
Applying \eqref{Deltak_bound} and \eqref{alphak}  in \eqref{inequality} and taking total expectation, we obtain the following
\begin{align*}
    \E{\big[f(x_{k+1})\big]} \leq \E{\big[f(x_{k})\big]} - \frac 1 2 \alpha \delta_{\ell} 
    \E{\big[\|d(x_k)\|^2\big]} 
     +\kfive q^{k} +\ksei \mu_k
		+ \frac{L}{2}\kdue\alpha^2 \delta_u^2,
\end{align*}
and consequently 
\begin{align*}
    \E{\big[\|d(x_k)\|^2\big]} \leq \frac{2}{\alpha\delta_{\ell}}\left(\E{\big[f(x_k)\big]} - \E{\big[f(x_{k+1})\big]}\right) 
    + \frac{2\kfive}{\alpha \delta_{\ell}} q^{k}   + \frac{2\ksei}{\alpha \delta_{\ell}} \mu_k+ 
		\frac{\alpha L\kdue\delta_u^2}{\delta_{\ell}} .
\end{align*}
Summing both sides of this inequality for $k=0, \ldots,K$, and dividing  by $K$, we have 
\begin{align*}
 \E\left [\frac{1}{K} \sum_{k=0}^K\|d(x_k)\|^2\right] &\leq \frac{2}{\alpha\delta_{\ell} K}\left(f(x_0) - \E{\big[f(x_{K+1})\big]}\right) 
 + \frac{2\kfive}{\alpha \delta_{\ell}K}\sum_{k=0}^K q^{k}\nonumber \\
 &  \quad  + \frac{2\ksei}{\alpha \delta_{\ell}K}\sum_{k=0}^K \mu_k+\frac{\alpha L\kdue\delta_u^2}{\delta_{\ell}} \\
 &\leq \frac{2}{\alpha\delta_{\ell} K}(f(x_0) - f^* )
 + \frac{2\kfive}{\alpha \delta_{\ell}K}\sum_{k=0}^K q^{k}
 \nonumber \\
 &  \quad  + \frac{2\ksei}{\alpha \delta_{\ell}K}\sum_{k=0}^K \mu_k+ \frac{\alpha L\kdue\delta_u^2}{\delta_{\ell}} .
\end{align*}
Since  $\sum_{k=0}^\infty  q^{k}$ and $\sum_{k=0}^\infty \mu_{k}$ are summable, the claim follows.
\end{proof}

\section{Related Work}\label{sec:offosqp}
The extension of methods with random models for the unconstrained setting to the  setting of deterministic  equality and inequality constrained problems is a recent area of research which is drawing much interest, see \cite{BellaGratMoriToin25, curtis_siopt, CurtRobiZhou24, curtis_2026, krejic2025ipas,krejic2025box, krejic2025nolin, mahoney_siopt}.
Focusing on papers \cite{curtis_siopt, curtis_2026, krejic2025ipas, mahoney_siopt} for deterministic equality constrained optimization, we sketch their main features.

The work \cite{krejic2025ipas} introduces 
an Inexact Projection with Additional Sampling gradient method (\ipas) for weighted-sum minimization with linear equality constraints; we refer to the end of \S \ref{Sec.method} for the description of the control imposed on  $\|r(y_k)\|$.

Despite inexact projections in \ipas\ and in \names being controlled differently, the sequence $\{e(x_k)\}$ of infeasibility measures behaves similarly (cf. our Lemma  \ref{feasibility} and \cite[Lemma 4.2]{krejic2025ipas}).

\ipas\ employs stochastic estimates of the objective function. Specifically, functions and gradients are approximated by sampling and the step-size is determined by a non-monotone line-search rule over an approximate objective function. Additional sampling is used to decide whether a trial point should be accepted, as well as to decide if the sample size needs to be increased. Although this provides an  adaptive batch size strategy, the additional sampling combined with the line search yields additional costs and a more elaborated algorithm with respect to the  scheme proposed in this paper. In principle, \names also allows a variable sample size scheme, but it lacks a precise rule for its guidance.

Papers \cite{curtis_siopt, curtis_2026, mahoney_siopt}
present objective function-free procedures in the class of  Sequential Quadratic Programming (SQP) algorithms.
The procedures in \cite{curtis_siopt,   mahoney_siopt} employ stochastic gradients of the objective function  and prescribed approximations of the Hessian of the objective function and/or a Lagrangian function; the step-size selection is adaptive and is based on estimated Lipschitz constants. In particular, in \cite{curtis_siopt} the search direction  results from the use of a merit function with $\ell_1$-norm penalty function; it is computed solving a quadratic optimization problem based on a local quadratic minimizer of the objective function and a local affine model of the constraint.
The choice of the step-size is inspired by  a line search strategy and is based on a rule using Lipschitz constant estimates.  
The hyper-parameters of the algorithm are: a sequence of Lipschitz constant estimates of the objective function, a sequence of Lipschitz constant estimates of the constraint function, and a sequence to control the step-size. Assuming unbiased stochastic gradients and condition (\ref{var_bound1}), the analysis shows that the generated sequence achieves stationarity and feasibility in expectation. For the linearly constrained case considered here, the algorithms in \cite{curtis_siopt, curtis_2026, mahoney_siopt} generate feasible iterates.

In \cite{mahoney_siopt} the search direction is decomposed into a normal step and a tangential step. 
Exact projections are supposed to be computable and consequently the normal step has a closed form. On the other hand, the tangential step solves a trust-region problem which employs a basis for the null space of the Jacobian of the constraints. The adaptive trust-region radius is computed using estimates of the Lipschitz constants of both the objective function and the constraint functions.
The hyper-parameters of the algorithm are: a sequence of Lipschitz constant estimates of the objective function, a sequence of Lipschitz constant estimates of the constraint function, and two sequences of positive scalars to control the trust-region radius. Assuming that the stochastic gradient is unbiased and condition (\ref{var_bound1}), the theoretical analysis shows that KKT residuals converge to zero almost surely.

Finally, the recent work \cite{curtis_2026} extends stochastic momentum methods for unconstrained optimization to the stochastic SQP setting and proposes two algorithms: a projected stochastic heavy-ball SQP method and a projected stochastic Adam SQP method. These methods require exact projections, use projected stochastic gradient estimates in the momentum terms, and involve two predefined step-size-related sequences. Assuming unbiased stochastic projected gradients, condition (\ref{var_bound1}), and boundedness of $\{\|\nabla f(x_k)\|\}$, their convergence behavior is shown to be analogous to that of the corresponding unconstrained methods. In contrast, \names\ allows inexact projections.

\section{Experimental Configuration}\label{Sec.NumericalResults} 
All runs were performed in MATLAB R2025a on a workstation equipped with an Intel Ultra 9 processor, 128 GB DDR5 RAM, dual 2 TB SSDs, and an NVIDIA RTX PRO 4000 Blackwell GPU (24 GB), without using GPU acceleration.
First, we describe the test problems and our numerical implementation. Second, we show the performance of \names Algorithm. Finally, we compare our procedure with two algorithms recently proposed.
\subsection{Test problems}
We applied \names in the solution of six problems from the literature. The complete description is given below. 

\paragraph{Equality-constrained logistic regression problems}
{We used \textsc{Mushrooms}, \textsc{Mnist}, and \textsc{Diabetes} datasets from  \textsc{libsvm}\footnote{We used the files \texttt{Mnist.mat} and \texttt{diabetes.txt} from the \textsc{Libsvm} Library. }  \cite{chang2011libsvm}. 
For \textsc{Mushrooms}, we mapped labels ${1,2}$ to ${+1,-1}$ respectively, and used the feature matrix as returned.
For the multi-class dataset \textsc{Mnist}, we restricted the dataset   to a binary one with digits ${0,8}$ and mapped $0$ to $+1$ and $8$  to $-1$; all features were scaled to $[0,1]$.
For \textsc{Diabetes}, we remapped labels ${0,1}$ to ${-1,+1}$ respectively, and scaled the features to $[-1,1]$.
Given  $\{(z_i,y_i)\}_{i=1}^N$ with $z_i\in\mathbb{R}^n$ and $y_i\in\{-1,1\}$, we define the objective function in \eqref{problem}   \footnote{In our runs we used the stable formulation of the logistic regression.}
\begin{equation}
\label{eq:logistic}
f(x) \;=\; \frac{1}{N}\sum_{i=1}^{N} \log\!\big(1+\mathrm{e}^{-y_i z_i^\top x}\big).
\end{equation}
The constraint matrix $A \in \mathbb{R}^{m \times n}$ and the vector $b \in \mathbb{R}^{m}$ were generated using a fixed random seed (\texttt{rng(0)}); as for $n$, $m$ and  $N$, see Table~\ref{tab:probSize}. We used the initial iterate $x_0= A^{T} (A A^{T})^{-1} b$ for the runs with exact projection and $x_0=0$ for runs with inexact projection.}

\paragraph{Equality-constrained \textsc{Cutest} problems}
{We considered the problems {\textsc{Huestis}, \textsc{Dtoc1l}, and \textsc{Hs50}} subject to linear constraints. Following \cite{krejic2025ipas}, let $\widetilde f(x):\mathbb{R}^n\to\mathbb{R}$ denote the objective function in the \textsc{cute}st collection. We define a finite-sum objective function for (\ref{problem}) as
\begin{equation}\label{e.modf}
    f(x) = \widetilde f(x) + \sum_{i=1}^{N} \xi_i^2 \, \|x\|_2^2, 
\end{equation}
where each  $\xi_i$  is an independent random sample drawn from a Gaussian distribution, i.e., $\xi_i \sim \mathcal{N}(0, \sigma^2)$, with $\sigma=0.1$ and  fixed random seed (\texttt{rng(0)}). The constraint matrix $A \in \mathbb{R}^{m \times n}$  and the initial feasible point $x_0$ are given in the collection. The values of $n$, $m$ and  $N$  are given in  Table~\ref{tab:probSize}.}

\begin{table}[t]
\centering
\footnotesize
\caption{\small{Test problems.}}
\label{tab:probSize}
\renewcommand{\arraystretch}{1.3}
\begin{tabular*}{\textwidth}{@{\extracolsep{\fill}} l c c c c c c}
\hline

& \textsc{Mnist}
& \textsc{Mushrooms} 
& \textsc{Diabetes} 
& \textsc{Dtoc1l} 
& \textsc{Huestis} 
& \textsc{Hs50} \\ 
\hline

$N$ 
& $11774$ 
& $8124$ 
& $768$ 
& $10000$ 
& $10000$ 
& $10000$ \\[0.5mm]

$n$ 
& $780$ 
& $112$ 
& $8$ 
& $58$ 
& $10$ 
& $5$ \\[0.5mm]

$m$ 
& $\lfloor 0.5 n \rfloor$ 
& $\lfloor 0.5 n \rfloor$ 
& $\lfloor 0.5 n \rfloor$ 
& $36$ 
& $2$ 
& $3$ \\[0.5mm]

\hline
\end{tabular*}
\end{table}

\subsection{Algorithmic Implementation}\label{SubSec_conf}
For the computation of the stochastic gradients, the partition $\{\mathcal{N}_k^i\}_{i=1}^r$ in (\ref{partition}) was  kept fixed along the iterations. 
Unless explicitly stated, the mini-batch size   was set to 64 for \textsc{Diabetes} and to 256 for the other datasets. We  built ten independent random partitions, each of which created by using an independent random seed\footnote{If $N$ is not a multiple of the number of partitions in \eqref{gki}, one mini-batch in \eqref{partition} has smaller cardinality.}. For each of such  partitions, the mini-batches were selected randomly and independently along the iteration.

Regarding the selection of the step-sizes in \name, we set $\delta_{\ell}=10^{-3}$, $\delta_{u}=10^{2}$. We tested three strategies denoted \textbf{S1}, \textbf{S2} and \textbf{S3}.
They differ in the choice of the scaling parameter $\alpha_{k}$ and/or of the parameter $\delta_{k}$. 
In  \textbf{S1}, we considered a fixed sequence $\{\alpha_{k}\}$,  $\alpha_{k}=\alpha$ for all $k\ge0$,
while the scalar $\delta_k$ was computed with a Barzilai-Borwein approach due to the potential of such step-sizes in a stochastic environment, see e.g., \cite{BMY_JCAM, Dai2005}. Our requirement that $\Delta_{k}$ is fully determined at $x_{k }$ motivated the use of  the retarded Barzilai-Borwein steps \cite{Dai2005, Friedlander1999}, defined as
$$\delta_k = \displaystyle \left| \frac{d_{k-\widetilde q}^T d_{k-\widetilde q}}{d_{k-\widetilde q}^T z_{k-\widetilde q}} \right|,$$
where $\widetilde q = \min\{q, k-1\}$ with $q \ge 1$, and $d_{t-1} = x_t - x_{t-1}$, $z_{t-1} = g_t - g_{t-1}$ with $t \ge 1$; see \cite[Eq.~(2.3)]{Dai2005}. Setting $q = 1$ yields  
\begin{equation}\label{BBdelay_s}
    \delta_k = \left| \frac{d_{k-1}^T d_{k-1}}{d_{k-1}^T z_{k-1}} \right|. 
\end{equation} 
Following guidelines on stochastic Barzilai-Borwein step-lengths, it is advisable to compute $z_t$ using  stochastic gradients with the same mini-batch; since this would require an additional gradient evaluation at each iteration, we updated $\delta_k$ using (\ref{BBdelay_s}) every $20$ iterations. 
The strategy \textbf{S1} is adaptive due to the form of  $\delta_k$,  and  the sequence $\{\Delta_k\}$ is not supposed to be decreasing.

In \textbf{S2}, we considered $\delta_k$ as in (\ref{BBdelay_s}), and  the diminishing  parameter $\alpha_{k+1}$ of the form
\begin{eqnarray}
\alpha_{k+1} & = &\frac{a}{a+k} c_{k}(\gamma_0, \gamma_1),  \label{e.dimAlpha}\\
c_{k}(\gamma_0, \gamma_1) &= & \gamma_1 + 0.5(\gamma_0 - \gamma_1)\Big(1 + \cos\Big(\frac{k\pi}{k_{\max}}\Big)\Big),  
\label{e.cosdecay}
\end{eqnarray}
with $\gamma_0, \gamma_1 > 0$ and $k_{\max}$ representing the maximum number of iterations allowed,  \cite{loshchilov2016sgdr, matlab_cosine}. The cosine-decay step rule (\ref{e.cosdecay}) implies
$c_0(\gamma_0,\gamma_1)=\gamma_0$, $c_{k_{\max}}(\gamma_0,\gamma_1)=\gamma_1$.
In fact, the strategy \textbf{S2} is still adaptive because of  $\delta_k$, but now $\Delta_k$ is driven in intervals of diminishing length, i.e., from  $[\gamma_0 \delta_{\ell}, \gamma_0\delta_u]$  to   $[ \frac{a\gamma_1}{a+k_{\max}}\delta_{\ell}, \,\frac{a\gamma_1}{a+k_{\max}} \delta_u]$ as $k$ reaches $k_{\max}$.

In the strategy  \textbf{S3} we fixed $\delta_{k} = 1$ for all $k\ge0$, and  set $\alpha_{k+1}$ as in (\ref{e.dimAlpha})-(\ref{e.cosdecay}). Taking into account that $\pi_{\delta}(1)=1$ for reasonable values of $\delta_{\ell}$, $\delta_u$, we have $\Delta_{k}=\alpha_k$ and decreasing step-sizes for increasing values of $k$.

Table~\ref{tab:SSS} summarizes the three strategies  \textbf{S1}-\textbf{S3}. In our experiments we set $a = 1000$ in  (\ref{e.dimAlpha}) and $\gamma_1 = 10^{-5}$ in (\ref{e.cosdecay}). The parameter $\alpha$ in \textbf{S1}  and  the parameter $\gamma_0$ in \textbf{S2} and \textbf{S3} were varied;
we tuned $\alpha$ and $\gamma_0$   exploring a set of prefixed values as shown in \S \ref{Sec.Eff}. The initial step-size is set to $\Delta_0=\alpha \delta_{\ell}$ in \textbf{S1} and $\Delta_0=\gamma_0 \delta_{\ell}$ in \textbf{S2} and \textbf{S3}.

The exact projection $\pi_S$ in \eqref{exactprojection} was computed by the Cholesky factorization of $AA^T$. Exploiting this factorization, we precomputed the  matrix  $D = A^T(AA^T)^{-1}$ once prior to the iterative process. Afterwards, the projection computation required products of the matrix $D$ times a vector.

 The evaluation of the inexact projection $\widetilde \pi_S$ in (\ref{inexactprojection}) was performed  applying the
Conjugate Gradient Method (Matlab built-in function \texttt{pcg}) to \eqref{res} with stopping criterion  \eqref{resb}. We set $\eta_k = \eta>0$ for all $k\ge 0$, and $\mu_k = \mu_0 \rho^k$, for all $k\ge 0$, $\mu_0>0$ and $\rho\in (0,1)$; the ablation study on  $\eta$,  $\mu_0$ and $\rho$ is given in \S \ref{Sec.Sensi}. 
The stopping criteria of \texttt{pcg} are the relative residual condition $\|r(y_k)\|/ \|Ay_k - b\| < \tau$ and the maximum iteration limit (\textsc{iter}). We set $\textsc{iter}= m$ and let  
$$ \tau = \max\, \left \{10^{-10}, \min\, \left \{\frac{\eta_k \|Ax_k-b\| +\mu_k}{\|Ay_k - b\|}, 10^{-3}\right \}\, \right \}.$$
The thresholds above avoid too small and too large values of  $\|r(y_k)\|$. Moreover,  to save computations, we skipped projection when  $\|Ay_k - b\| \leq 10^{-12}$.

If not explicitly stated, each run consists of $k_{\max}=10^4$ iterations.  For each configuration of the parameters involved, the evaluation metrics were averaged over ten independent runs to construct the mean performance values. In the generated plots, the $y$-axis displays these metrics on a base-10 logarithmic scale.

\begin{table}[t]
\centering
\caption{\small{Strategies for computing the step-length $\Delta_{k+1}$.}}
\label{tab:SSS}
\renewcommand{\arraystretch}{1.3}
\footnotesize
\begin{tabular*}{\textwidth}{@{\extracolsep{\fill}} l c c}
\hline
 Strategy & \textbf{$\alpha_{k+1}$} & \textbf{$\delta_k$} \\ 
\hline
$\mathbf{S1}$ 
& $\alpha > 0$ 
& $\delta_k$ in \eqref{BBdelay_s} \\[.5mm]

$\mathbf{S2}$   
& $\alpha_{k+1}$ \mbox{ in } (\ref{e.dimAlpha})-(\ref{e.cosdecay})  
& $\delta_k$ in \eqref{BBdelay_s} \\[0.5mm]

$\mathbf{S3}$  
&  $\alpha_{k+1}$ \mbox{ in } (\ref{e.dimAlpha})--(\ref{e.cosdecay})  
& $\delta_k = 1$  \\[2mm]
\hline
\end{tabular*}
\end{table}


\begin{figure}[H]
\centering
\begin{subfigure}[t]{\textwidth}
\centering
\includegraphics[scale=0.2]{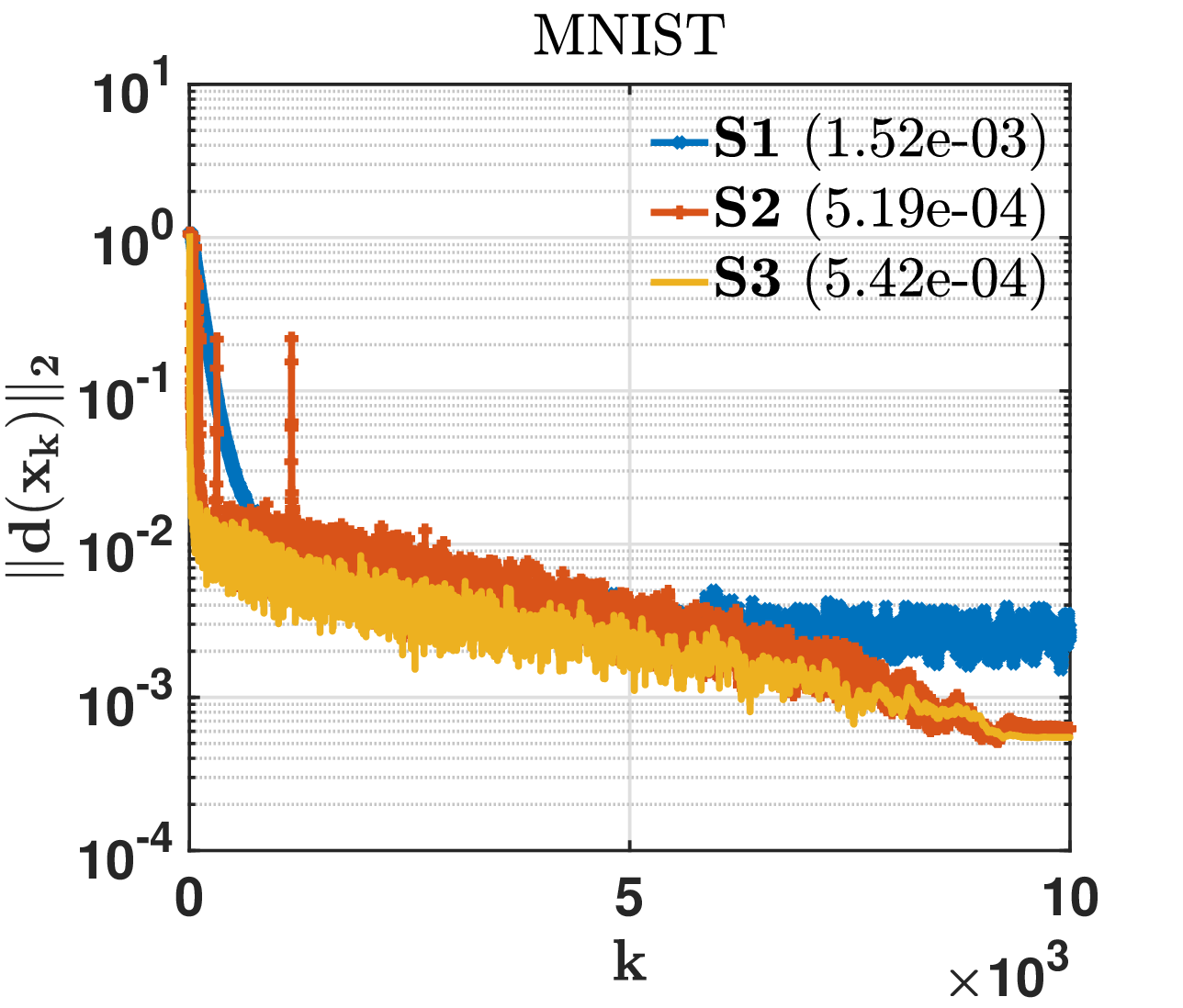}\hspace{-3mm} 
\includegraphics[scale=0.2]{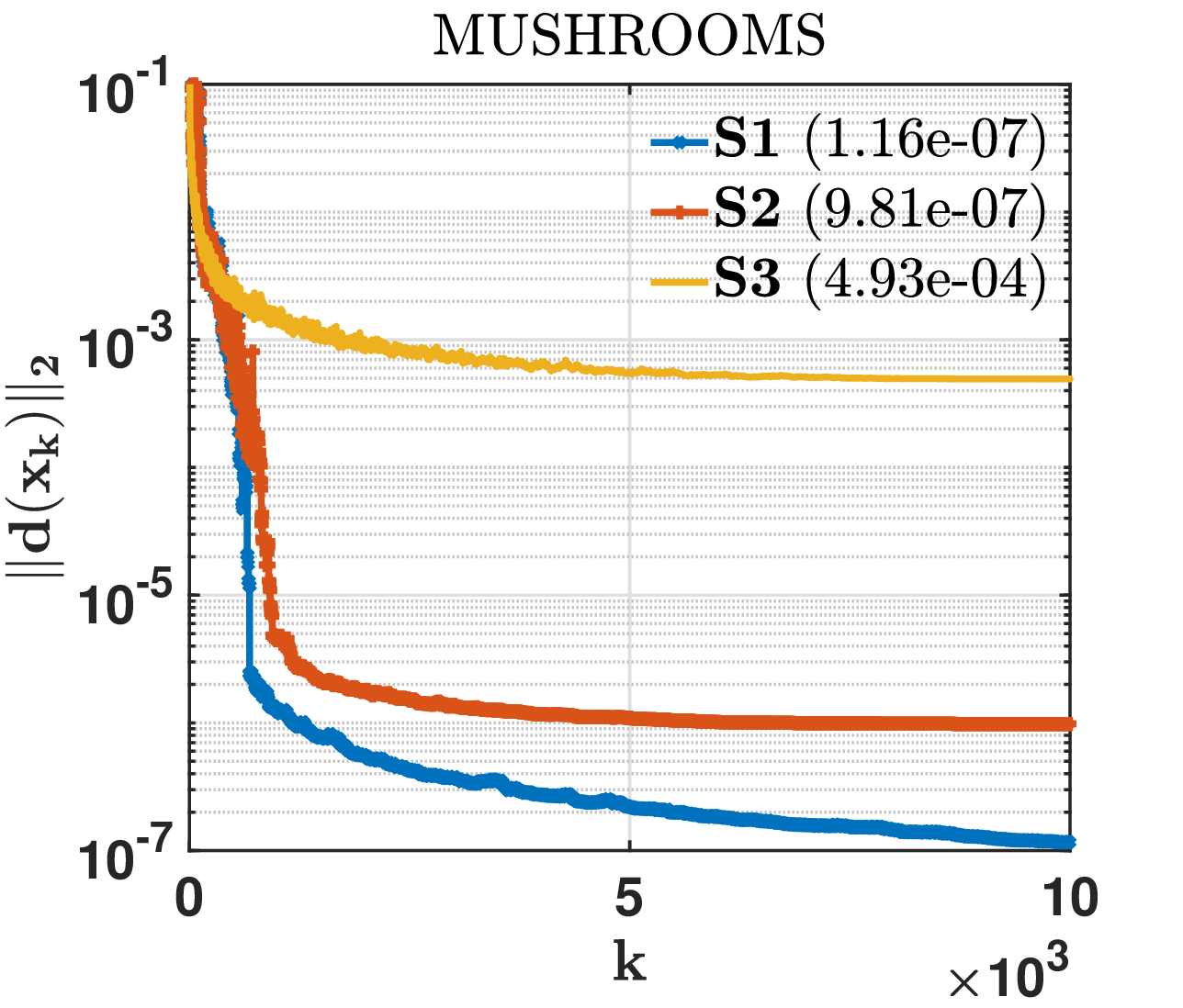}\hspace{-3mm} 
\includegraphics[scale=0.2]{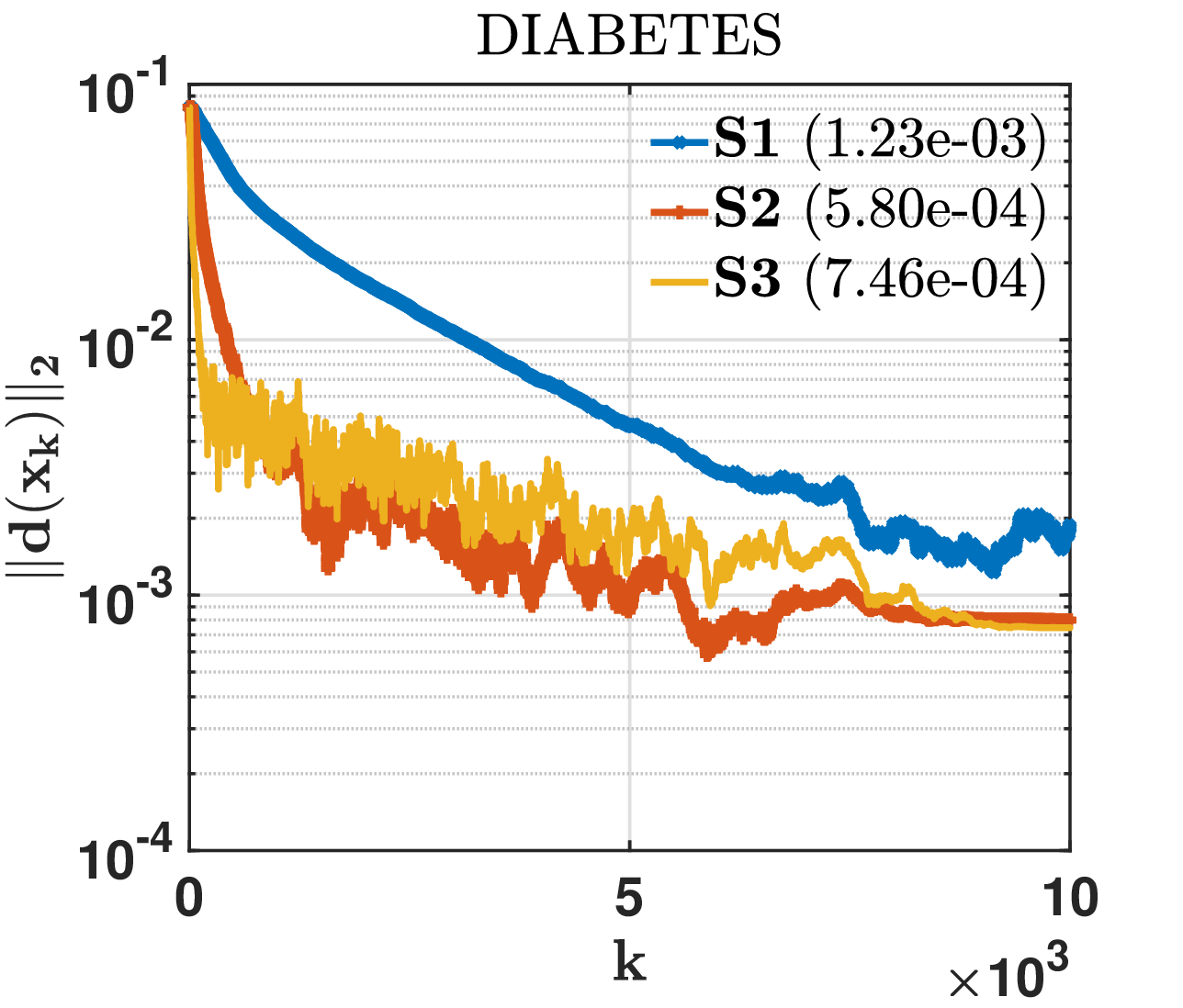}
\end{subfigure}  
\medskip
\begin{subfigure}[t]{\textwidth}
\centering
\includegraphics[scale=0.2]{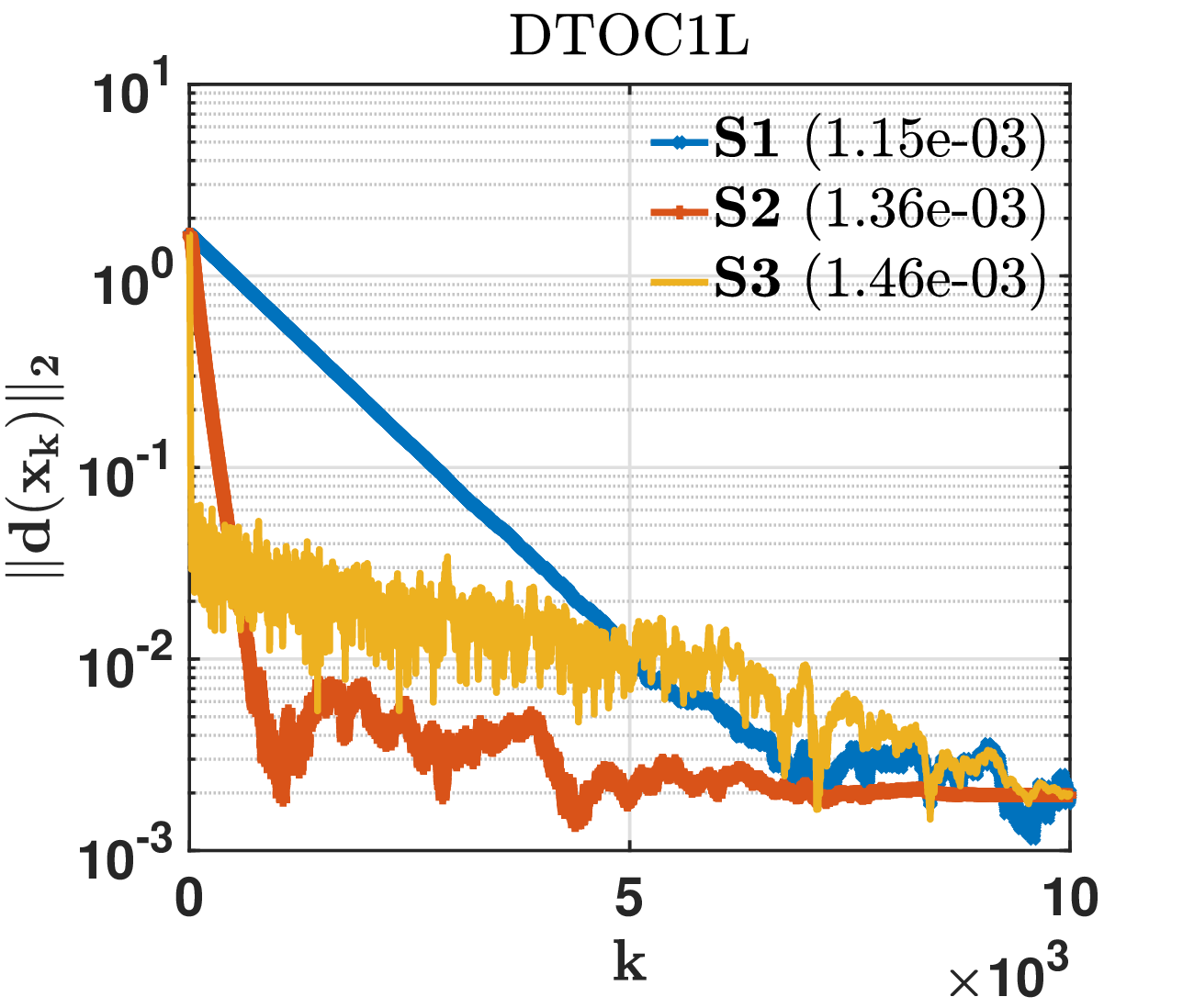}\hspace{-3mm} 
\includegraphics[scale=0.2]{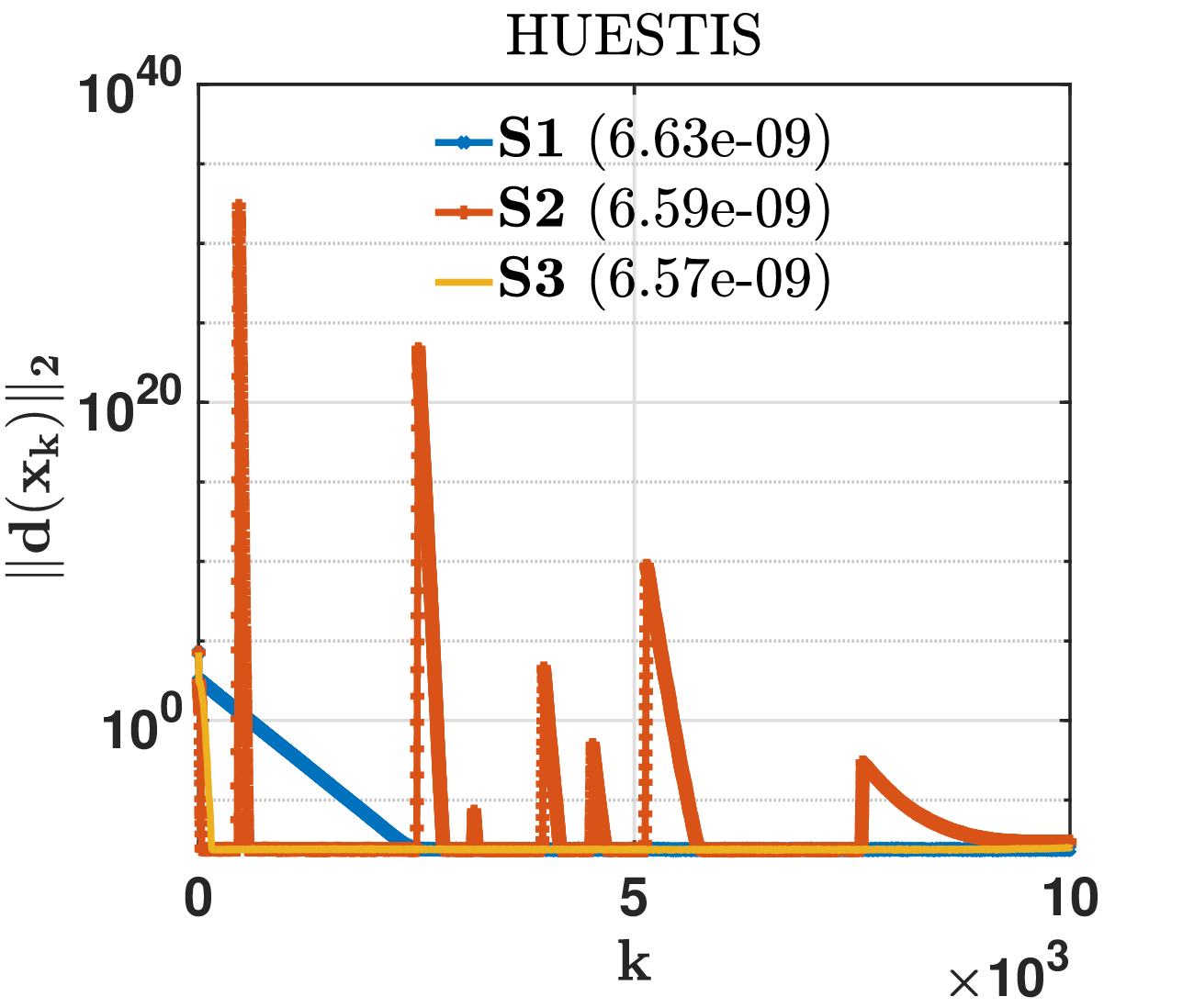}\hspace{-3mm}
\includegraphics[scale=0.2]{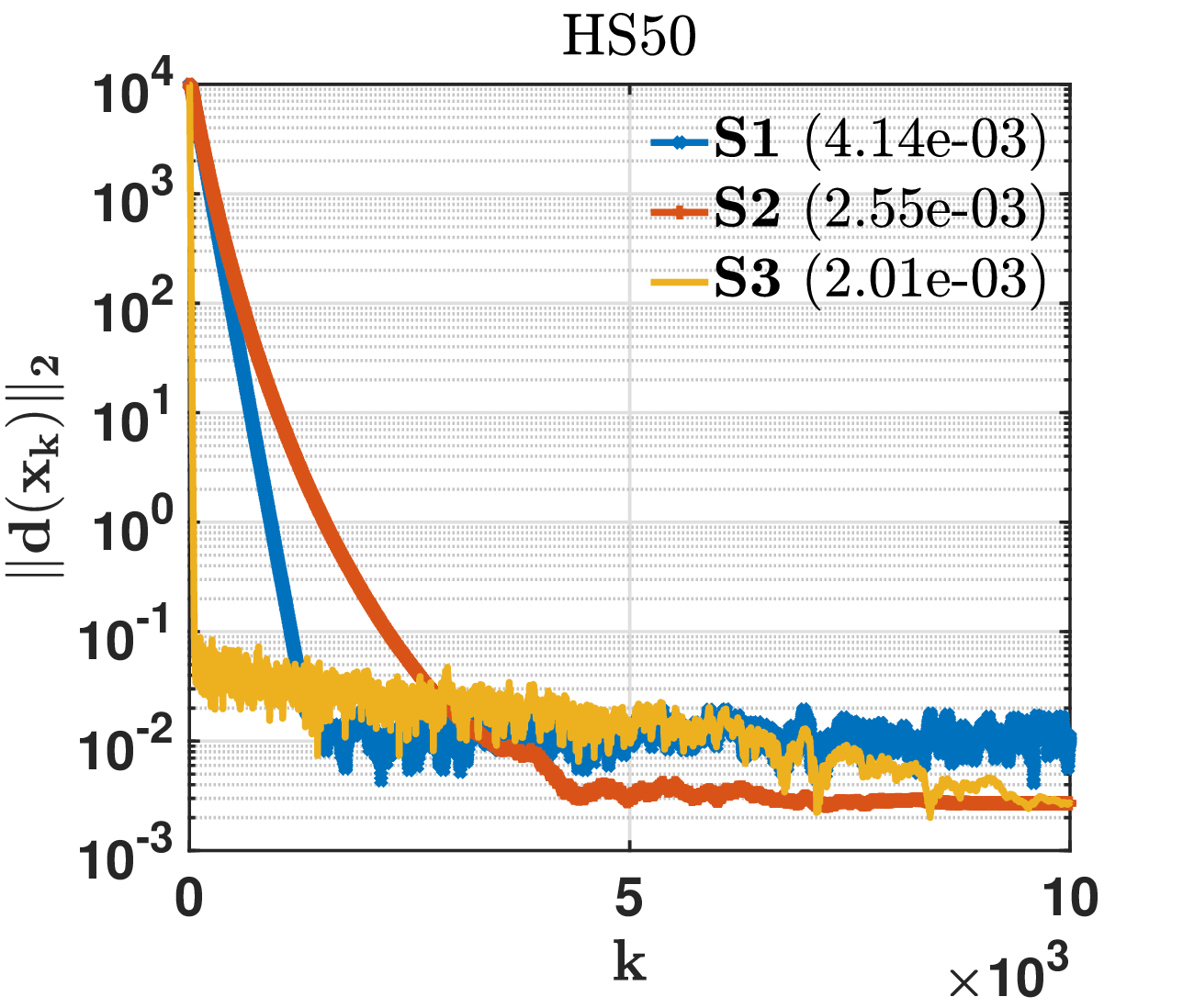}
\end{subfigure}
\caption{\small   The average values of $\|d(x_k)\|$ obtained by \name \ using the exact projection  and step -size strategies \textbf{S1}-\textbf{S3} with tuned $\alpha$ and $\gamma_0$.}
  \label{fig:exact}
\end{figure}
\section{Numerical Results}\label{Sec.Num}
In \S \ref{Sec.Eff} we analyze the performance of our algorithm implemented with the step-size strategies \textbf{S1}-\textbf{S3} and run for varying hyper-parameters, including the batch-size. In \S \ref{Sec:inexactp} we analyze the performance of \names when the projection is calculated inexactly. In \S~\ref{Sec.PerfCompSQP} and \S~\ref{Sec.PerfCompIPAS}, we compare the performance of \names with the Projected Stochastic Adam SQP method (\sqp) proposed in \cite{curtis_2026} and with the \ipas\ algorithm proposed in \cite{krejic2025ipas}, respectively.

\subsection{Performance of PSG\_LECO and Ablation Study}\label{Sec.Eff}
\subsubsection{Step-Size Selection}

\names Algorithm is an adaptation of the Stochastic Gradient method to the linearly constrained case and inherits its dependence from the step-size selection from both a theoretical and numerical point of view. Hence, our numerical analysis of \names method was first conducted using the exact projection $\pi_S$ in (\ref{exactprojection}) and focusing on the strategy for choosing the step-size and the parameters therein: the parameter $\alpha$ in strategy \textbf{S1}, and the parameter $\gamma_0$ in strategies  \textbf{S2} and \textbf{S3}. In principle, in the strategies  \textbf{S1} and \textbf{S2} the frequency of the update (\ref{BBdelay_s})  of $\delta_k$    may play a role, but, as briefly shown at the end of this section, we experimentally verified that it marginally affects the performance; for this reason, the discussion below refers to the update (\ref{BBdelay_s})  of $\delta_k$  every 20 iterations, with the batch-size as specified in \S\ref{SubSec_conf}.

We conducted our experiments by varying the hyper-parameters $\alpha$ and $\gamma_0$ within the range $\mathcal{T}_1=\{10^{-3}, 10^{-2}, 10^{-1}, 1, 10\}$. The remaining hyper-parameters in  \name \ were specified in \S\ref{SubSec_conf}. 
In Table~\ref{tab:OptParSens}, we display the smallest  average value of $\|d(x_k)\|$ obtained for varying strategies and parameters. The performance of \name \ depends on the step-size rule and below we discuss the results obtained.

For values of $\alpha$ in the set $\{10^{-2}, 10^{-1}, 1\}$, the strategy \textbf{S1} achieves low values for optimality measure and the value of $\alpha$ influences the performance  only in the solution of the problem \textsc{Mushrooms}.  In  general, using the smallest value $\alpha=10^{-3}$ is not  effective.
Referring to a failure in the case when the smallest achieved value of $\|d(x_k)\|$  is higher than $10^{-1}$,  we observe failures for the value  $\alpha=10$ in problems  \textsc{Dtoc1l}, \textsc{Huestis} and \textsc{HS50}. The strategy \textbf{S2} shows failures in the solution of the problems \textsc{Dtoc1l},  \textsc{Huestis} and \textsc{HS50}  when $\gamma_0=10^{-3}, 10$. Excluding failures,  out of the complete runs, the accuracy achieved by  \names with  \textbf{S2} is higher or comparable to that of \names with \textbf{S1} except for problem \textsc{Mushrooms}. Further, employing \textbf{S2} outperforms  \textbf{S3} in terms of optimality. The strategy \textbf{S3} appears to be the less robust among the three tested. Failures occur for values   $\gamma_0$ in the set $\{10^{-2}, 10^{-1}, 1, 10\}$ in the solution of \textsc{Dtoc1l}, \textsc{Huestis} and \textsc{HS50}.

The results obtained suggest that \names Algorithm coupled with the strategy \textbf{S2} attains the highest accuracy, and that the adaptive choice for $\delta_k$ positively affects the performance of the method. In order to provide more insight into \names Algorithm, in Figure \ref{fig:exact}
we plot the average values of $\|d(x_k)\|$ versus the iterations obtained with the  strategies \textbf{S1}-\textbf{S3}; the results displayed  correspond to the tuned parameters which gave 
comparatively better performance in terms of optimality in the experimental study presented above and are reported in Table \ref{tab:optimalParam}.

\begin{table}[t]
\centering
\caption{\small The minimum average value of $\|d(x_k)\|$ obtained by \name\ over all the datasets, using the exact projection.}
\label{tab:OptParSens}
\setlength{\tabcolsep}{2pt} 
\begin{tabularx}{\textwidth}{|c|*{6}{>{\centering\arraybackslash}X|}}
\hline
& \multicolumn{3}{c|}{\textbf{MNIST}} 
& \multicolumn{3}{c|}{\textbf{MUSHROOMS}} \\ \hline

$\alpha, \gamma_0$ & \textbf{S1} & \textbf{S2} & \textbf{S3} & \textbf{S1} & \textbf{S2} & \textbf{S3} \\ \hline

$10^{-3}$
& 7.326426e-03 & 2.083257e-01 & 2.788223e-01 
& 1.013773e-02 & 9.433663e-02 & 2.964358e-01 \\ \hline

$10^{-2}$  
& 1.518623e-03 & 3.853655e-03 & 4.195787e-02 
& 1.099648e-03 & 6.204044e-03 & 6.876800e-02 \\ \hline

$10^{-1}$
& 2.034988e-03 & 7.589910e-04 & 6.799109e-03 
& 1.036599e-04 & 6.369529e-04 & 1.472482e-02 \\ \hline

1      
& 3.859859e-03 & 5.192025e-04 & 1.444848e-03 
& 7.020857e-06 & 3.794461e-05 & 2.981289e-03 \\ \hline

10     
& 4.092225e-03 & 5.560996e-04 & 5.421119e-04 
& 1.158238e-07 & 9.810904e-07 & 4.932504e-04 \\ \hline

\end{tabularx}

\begin{tabularx}{\textwidth}{|c|*{6}{>{\centering\arraybackslash}X|}}
\hline
& \multicolumn{3}{c|}{\textbf{DIABETES}} 
& \multicolumn{3}{c|}{\textbf{DTOC1L}} \\ \hline

$\alpha,\gamma_0$ 
& \textbf{S1} & \textbf{S2} & \textbf{S3}
& \textbf{S1} & \textbf{S2} & \textbf{S3} \\ \hline

$10^{-3}$ 
& 1.228244e-03 & 1.921677e-02 & 7.151869e-02 
& 1.154305e-03 & 2.907515e-01 & 1.458438e-03 \\ \hline

$10^{-2}$ 
& 1.638480e-03 & 5.799938e-04 & 3.518885e-02 
& 2.698684e-03 & 1.361814e-03 & 1.806219e-03 \\ \hline

$10^{-1}$ 
& 4.407748e-03 & 8.521068e-04 & 2.688412e-03 
& 5.849365e-03 & 1.747942e-03 & 1.608953e+00 \\ \hline

$1$       
& 2.432084e-02 & 1.463002e-03 & 7.456083e-04 
& 3.145822e-02 & 1.817416e-03 & 1.315883e+00 \\ \hline

$10$      
& 7.998332e-02 & 2.392200e-03 & 9.455264e-04 
& 1.273853e+00 & 1.273853e+00 & 1.614671e+00 \\ \hline

\end{tabularx} 

\begin{tabularx}{\textwidth}{|c|*{6}{>{\centering\arraybackslash}X|}}
\hline
& \multicolumn{3}{c|}{\textbf{HUESTIS}} 
& \multicolumn{3}{c|}{\textbf{HS50}} \\ \hline

$\alpha,\gamma_0$ 
& \textbf{S1} & \textbf{S2} & \textbf{S3}
& \textbf{S1} & \textbf{S2} & \textbf{S3} \\ \hline

$10^{-3}$ 
& 1.389308e-02 & 5.864945e+01 & 6.663500e-09 
& 4.170339e-01 & 1.755844e+03 & 2.008578e-03 \\ \hline

$10^{-2}$ 
& 6.625383e-09 & 1.024091e-05 & 6.571084e-09 
& 4.142960e-03 & 2.547671e-03 & 9.893982e+03 \\ \hline

$10^{-1}$ 
& 6.572476e-09 & 6.383179e-09 & 3.253096e+02 
& 8.063442e-03 & 2.399365e-03 & 9.711535e+03 \\ \hline

$1$       
& 6.833632e-09 & 6.590809e-09 & 2.686514e+02 
& 4.337829e-02 & 2.503668e-03 & 7.920287e+03 \\ \hline

$10$      
& 2.064706e+02 & 2.064706e+02 & 2.979305e+02 
& 9.914297e+03 & 1.078018e+04 & 1.672371e+05 \\ \hline

\end{tabularx}
\end{table}

Figure~\ref{fig:exact}  displays the  trend of $\|d(x_k)\|$ through the iterative procedure for the three strategies with selected parameter values in Table~\ref{tab:optimalParam}; the attained minimum value of $\|d(x_k)\|$ is reported in parentheses in the legend.  
We observe that Algorithm~\name\ coupled with the \textbf{S3} strategy generally exhibits good progress toward optimality during the initial phase of the execution and shows a behavior comparable to \textbf{S2} in the final stage, except on the \textsc{Mushrooms} dataset. On the other hand, the \textbf{S2} strategy consistently achieves better or comparable optimality measures against \textbf{S3} in the middle stage of the execution, except for \textsc{Huestis} where  some awkward peaks in $\|d(x_k)\|$ occurred, though recovered at the end of the process. Except for \textsc{Mushrooms}, the  performance of \textbf{S2} is overall superior to  the  performance of \textbf{S1} and this fact can be attributed to the incorporation of the diminishing $\alpha_{k+1}$ in \textbf{S2}. Summarizing the results obtained, a tuned implementation of Algorithm~\name\ combined with the \textbf{S2} strategy is effective on our test problems. 

We conclude this analysis, showing the impact of the BB step-length update period, denoted by  $C$, of $\delta_k$ in \eqref{BBdelay_s}. The results presented above correspond to $C=20$. We tested the values  $C\in \mathcal{T}_2 = \{1, 5, 20, 100\}$ in the strategy \textbf{S2} with  $\gamma_0$ as in Table~\ref{tab:optimalParam}. Figure~\ref{fig:exact_m} illustrates the behavior of the optimality measure versus the iteration count. For the sake of readability, we display one point every 20 iterations. The behavior of the optimality measure is only slightly influenced by the value of $C$, although $C=100$ appears to be less effective. On the other hand, the cost in terms of stochastic gradient evaluations depends on $C$ since the smaller $C$ the higher the cost required to evaluate  $z_{k-1}$ in (\ref{BBdelay_s}); thus $C=20$ represents a good tradeoff between efficiency and cost.

\begin{figure}[t]
\centering
\begin{subfigure}[t]{\textwidth}
\centering
\includegraphics[scale=0.19]{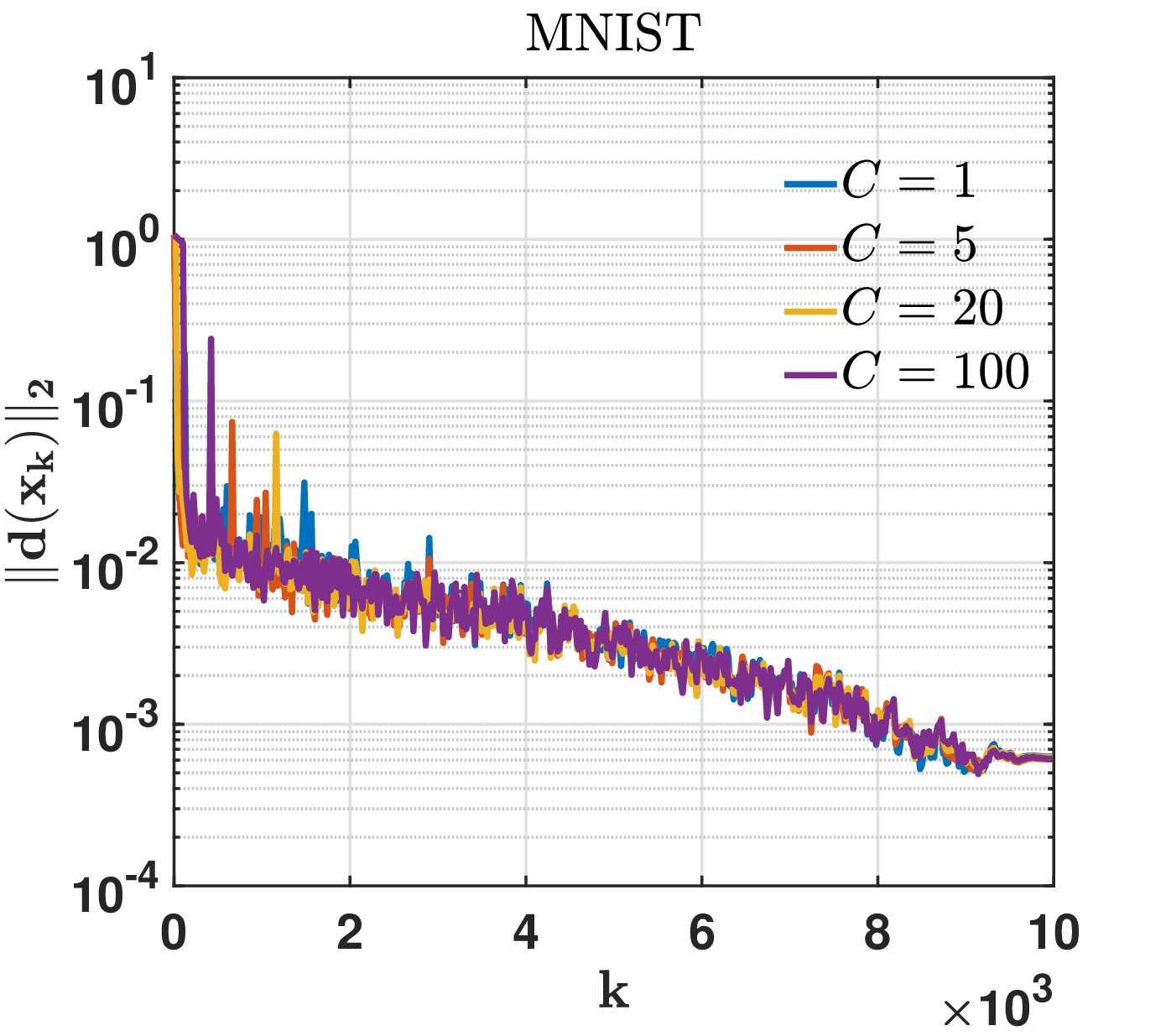}\hspace{-3mm}  
\includegraphics[scale=0.19]{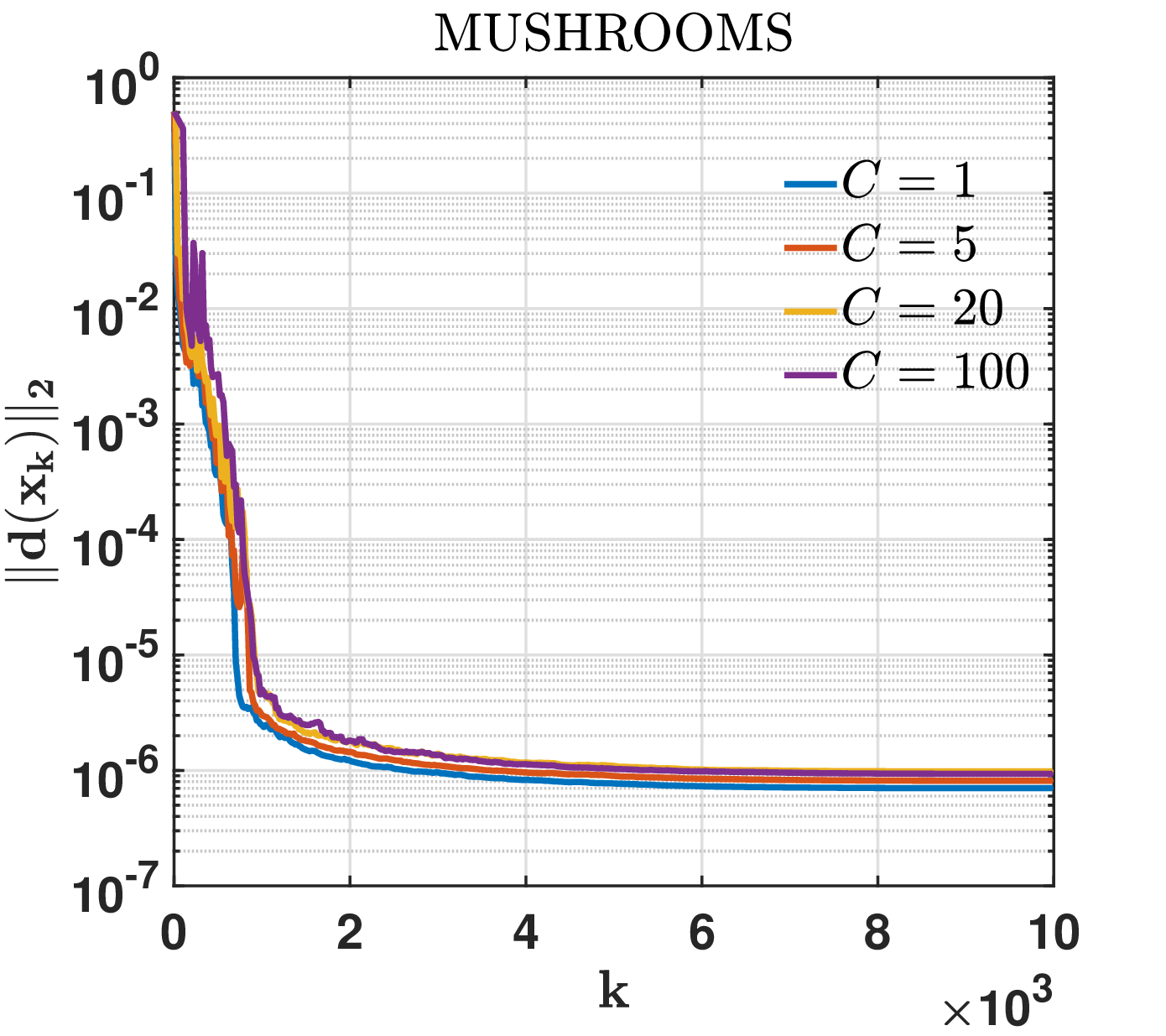}\hspace{-3mm}
\end{subfigure}
\caption{\small Average value of $\|d(x_k)\|$ vs. iterations obtained by \name\ with \textbf{S2} and exact projection for varying $C$.}
  \label{fig:exact_m}
\end{figure}


\begin{table}[t]
\centering
\caption{ \small{Selected parameter values for different step-size strategies across datasets.}}
\label{tab:optimalParam}
\renewcommand{\arraystretch}{1.3}
\footnotesize 
\begin{tabular*}{\textwidth}{@{\extracolsep{\fill}} l c c c c c c}
\hline
Strategy ($\alpha/\gamma_0$) 
& \textsc{Mnist}
& \textsc{Mushrooms} 
& \textsc{Diabetes} 
& \textsc{Dtoc1l} 
& \textsc{Huestis} 
& \textsc{Hs50} \\ 
\hline

$\mathbf{S1}\, (\alpha)$ 
& $10^{-2}$ 
& $10$ 
& $10^{-3}$ 
& $10^{-3}$ 
& $1$ 
& $10^{-2}$ \\[.5mm]

$\mathbf{S2}\, (\gamma_0)$ 
& $1$ 
& $10$ 
& $10^{-2}$ 
& $10^{-2}$ 
& $1$ 
& $10^{-2}$ \\[.5mm]

$\mathbf{S3}\, (\gamma_0)$ 
& $10$ 
& $10$ 
& $1$ 
& $10^{-3}$ 
& $10^{-2}$ 
& $10^{-3}$ \\[.5mm]
\hline
\end{tabular*}
\end{table}

\newpage
\subsubsection{Mini-Batch Size Selection}\label{Sec.Sensi}
The results presented above correspond to batch sizes, denoted as  $N_b$,  equal to 64 for \textsc{Diabetes} and to 256 for the remaining datasets, as stated in \S \ref{SubSec_conf}. Now, we investigate the influence of $N_b$ by testing values in the set $\mathcal{T}_3 = \{64,256, 0.2N\}$. Figure~\ref{fig:exact_Nk} shows the optimality measure versus iterations for \names with \textbf{S2} and $\gamma_0$ from Table~\ref{tab:optimalParam}; values are plotted every 20 iterations for readability. 

The initial behavior of the optimality measure is similar for all choices of $N_b$. The smallest optimality values obtained with $N_b=64$ and $N_b=256$ are comparable, except for \textsc{Mushrooms}, where $N_b=64$ is particularly effective. Choosing $N_b=0.2N$ provides some advantages, except on \textsc{Mushrooms}. Overall, with a tuned step-size selection, the choice of $N_b$ is not crucial for the effectiveness and efficiency of \name.

\begin{figure}[t]
\centering
\begin{subfigure}[t]{\textwidth}
\centering
\includegraphics[scale=0.145]{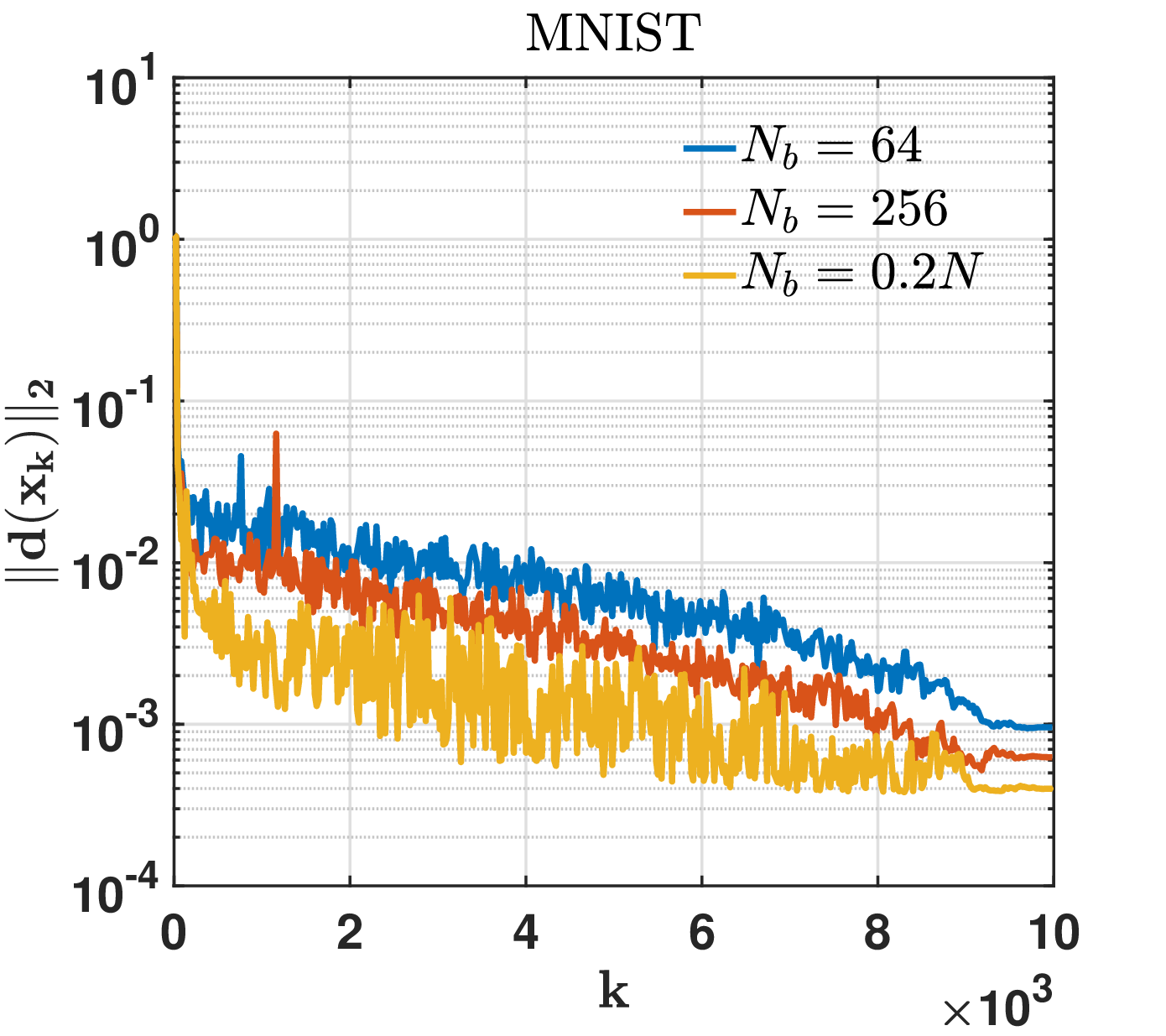}\hspace{-3mm} 
\includegraphics[scale=0.145]{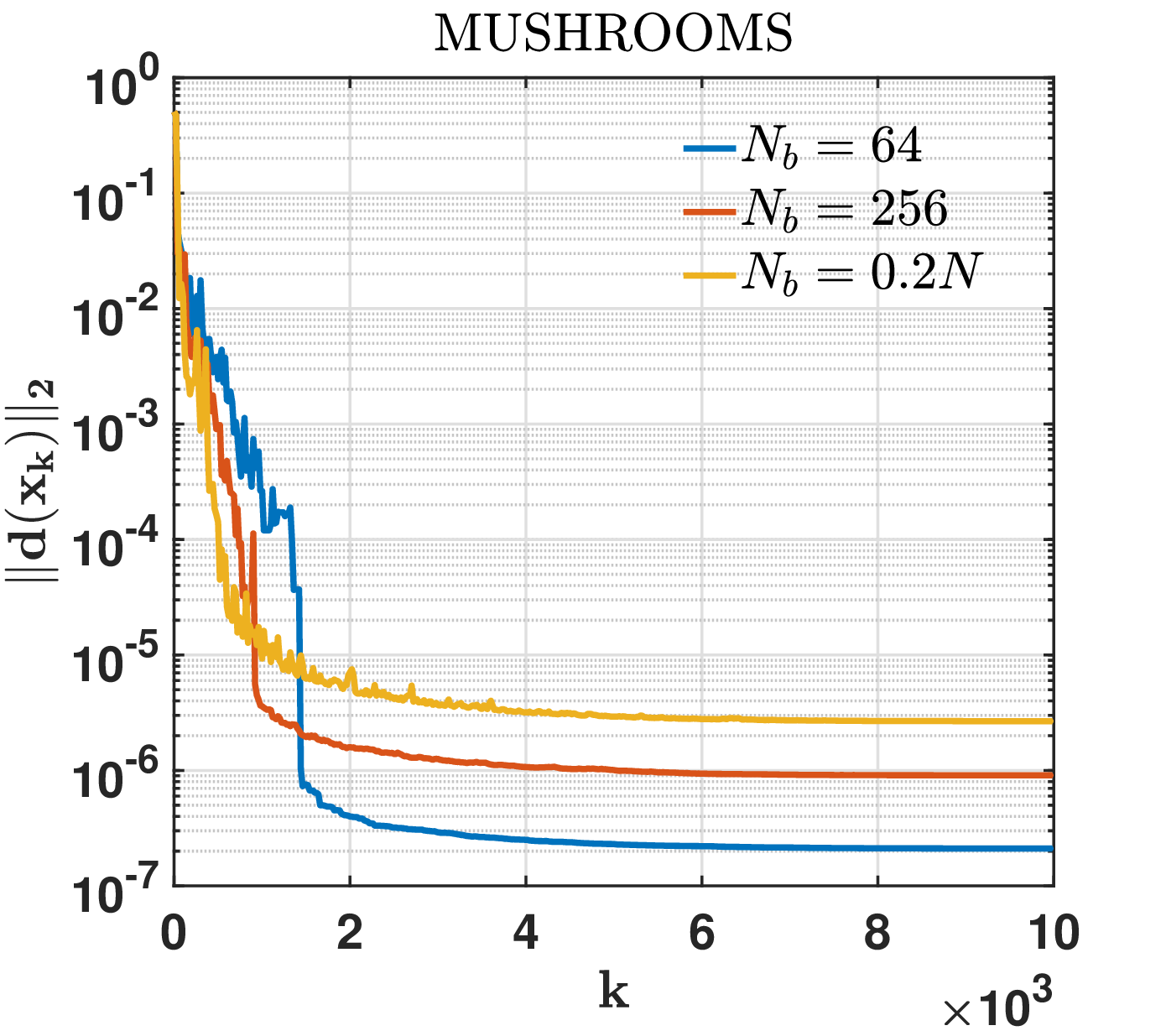}\hspace{-3mm}
\includegraphics[scale=0.145]{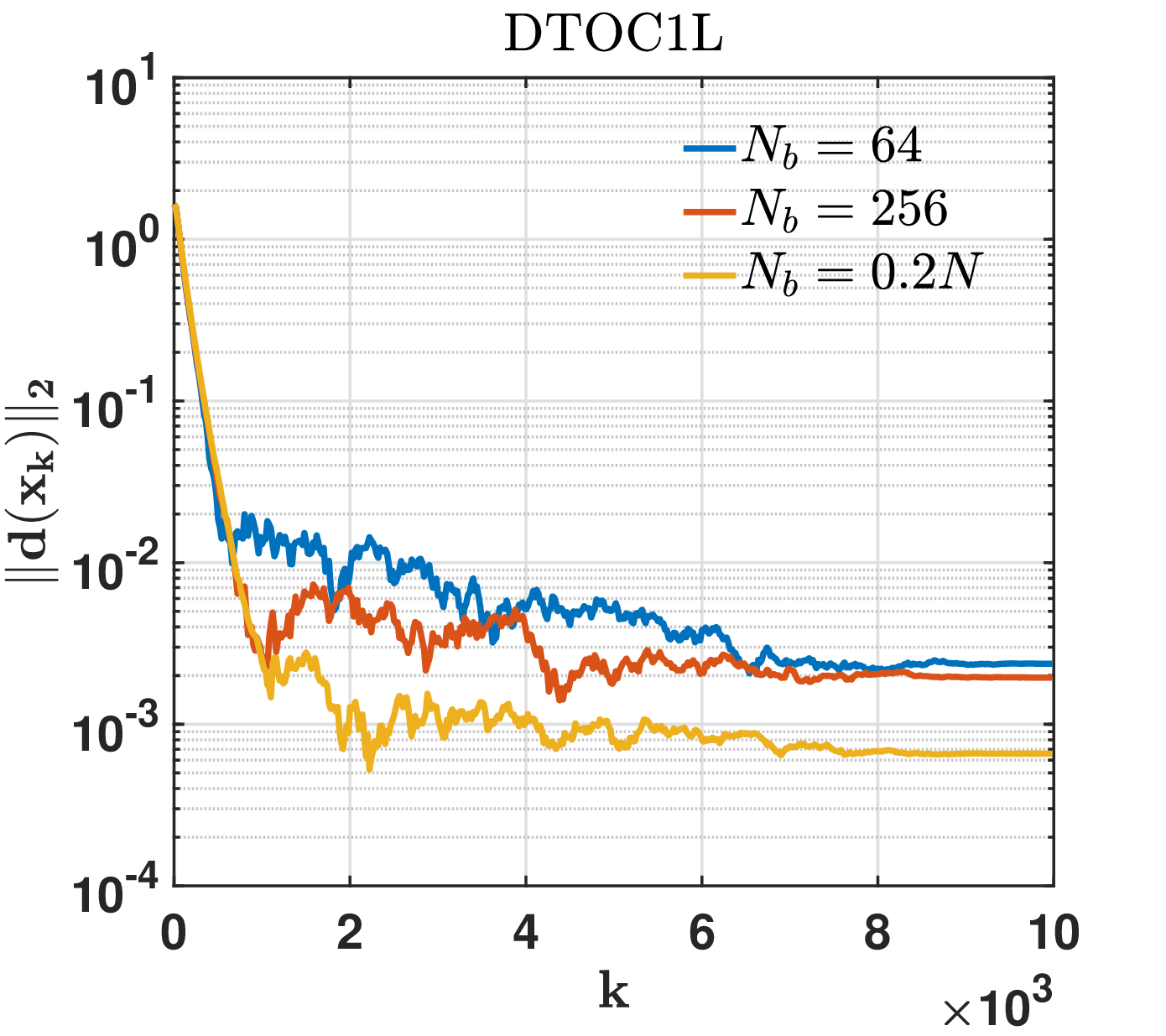}\hspace{-3mm}
\includegraphics[scale=0.145]{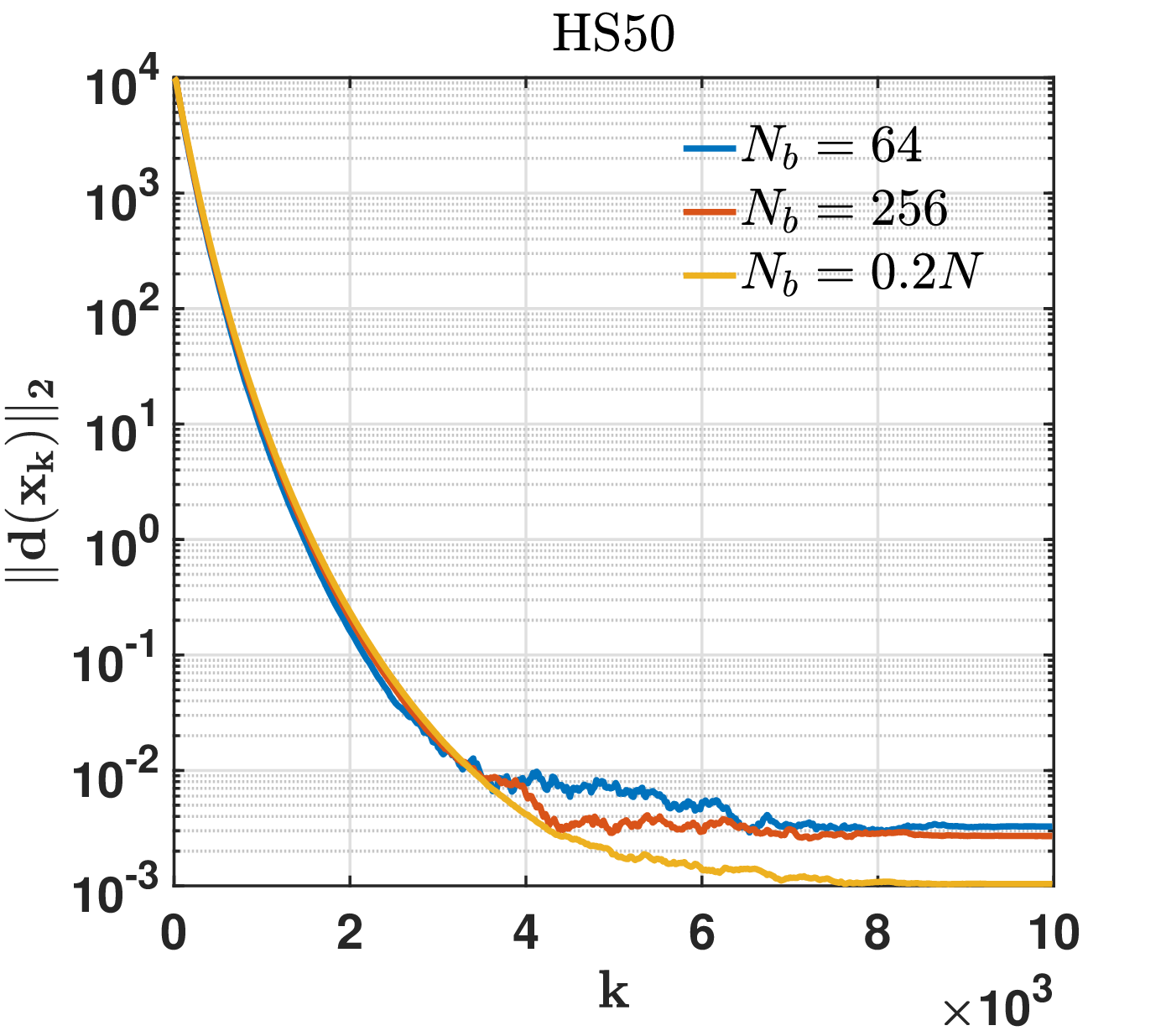}
\end{subfigure}  
\caption{\small Average value of $\|d(x_k)\|$  vs. iterations obtained by \name\ with strategy \textbf{S2} and exact projection for varying $N_b$.}
  \label{fig:exact_Nk}
\end{figure}

\subsection{Inexact Projection}\label{Sec:inexactp}{
 We now consider \names with the inexact projection $\widetilde{\pi}_S$ in (\ref{inexactprojection}). The evaluation of $\widetilde{\pi}_S$ is performed as discussed in \S \ref{SubSec_conf}. We tested \names Algorithm coupled with the strategy \textbf{S2} with $\gamma_0$ as in Table \ref{tab:optimalParam} over \textsc{Mnist} and \textsc{Mushroom} due to their relatively larger values of $m$. We let $\eta\in \mathcal{T}_4 = \{0.2, 0.5, 0.7, 0.9\}$, $\rho\in \mathcal{T}_5 = \{0.90, 0.95, 0.98, 0.99\}$, $\mu_0=0.1$, and studied the potential effect of allowing looser linear solver accuracy in the initial phase of the iterations.

Figure~\ref{fig:inexact_etamu} shows that $\|d(x_k)\|$ is generally insensitive to $\eta$ and $\rho$. Slightly higher values of  $\|d(x_k)\|$ are observed in the first phase of the iterative process for the large  values of $\rho$, but  $\mu_k = \mu_0 \rho^k$ is supposed to vanish eventually; thus changes in the values of $\rho$ as well as $\eta$ become negligible due to the specific form of the upper bound in (\ref{resb}).

Figure~\ref{fig:inexact_etamu2} illustrates the optimality measure versus execution time (in seconds), given the prefixed number of iterations $k_{\max}=10^4$. It refers to $\eta_k=\eta=0.5$, for all $k\ge 0$, and concerns 
the parameter setting  $\rho\in \mathcal{T}_5$ and $\mu_0\in \mathcal{T}_6 = \{10^{-1}, 1,10,100\}$. We note that  \name\ is quite  insensitive to the choice of $\mu_0$ in $\rho$, although large  values of $\mu_0$ and $\rho$ yield some gain in terms of timings.

Finally, in Figure~\ref{fig:inexact}, we show  both the optimality measure $\|d(x_k)\|$ and the infeasibility measure $\|Ax_k-b\|$ versus  the iterations for \names with strategies \textbf{S1}--\textbf{S3} and parameters as in Table \ref{tab:OptParSens}; the lowest achieved values are reported in the corresponding legends.
The results refer to $\eta_k =  \eta = 0.5$, for all $k\ge 0$, $\mu_0=10^{-1}$ and $\rho=0.95$. We note that the infeasibility decreases fast and that the decrease of optimality is not affected by the inexact projection. 

\begin{figure}[t]
\centering
\begin{subfigure}[t]{\textwidth}
\centering
\includegraphics[scale=0.155]{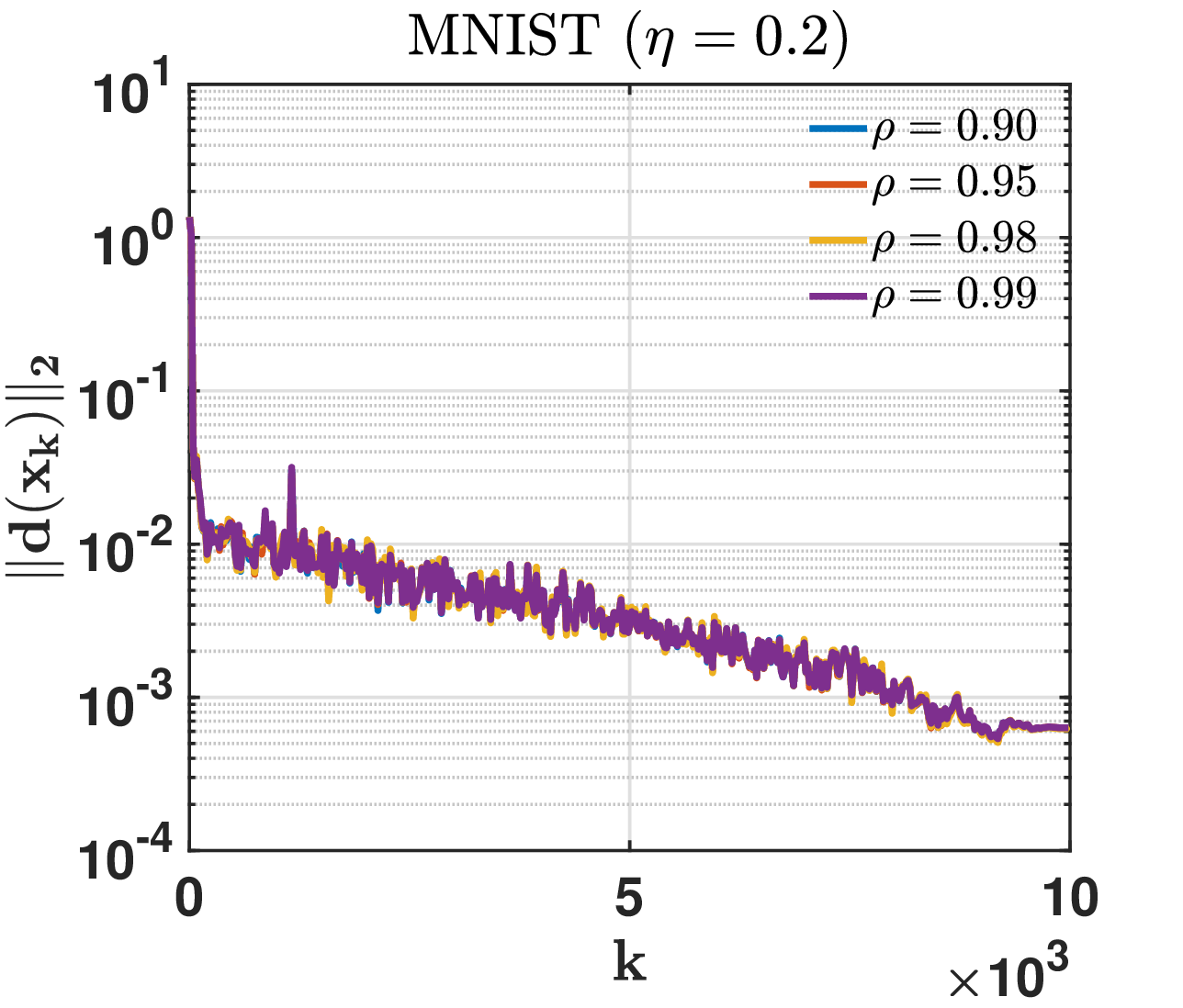}\hspace{-3mm} 
\includegraphics[scale=0.155]{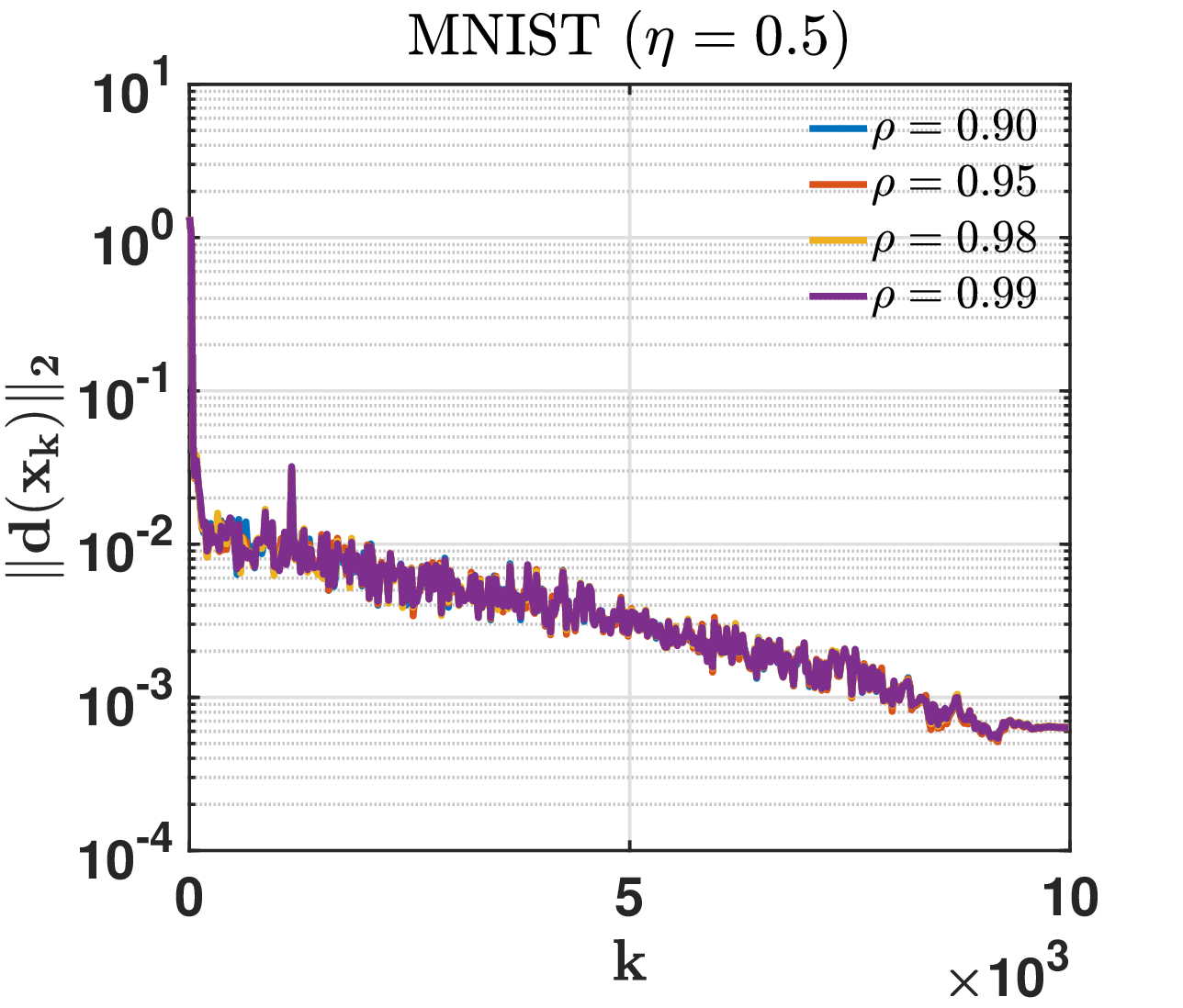}\hspace{-3mm}
\includegraphics[scale=0.155]{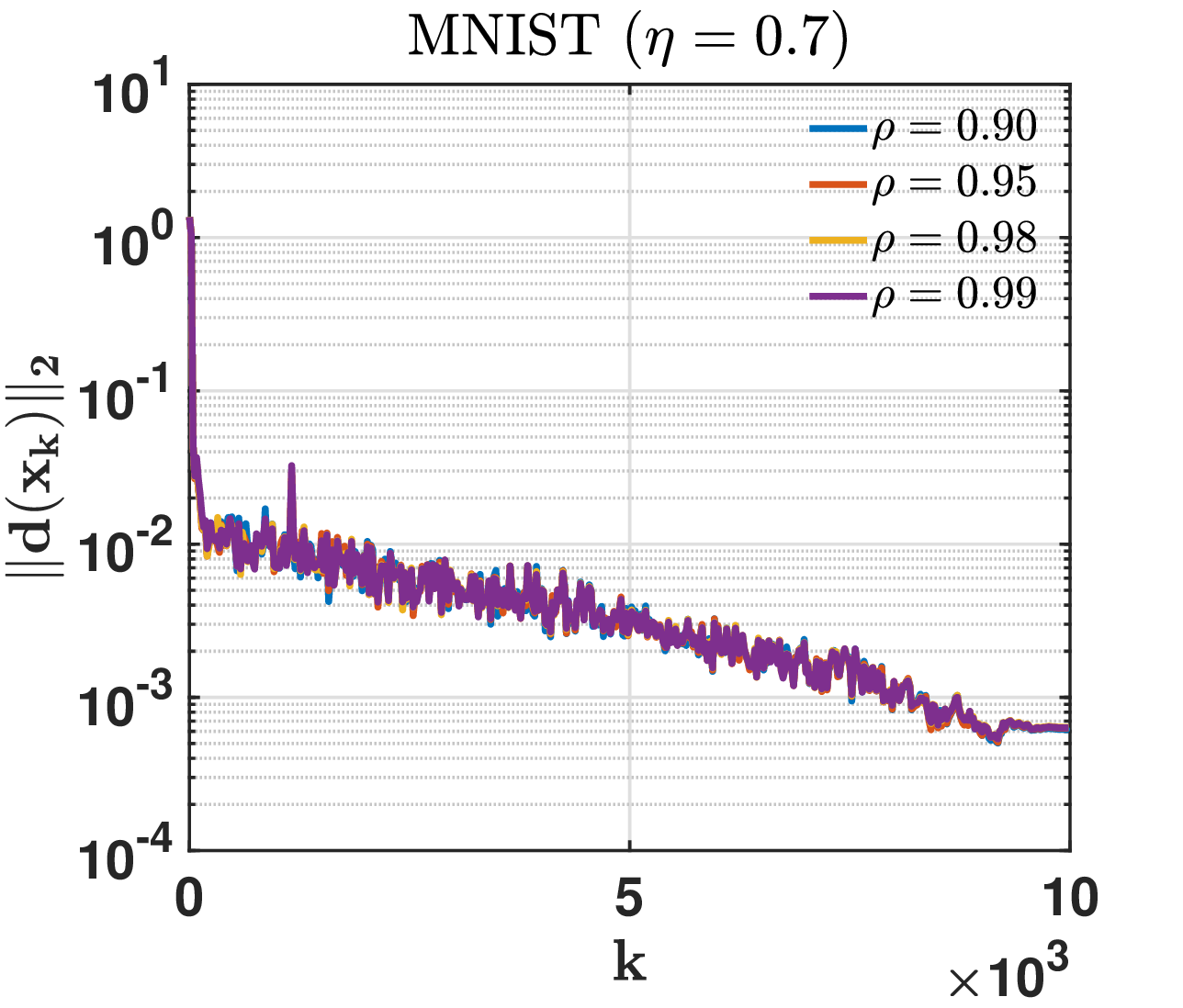}\hspace{-3mm}
\includegraphics[scale=0.155]{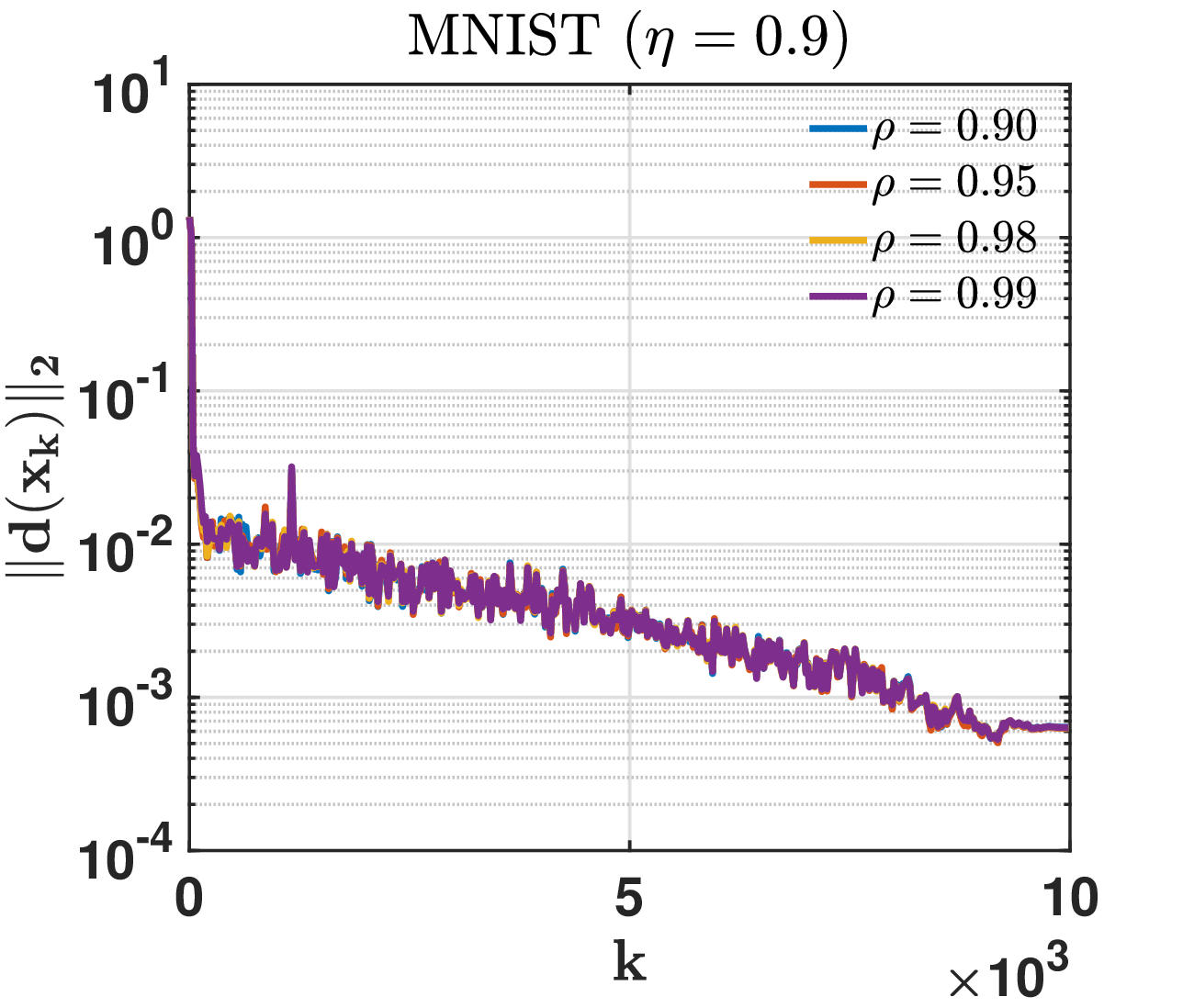}
\includegraphics[scale=0.155]{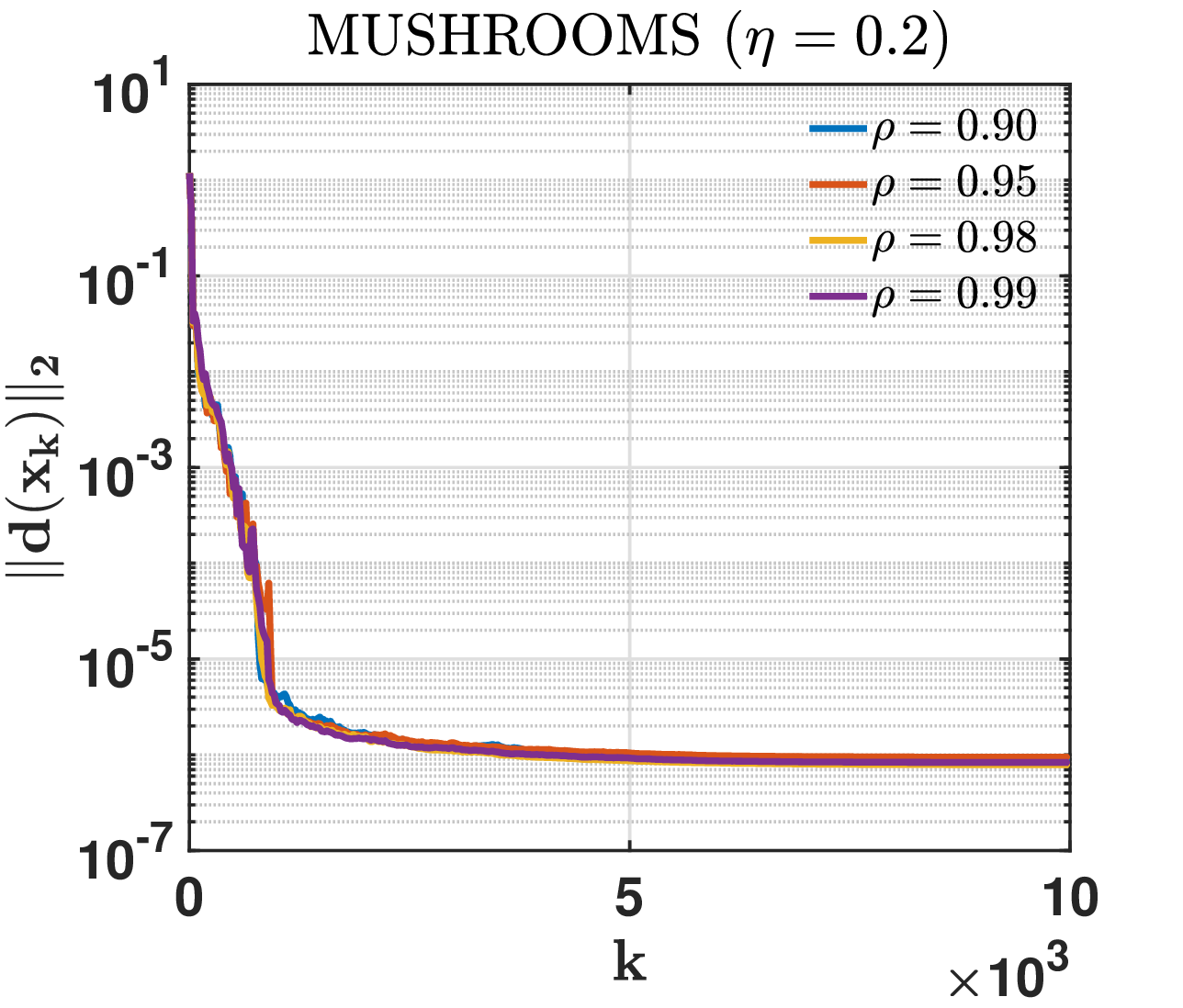}\hspace{-3mm} 
\includegraphics[scale=0.155]{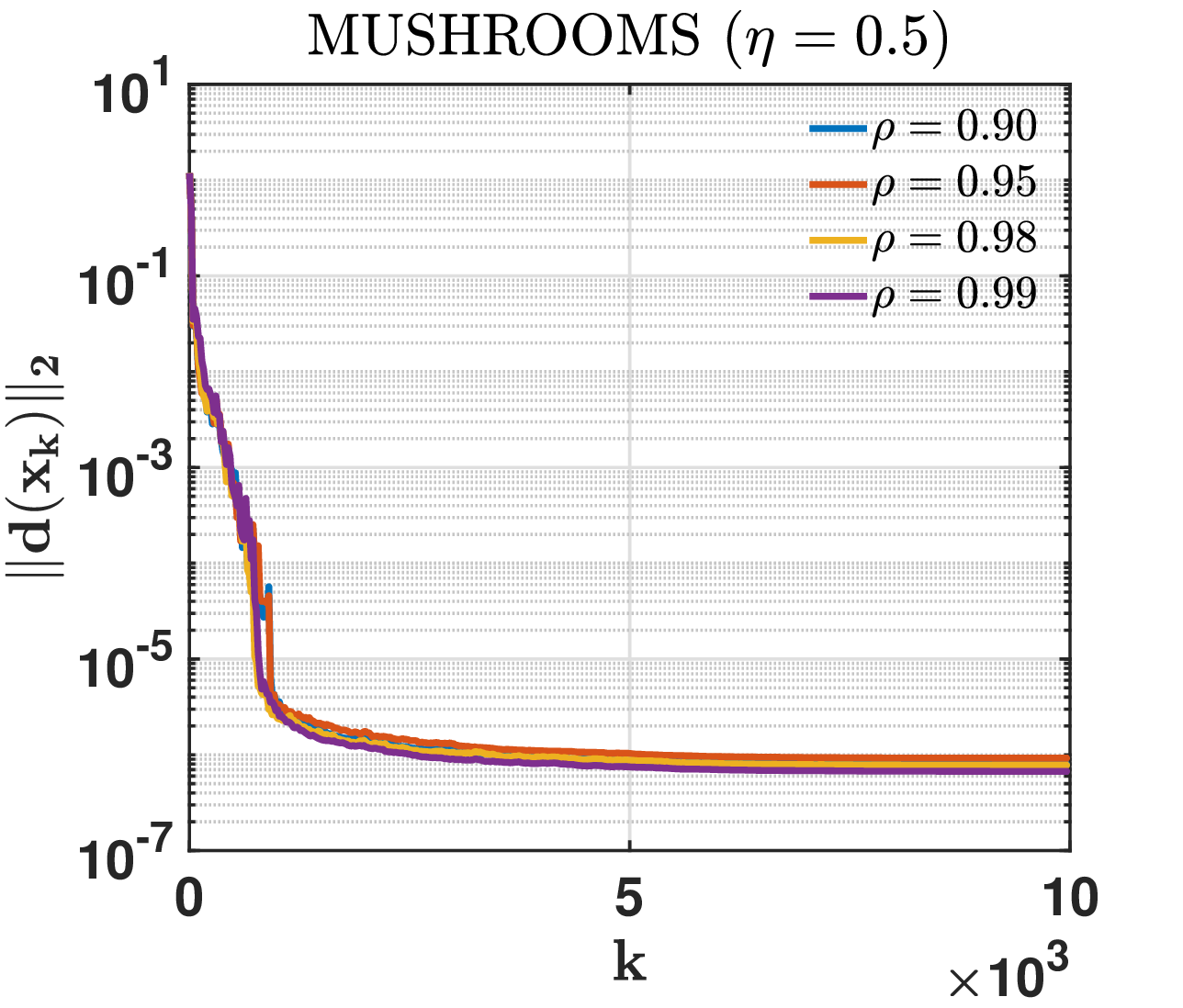}\hspace{-3mm}
\includegraphics[scale=0.155]{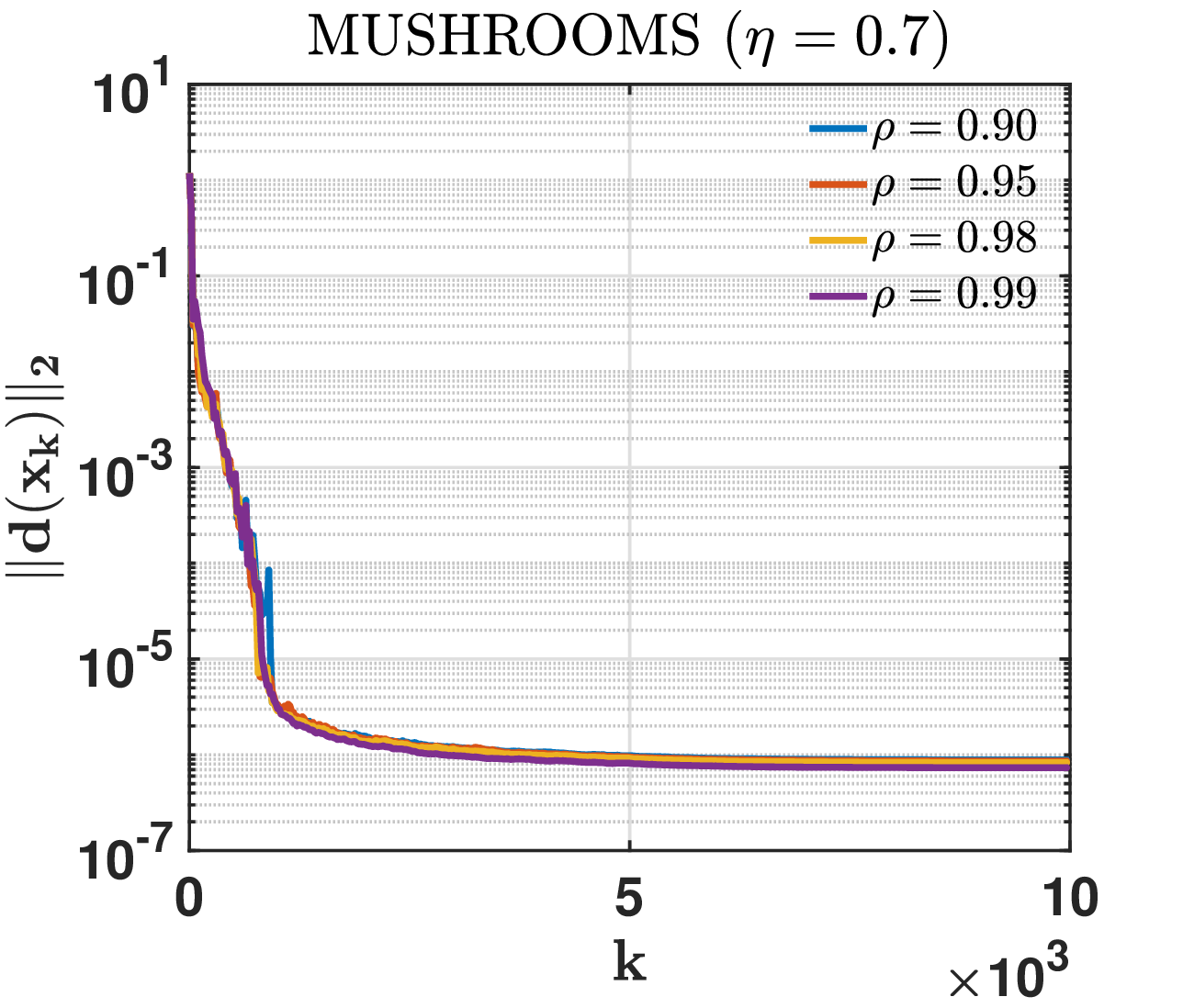}\hspace{-3mm}
\includegraphics[scale=0.155]{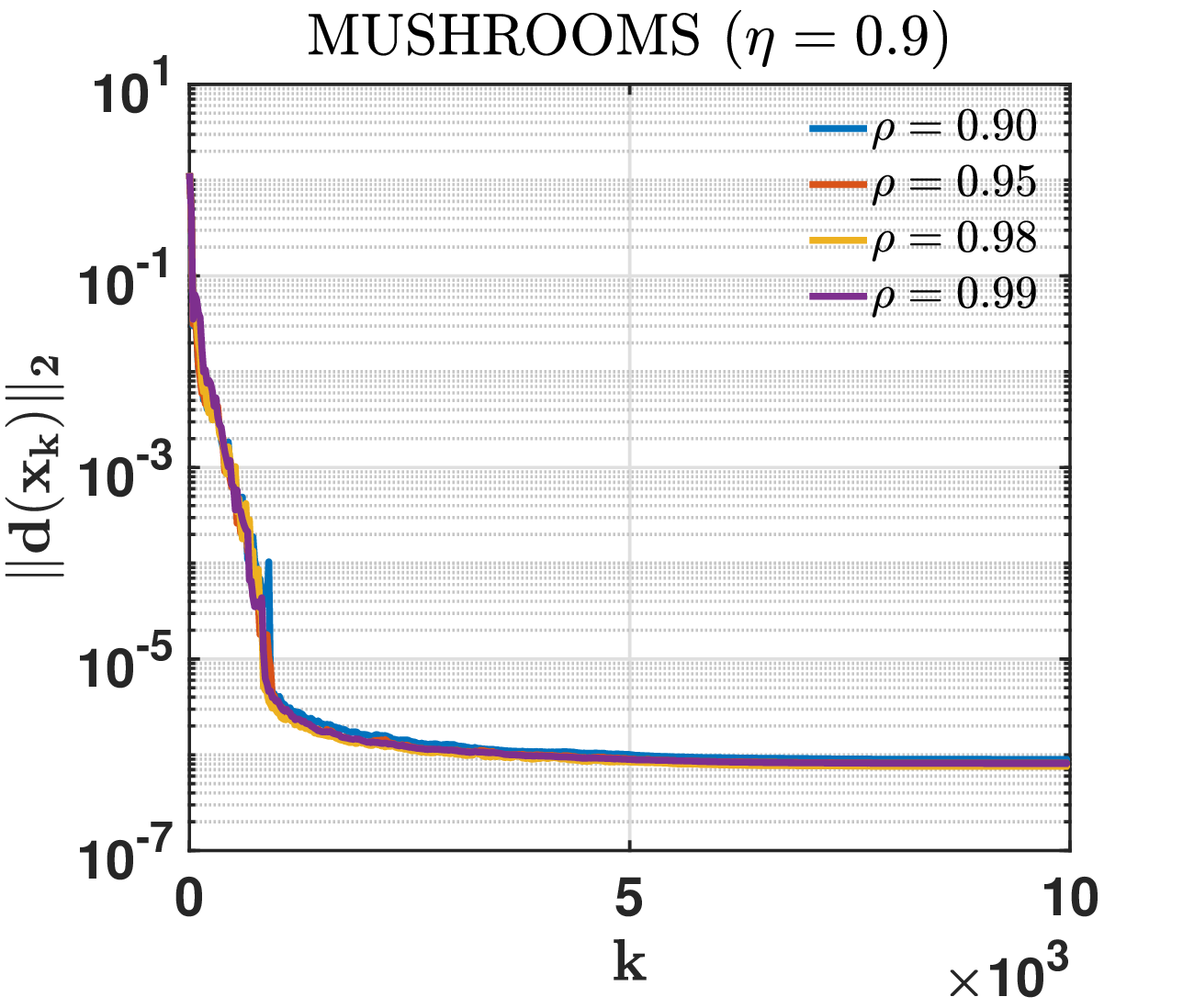}
\end{subfigure}  
\caption{\small Average values of $\|d(x_k)\|$ vs. iterations obtained by \name\  with strategy \textbf{S2}, inexact projection, $\mu_0=0.1$ and varying $\eta$, $\rho$.}
  \label{fig:inexact_etamu}
\end{figure}

\begin{figure}[t]
\centering
\begin{subfigure}[t]{\textwidth}
\centering
\includegraphics[scale=0.155]{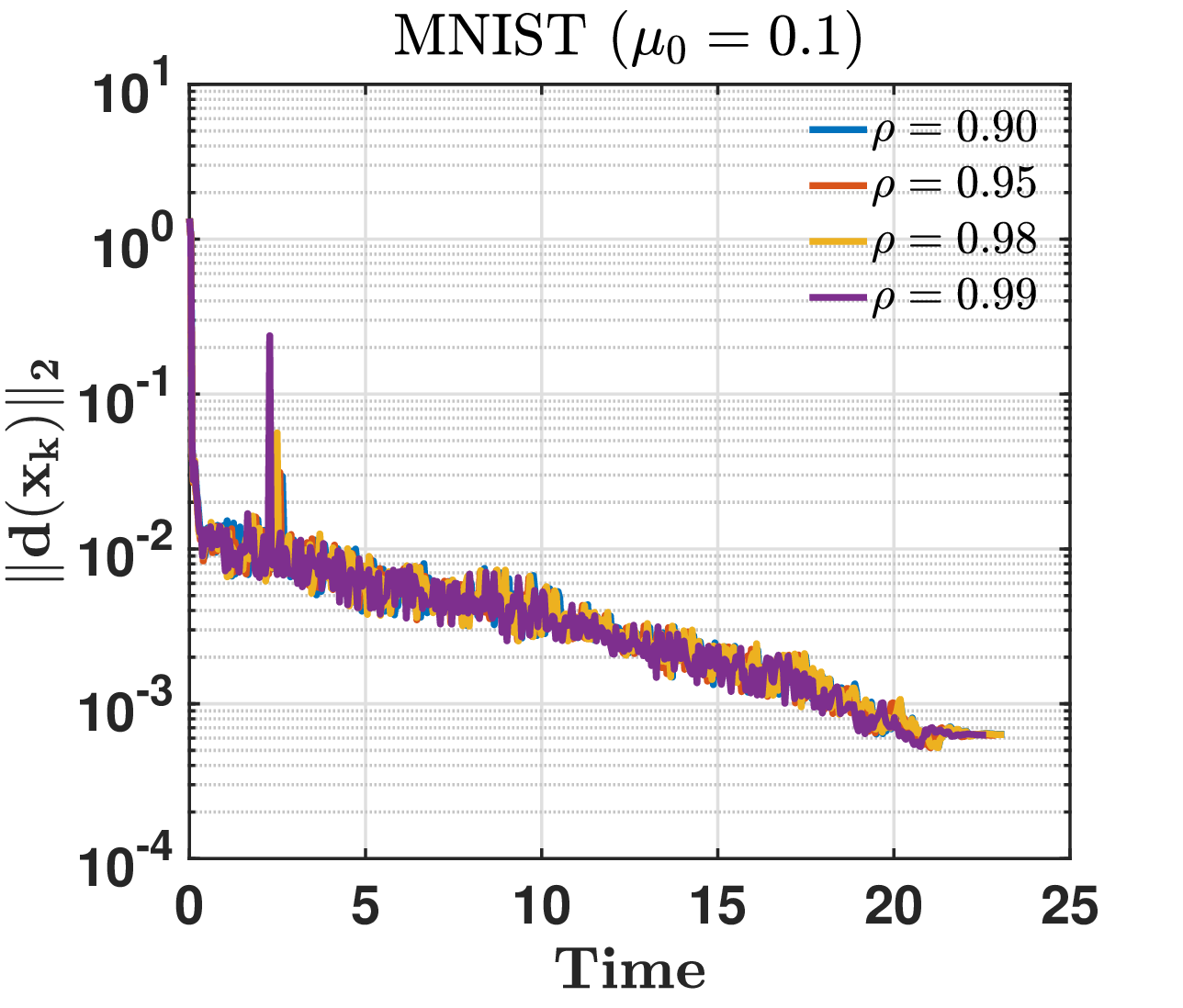}\hspace{-3mm} 
\includegraphics[scale=0.155]{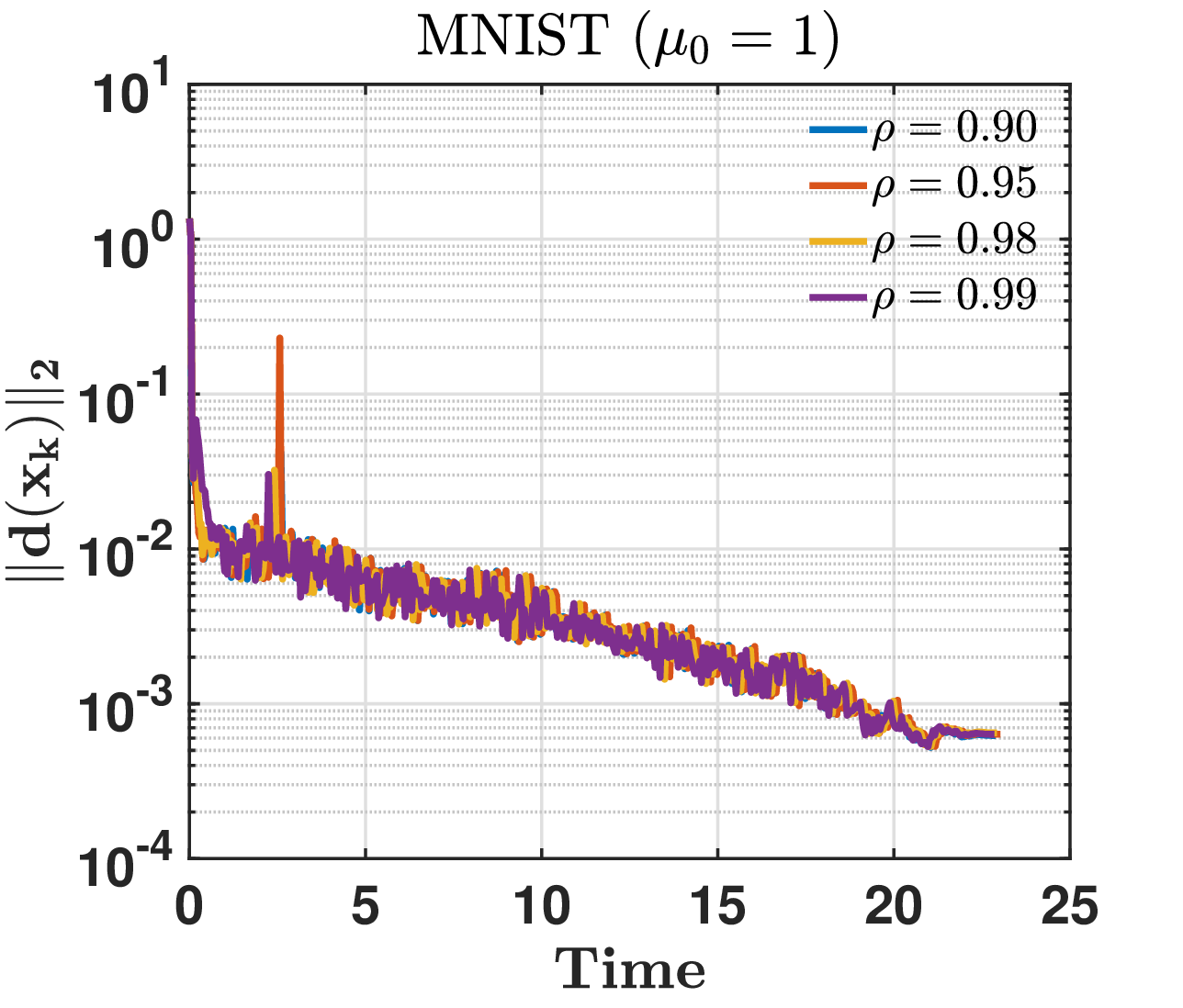}\hspace{-3mm}
\includegraphics[scale=0.155]{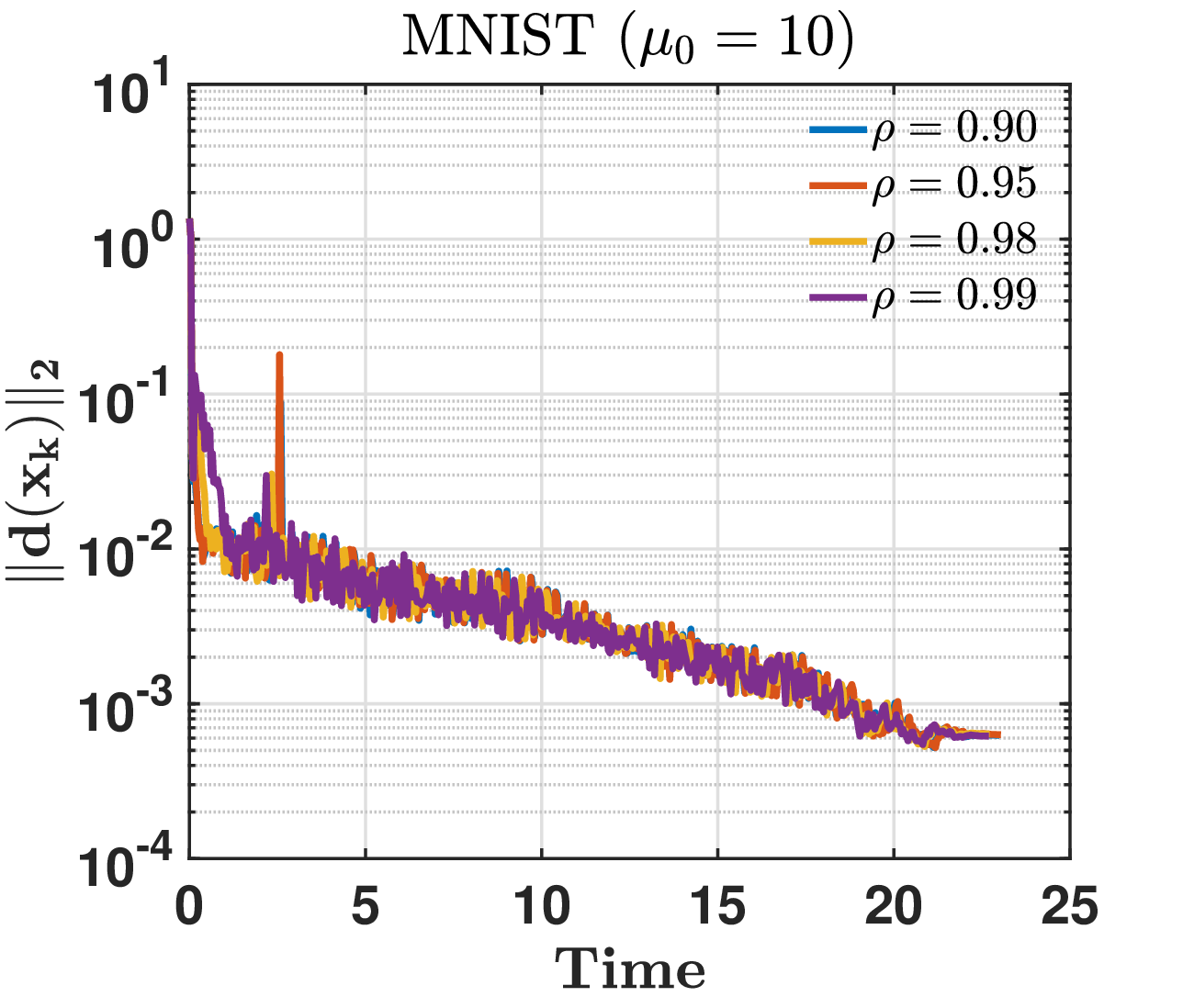}\hspace{-3mm}
\includegraphics[scale=0.155]{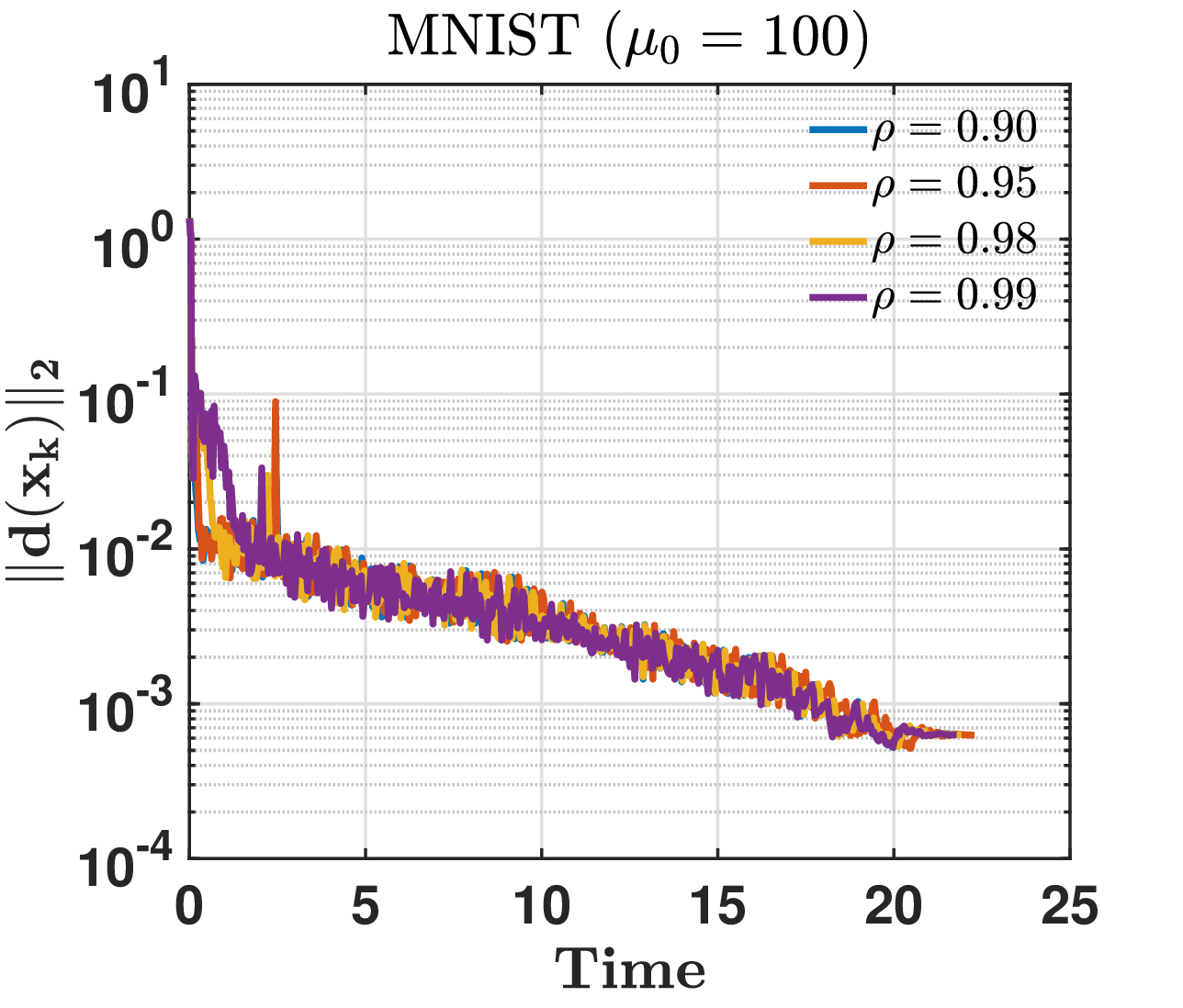}
\includegraphics[scale=0.155]{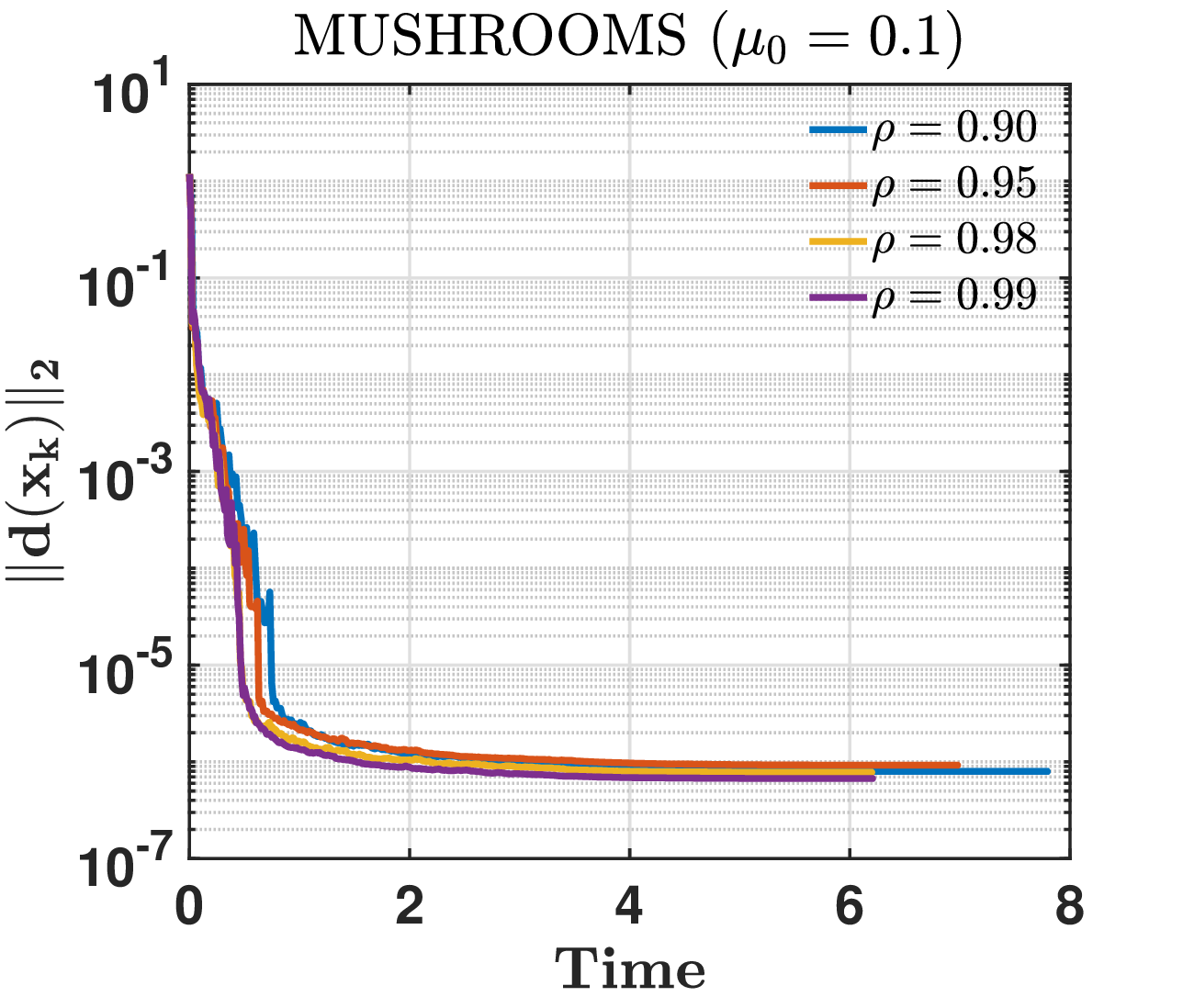}\hspace{-3mm} 
\includegraphics[scale=0.155]{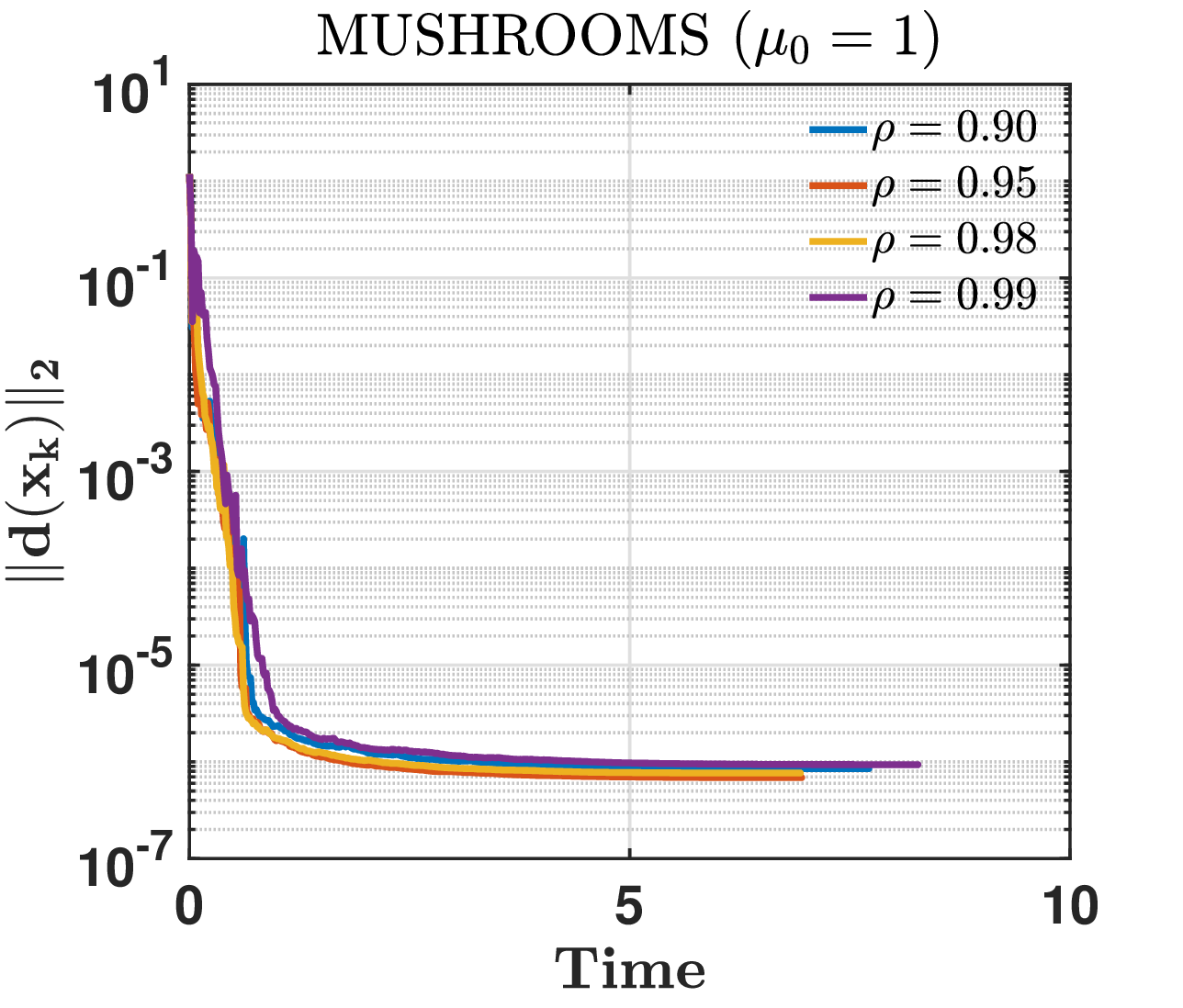}\hspace{-3mm}
\includegraphics[scale=0.155]{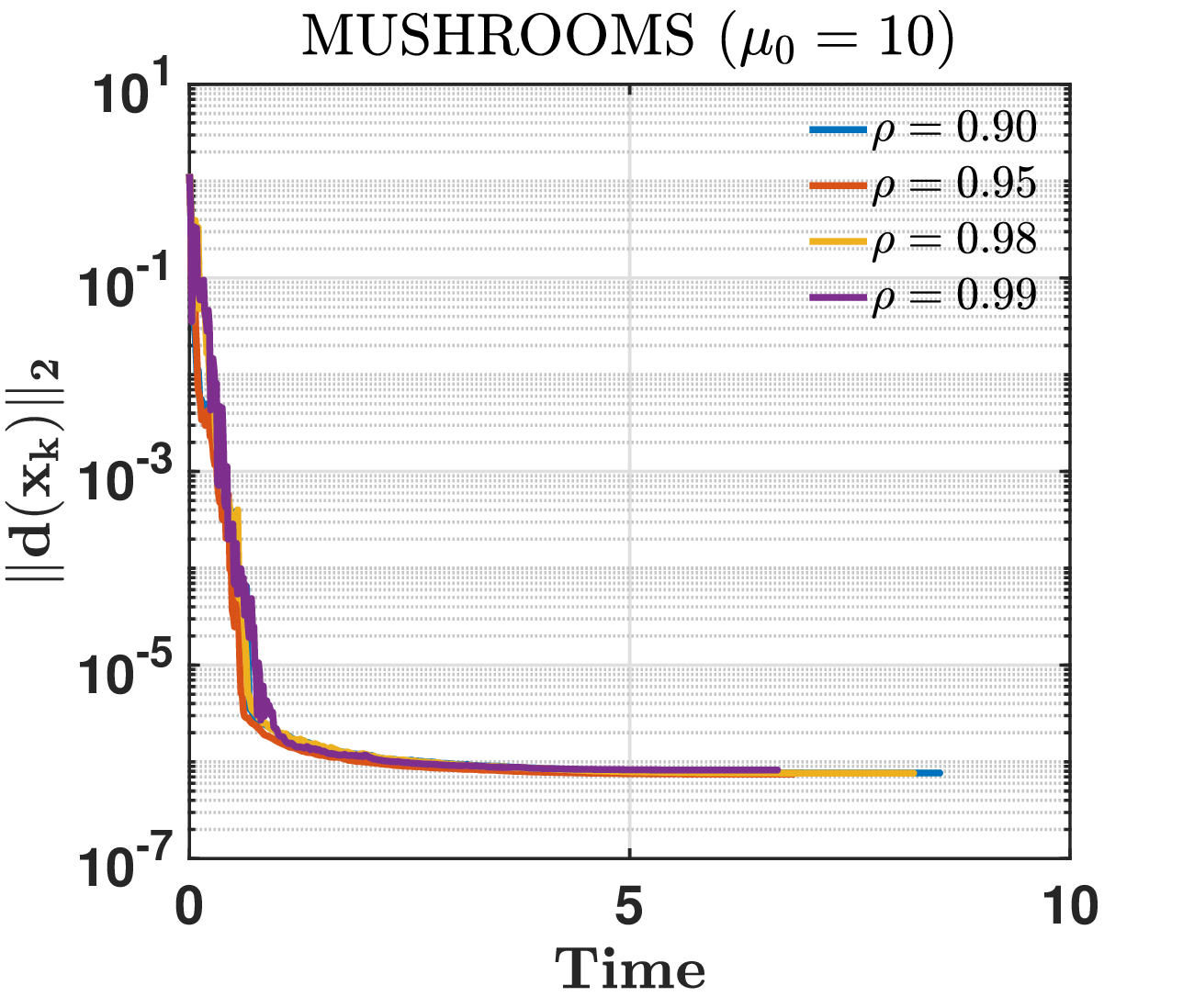}\hspace{-3mm}
\includegraphics[scale=0.155]{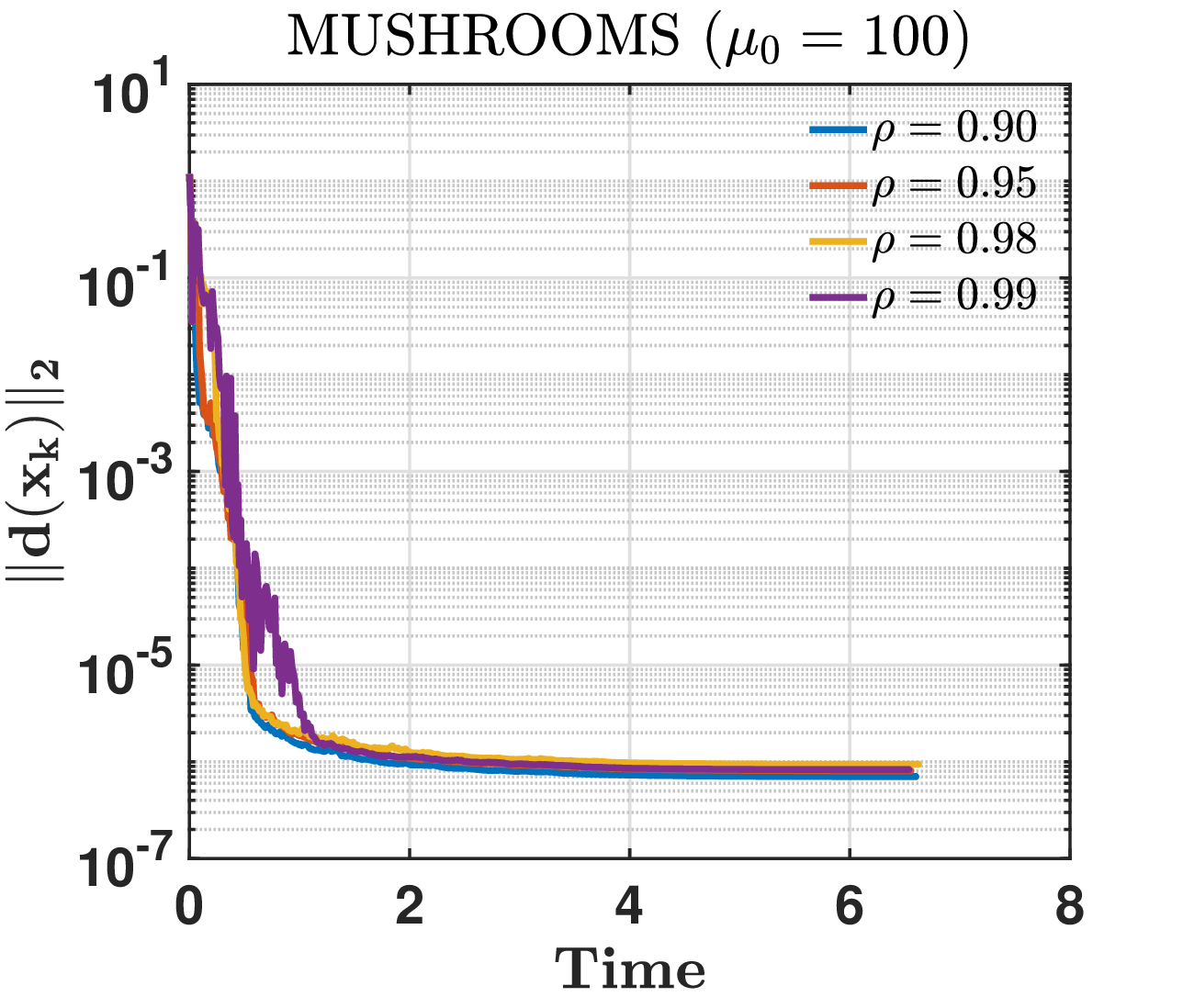}
\end{subfigure}  
\caption{\small Average value of $\|d(x_k)\|$ vs. CPU time obtained by \name\   with strategy \textbf{S2} and the inexact projection for $\eta=0.5$ and varying $\mu_0$, $\rho$.}
  \label{fig:inexact_etamu2}
\end{figure}

\begin{figure}[t]
  \centering
  \begin{subfigure}[t]{\textwidth}
    \centering
    \includegraphics[scale=0.19]{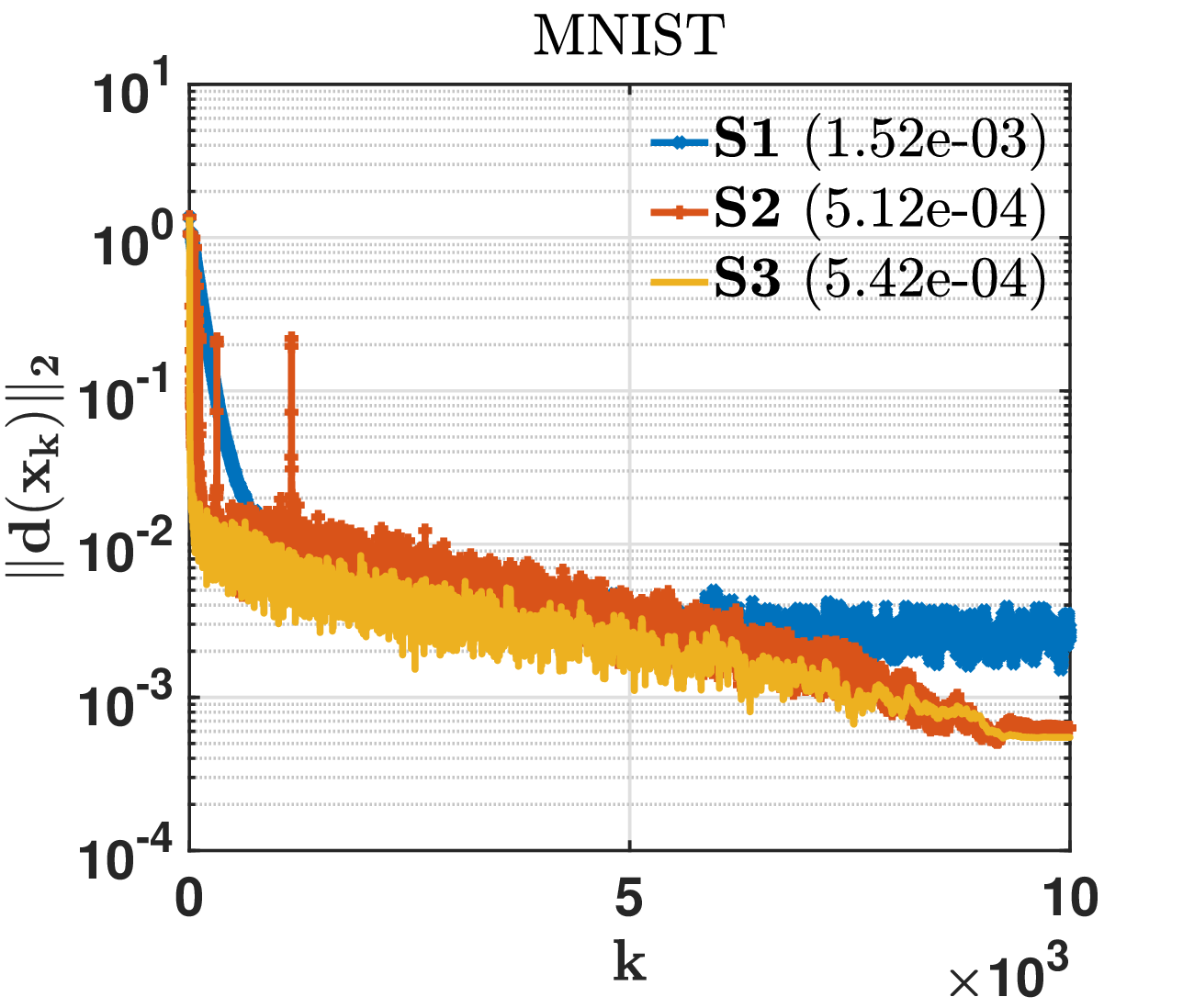}\hspace{3mm} 
    \includegraphics[scale=0.19]{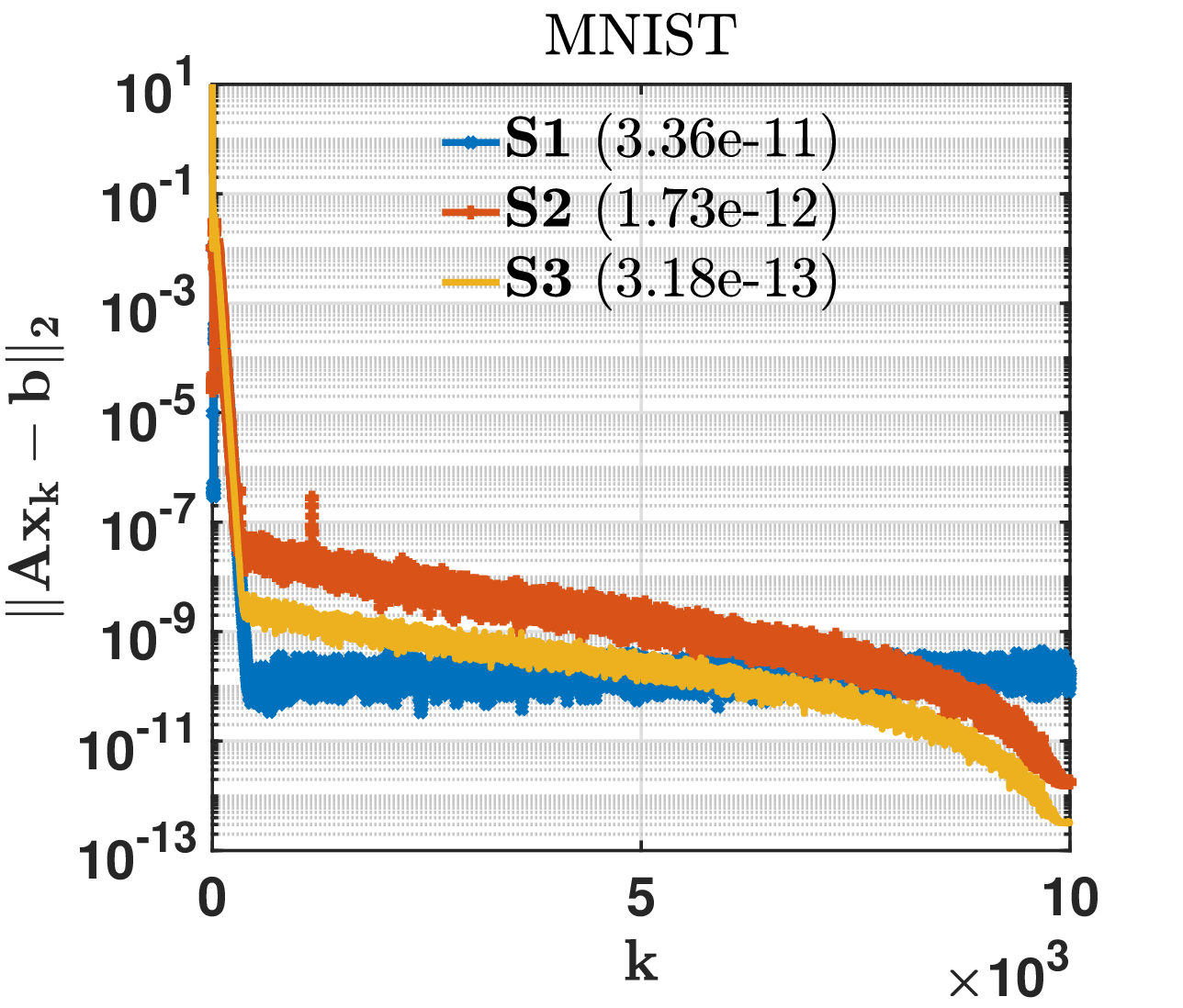}
    \\
    \includegraphics[scale=0.19]{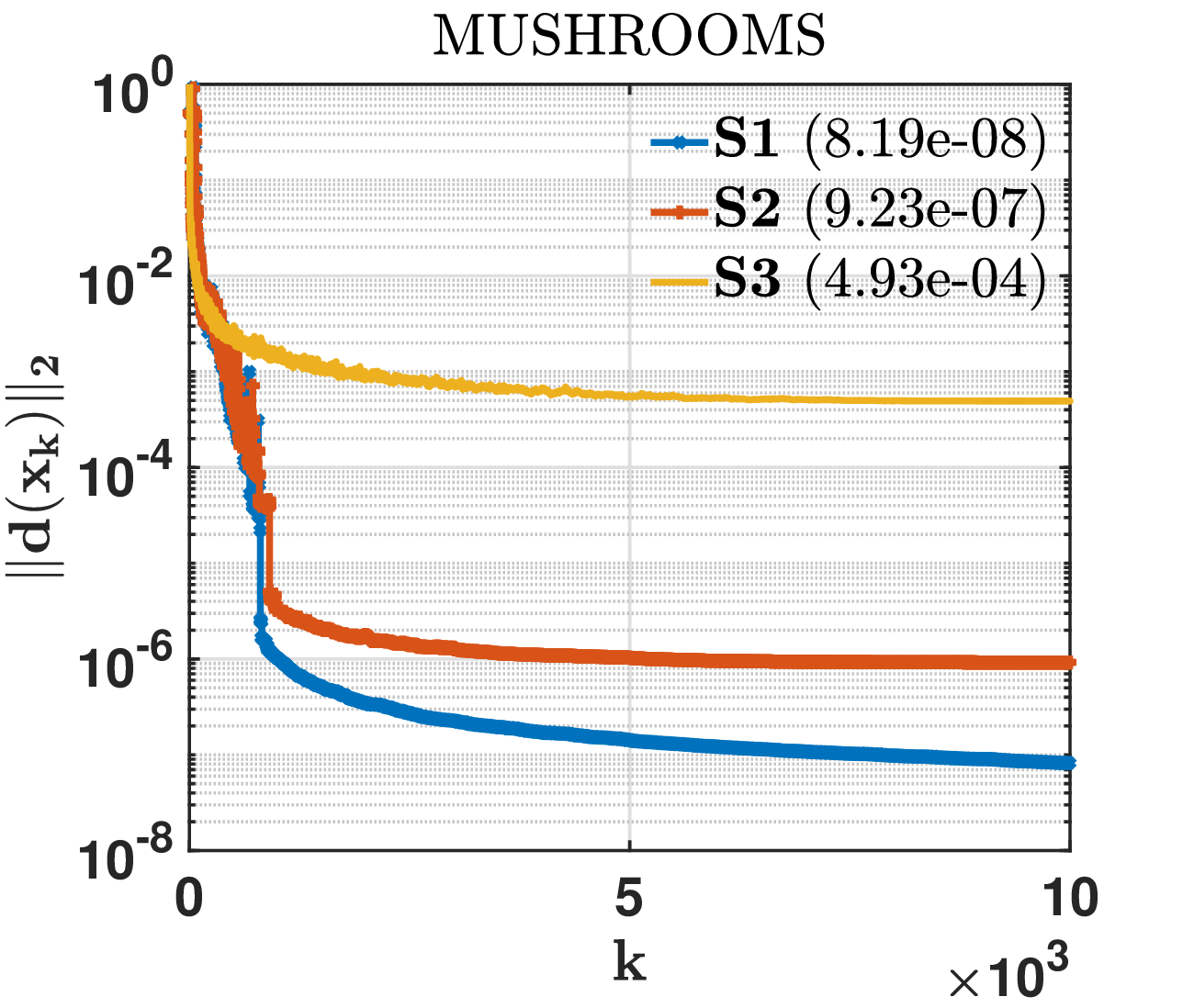}\hspace{3mm} 
    \includegraphics[scale=0.19]{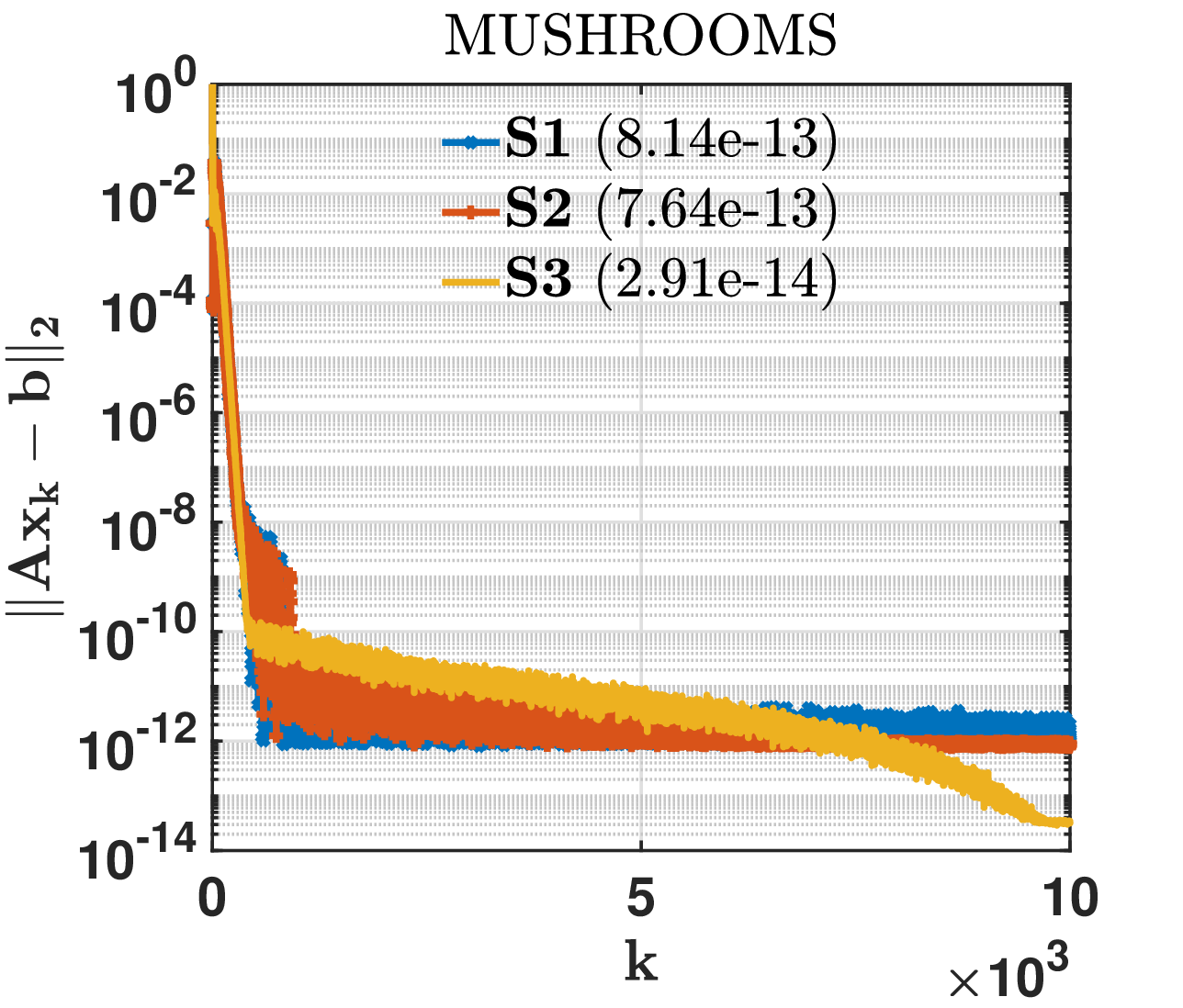}
  \end{subfigure}  
  \caption{\small Average value of $\|d(x_k)\|$ and $\|Ax_k - b\|$ vs. iterations obtained by \name \ with the strategies \textbf{S1}-\textbf{S3}  and  inexact projection.}
  \label{fig:inexact}
\end{figure}

\begin{table}[t]
\centering
\caption{\small {Step-size hyperparameters for PSG\_LECO and SQP\_ADAM for varying $N_b$.}}
\footnotesize
\label{tab:optimalParam2}
\begin{tabularx}{\textwidth}{|l|l|*{3}{>{\centering\arraybackslash}X|}}
\hline
\textbf{Dataset} & \textbf{Methods} & \textbf{$N_b=64$} & \textbf{$N_b=256$} & \textbf{$N_b=0.2N$}  \\ \hline

\textsc{\textbf{Mnist}} & 
\textbf{SQP\_ADAM} ($\bar\alpha$)&
0.001 & 0.001 & 0.001 \\
& \textbf{PSG\_LECO} (\textbf{S2}, $\gamma_0$)&
1 & 1 & 1  \\ \hline\hline
\textsc{\textbf{Mushrooms}} & 
\textbf{SQP\_ADAM} ($\bar\alpha$)&
1&  1&1\\
& \textbf{PSG\_LECO} (\textbf{S2}, $\gamma_0$)&
1 & 10 & 10  \\ \hline\hline
\end{tabularx}
\end{table}
\subsection{Comparison of PSG\_LECO with SQP\_ADAM}\label{Sec.PerfCompSQP}
 In this section we present results from the numerical comparison of \names with the procedure  \sqp\ \cite{curtis_2026}. Both algorithms \names and \sqp\ are first-order objective function-free procedures, and require a stochastic gradient at each iteration. As for the feasibility, \sqp\ was designed with exact projection. It employs a constant step-size $\bar \alpha$ and tuning is required. In~\cite{curtis_2026}, \sqp\ was run  with tuned step-sizes of order $10^{-4}$ to $10^{-2}$  for a batch size corresponding to $0.2N$.

We implemented \names with exact projection and the strategy \textbf{S2} where the update of $\delta_k$ in \eqref{BBdelay_s} was performed every $C=20$ iterations with an additional gradient evaluation is required for such an update. Focusing on \textsc{MNIST} and \textsc{Mushrooms} problems, \names and \sqp\  were tested with  $\gamma_0$ and $\bar\alpha $, respectively, from the set $\mathcal{T}_7 = \{10^{-4}, 10^{-3}, 10^{-2}, 10^{-1}, 1, 10\}$ and with batch size $N_k$ from $\mathcal{T}_3 = \{64,256, 0.2N\}$. The best-performing values of $\bar\alpha$ and $\gamma_0$ are reported in Table~\ref{tab:optimalParam2}, and   Figure~\ref{fig:PSG-SQP}  presents a comparison of the average optimality measure versus the iterations obtained with such parameters and exact projection; the value of $\|d(x_k)\|$ is displayed every 20 iterations for readability.

We note that both \names and  \sqp\ are reliable. In most cases, \names shows a faster decrease of the optimality measure in the initial phase of the process and is competitive with \sqp. \names achieves lower values of $\|d(x_k)\|$ on \textsc{Mnist}, a significantly lower optimality measure for $N_b=64$ on \textsc{Mushrooms}, and satisfactory optimality measures for the remaining mini-batch sizes ($N_b= 256$ and $0.2N$) on \textsc{Mushrooms} though \sqp\ is more accurate. This suggests the selection \textbf{S2} of the step-size compares well with the scaling and momentum formula in \sqp.

\begin{figure}[t]
\centering
\begin{subfigure}[t]{\textwidth}
\centering
\includegraphics[scale=0.19]{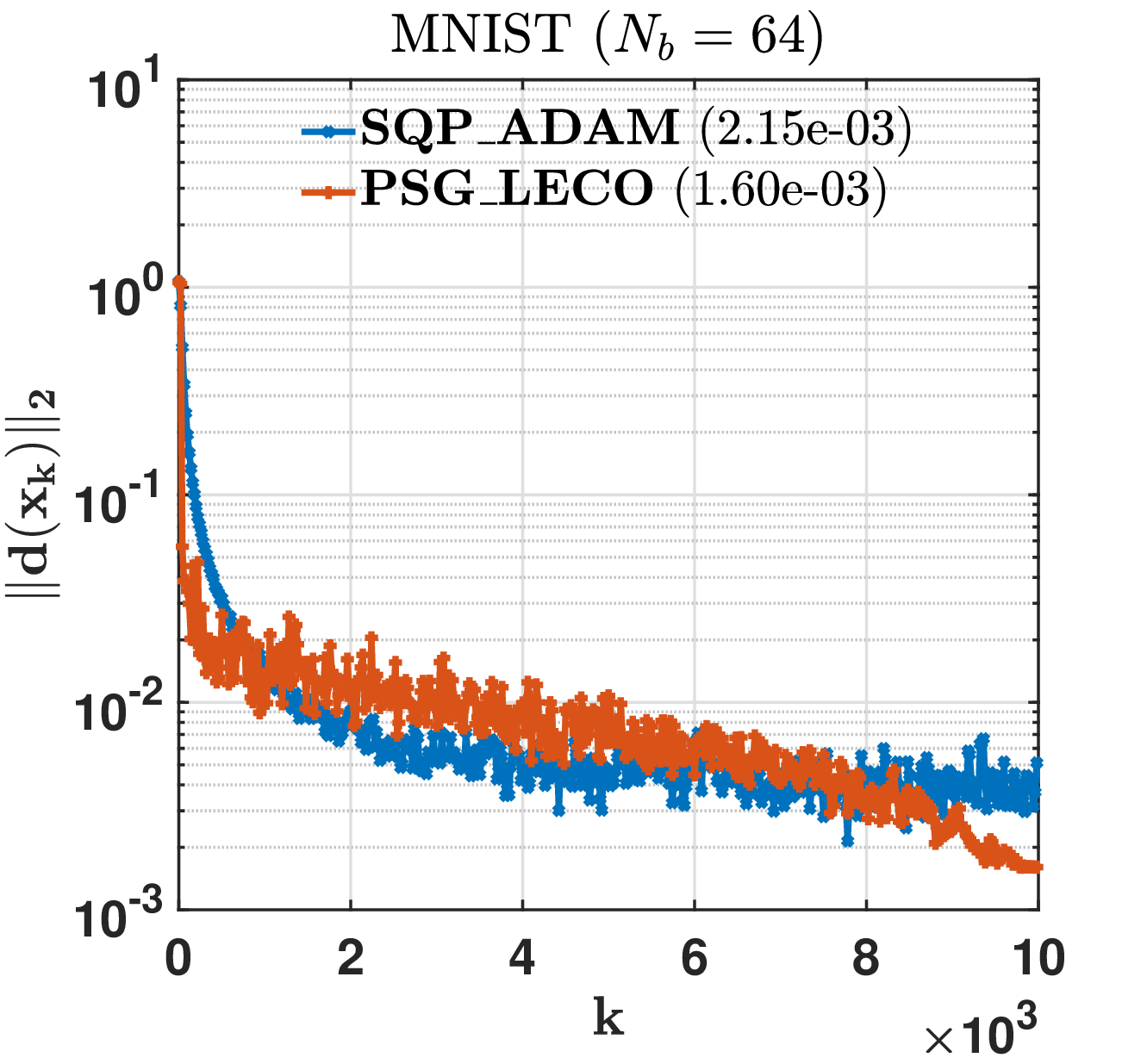}\hspace{-3mm}  
\includegraphics[scale=0.19]{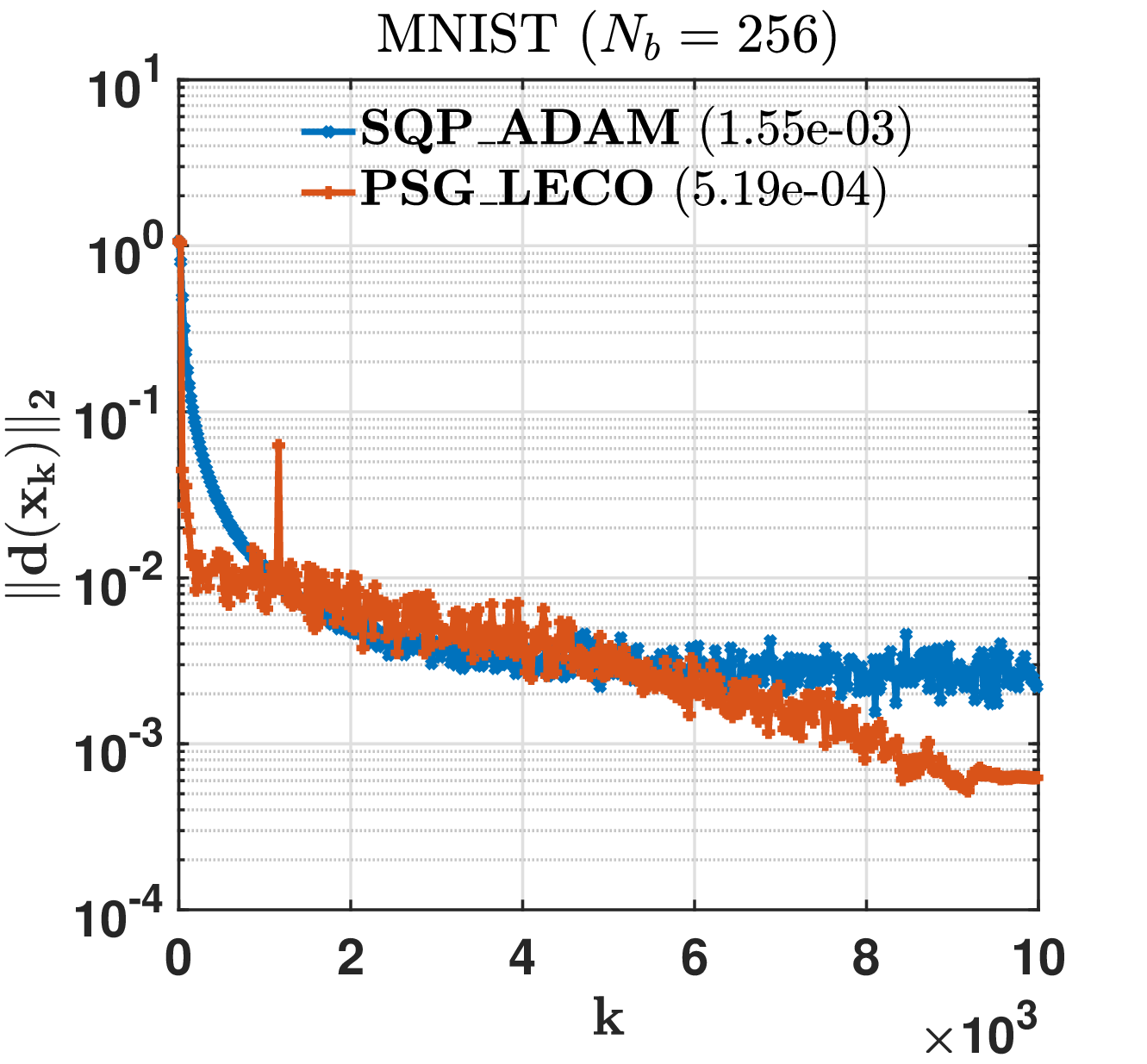}\hspace{-3mm}
\includegraphics[scale=0.19]{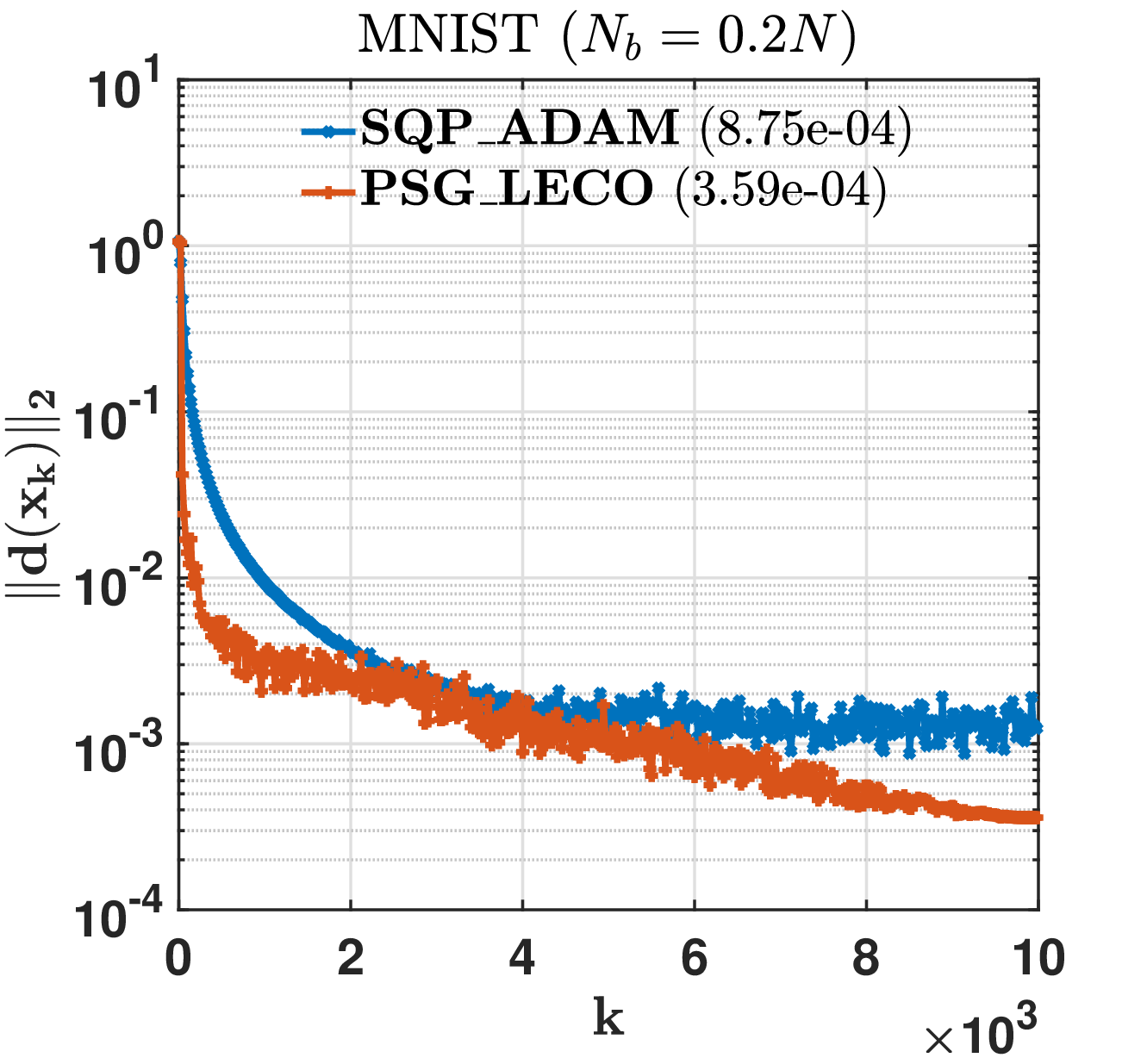}
\end{subfigure}  
\medskip
\begin{subfigure}[t]{\textwidth}
\centering
\includegraphics[scale=0.19]{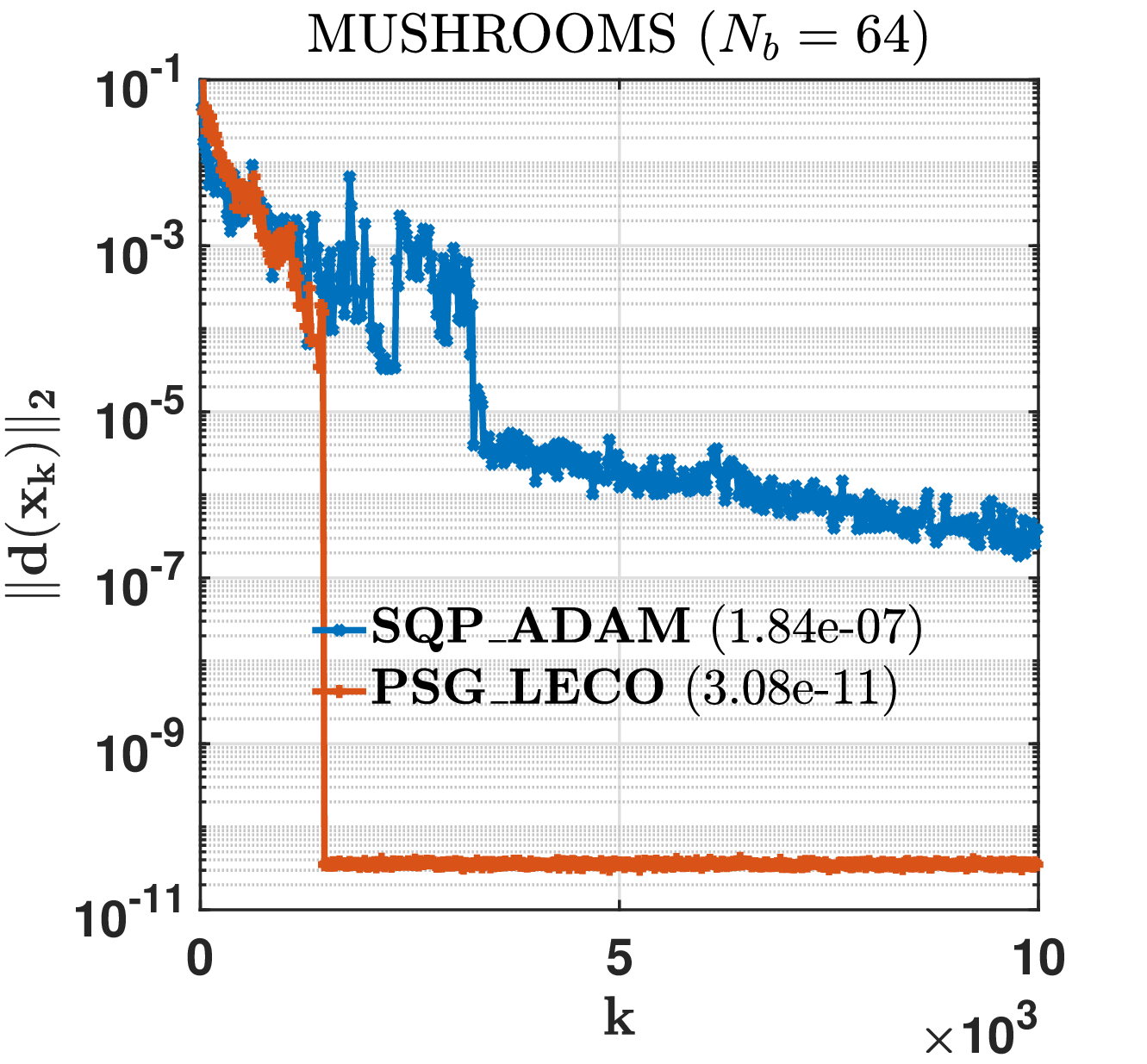}\hspace{-3mm} 
\includegraphics[scale=0.19]{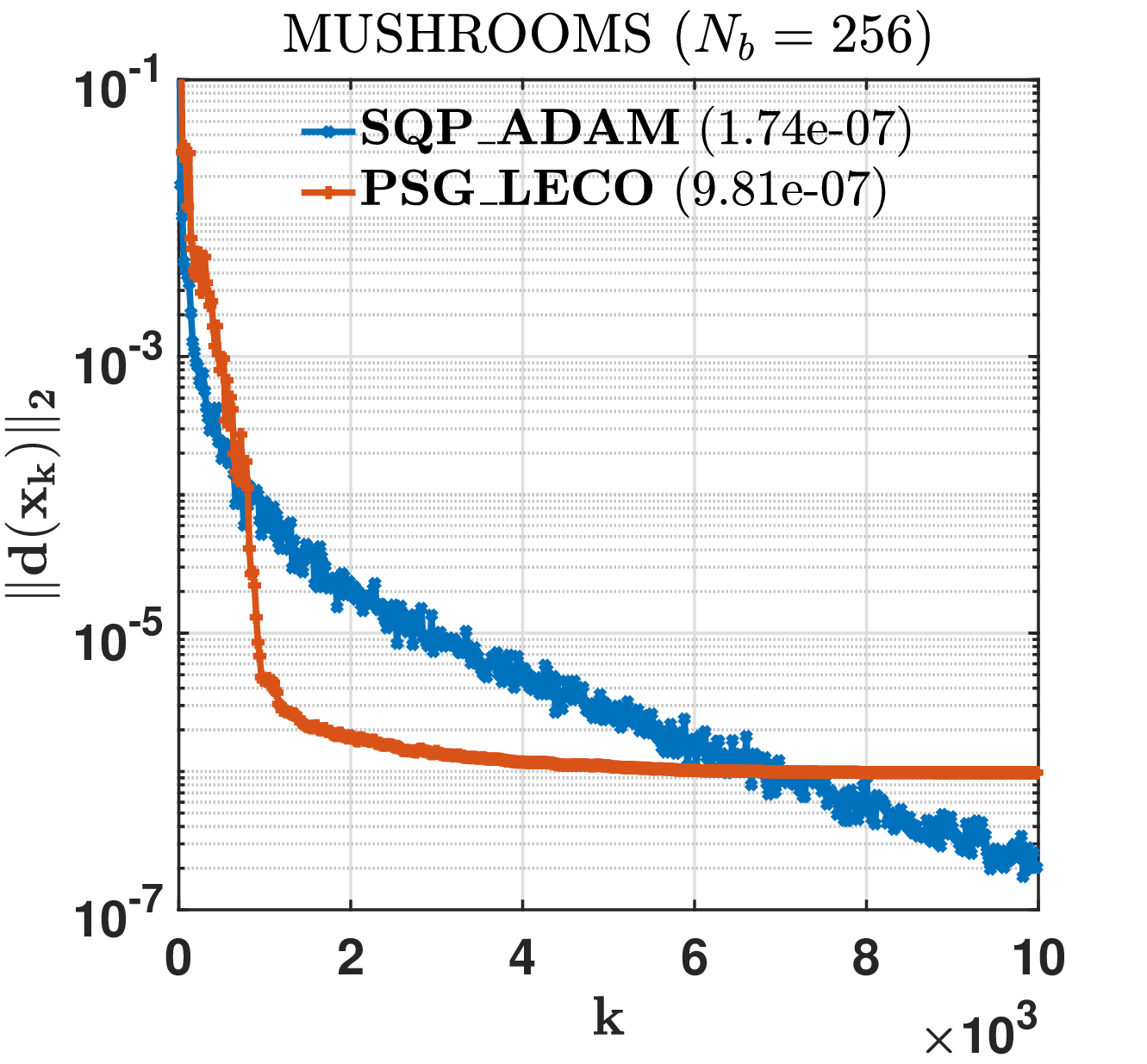}\hspace{-3mm} 
\includegraphics[scale=0.19]{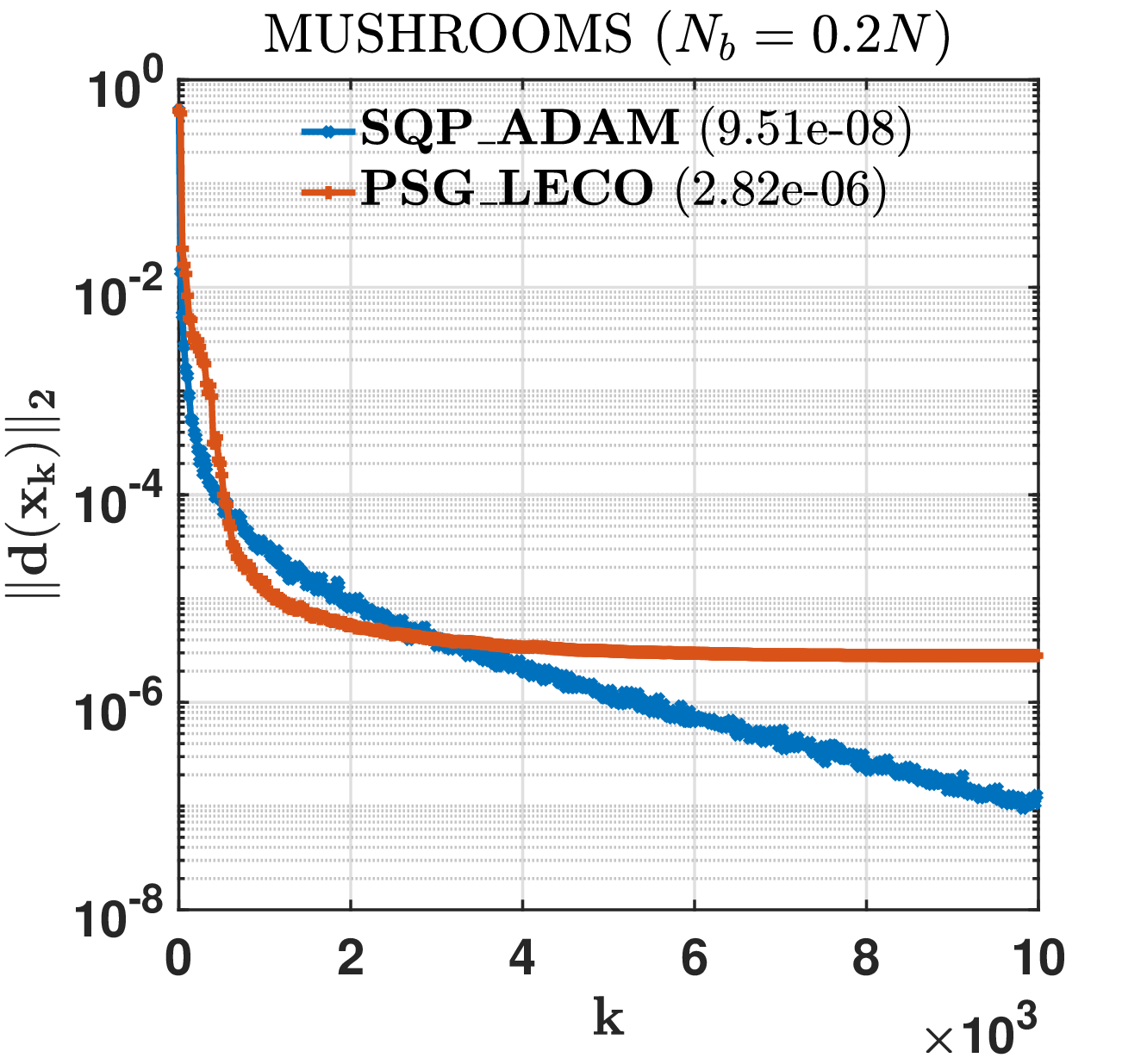}
\end{subfigure}
    \caption{\small Average values of $\|d(x_k)\|$ vs. iterations obtained by  \sqp\ Algorithm and \names Algorithm with strategy  \textbf{S2} and the exact projection.}
    \label{fig:PSG-SQP}
\end{figure}

\subsection{Comparison of PSG\_LECO with IPAS}\label{Sec.PerfCompIPAS}

We conclude our numerical validation by comparing our method with Algorithm \ipas\ \cite{krejic2025ipas}. 
Both methods use inexact projections; unlike \name, \ipas\ uses a stochastic line search procedure, and thus function evaluations, that monitors decrease at each iteration and dynamically adjusts the batch size.
To balance computational cost and optimality progress, the batch size in \ipas\ is increased only when the line search test fails. This algorithm was tested using its authors' code with the hyperparameter settings from \cite{krejic2025ipas}. \names was tested using the strategy \textbf{S2} with $\gamma_0$ as in Table \ref{tab:optimalParam},  $\eta=0.5$,  $\mu_0=0.1$, $\rho=0.95$ in (\ref{resb}), and choosing $N_b$ from $\mathcal{T}_3 = \{64,256, 0.2N\}$.
Both methods used the same initial point $x_0=0$.

For \ipas, we report in Figure~\ref{fig:IPS_Nk} the evolution of the mini-batch size versus iterations, and observe that $N_b$ remains considerably smaller than $N$. In Figure~\ref{fig:PSG-IPS}, we display the values of $\|d(x_k)\|$ vs. iterations (top) and CPU time (bottom), plotting one value every 20 iterations for readability. From the first row, obtained with the fixed maximum number of iterations $k_{\max}=10^4$, we observe that \names\ achieves a significantly faster decrease and lower optimality values compared to \ipas. Analogous conclusions can be drawn from the second row that displays results in term of 165 seconds for  \textsc{Mnist} and 30 seconds for \textsc{Mushrooms}; the latter results were obtained setting the maximum number of iterations to $k_{\max}=7\cdot 10^4$ for \textsc{Mnist} and $k_{\max}=12\cdot 10^4 $ for \textsc{Mushroom} and plotting the results for the previously specified amount of cpu time  which the largest time budget common to all the experiments reported.}
\begin{figure}[H]
\centering
\begin{subfigure}{\textwidth}
\centering
\includegraphics[scale=0.22]{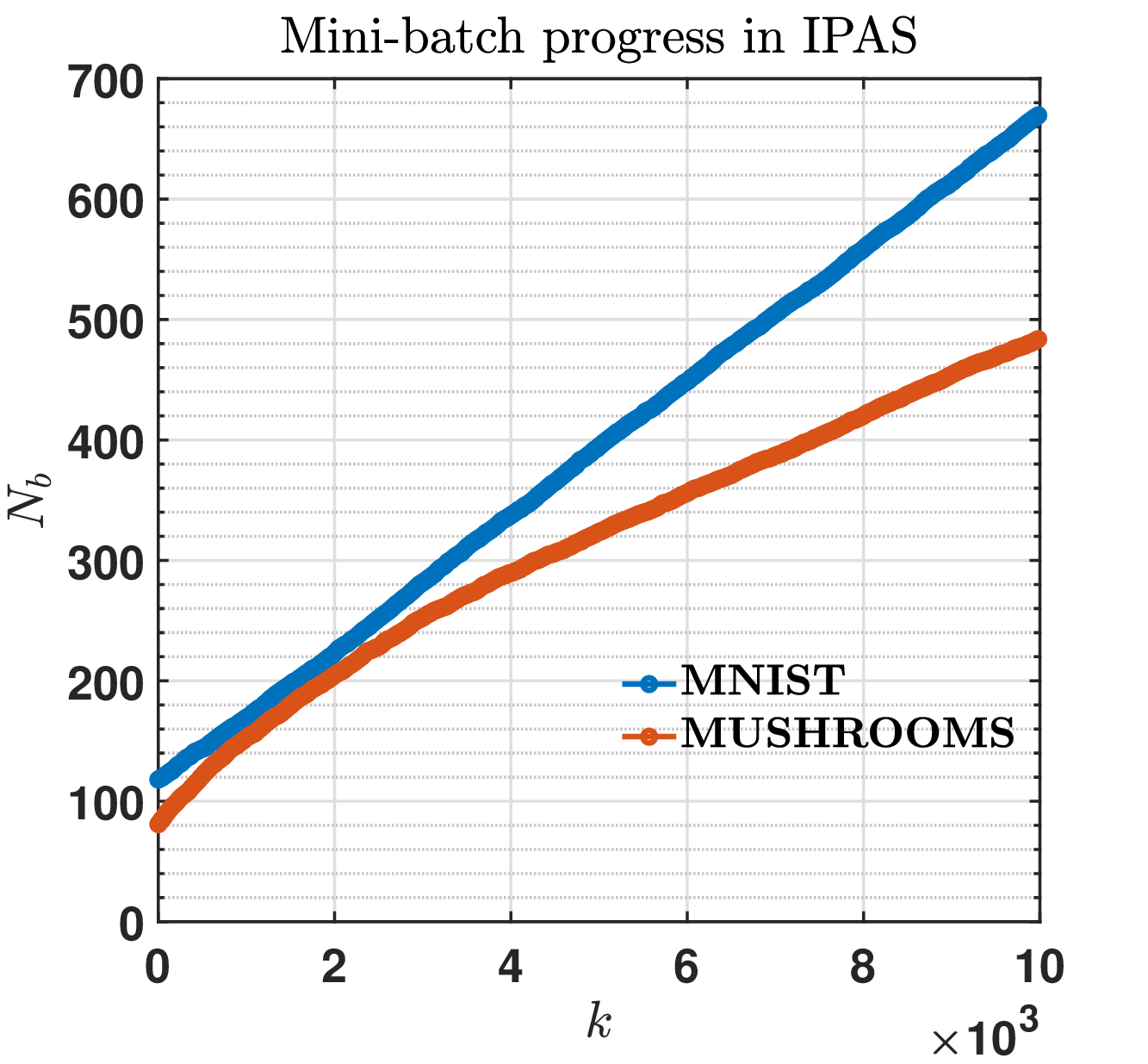}
\end{subfigure} 
\medskip
\caption{\small The progress of adaptive mini-batch in \ipas.}
\label{fig:IPS_Nk}
\end{figure}

\begin{figure}[H]
\centering
\begin{subfigure}{\textwidth}
\centering
\includegraphics[scale=0.2]{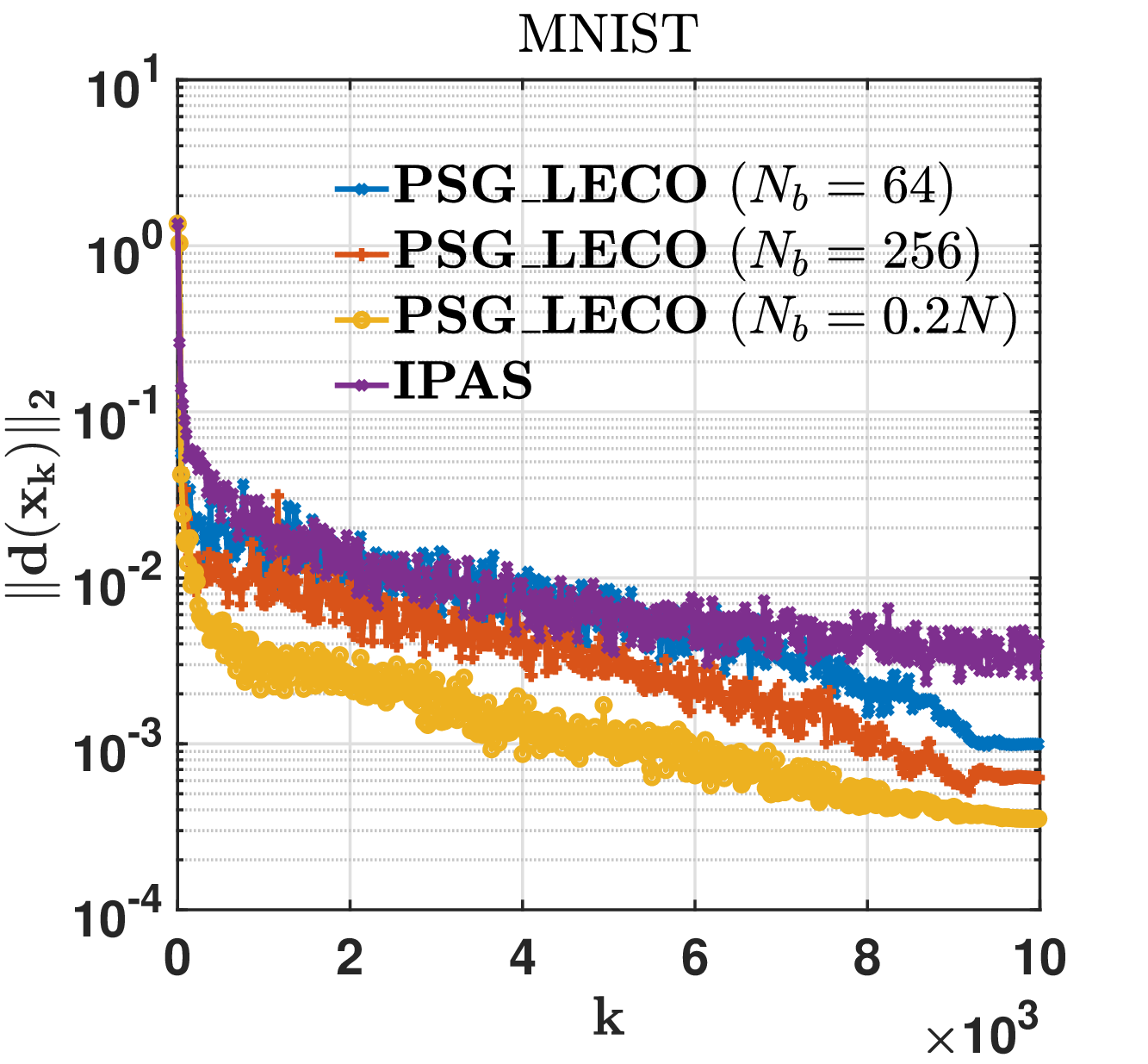}\hspace{-3mm}    
\includegraphics[scale=0.2]{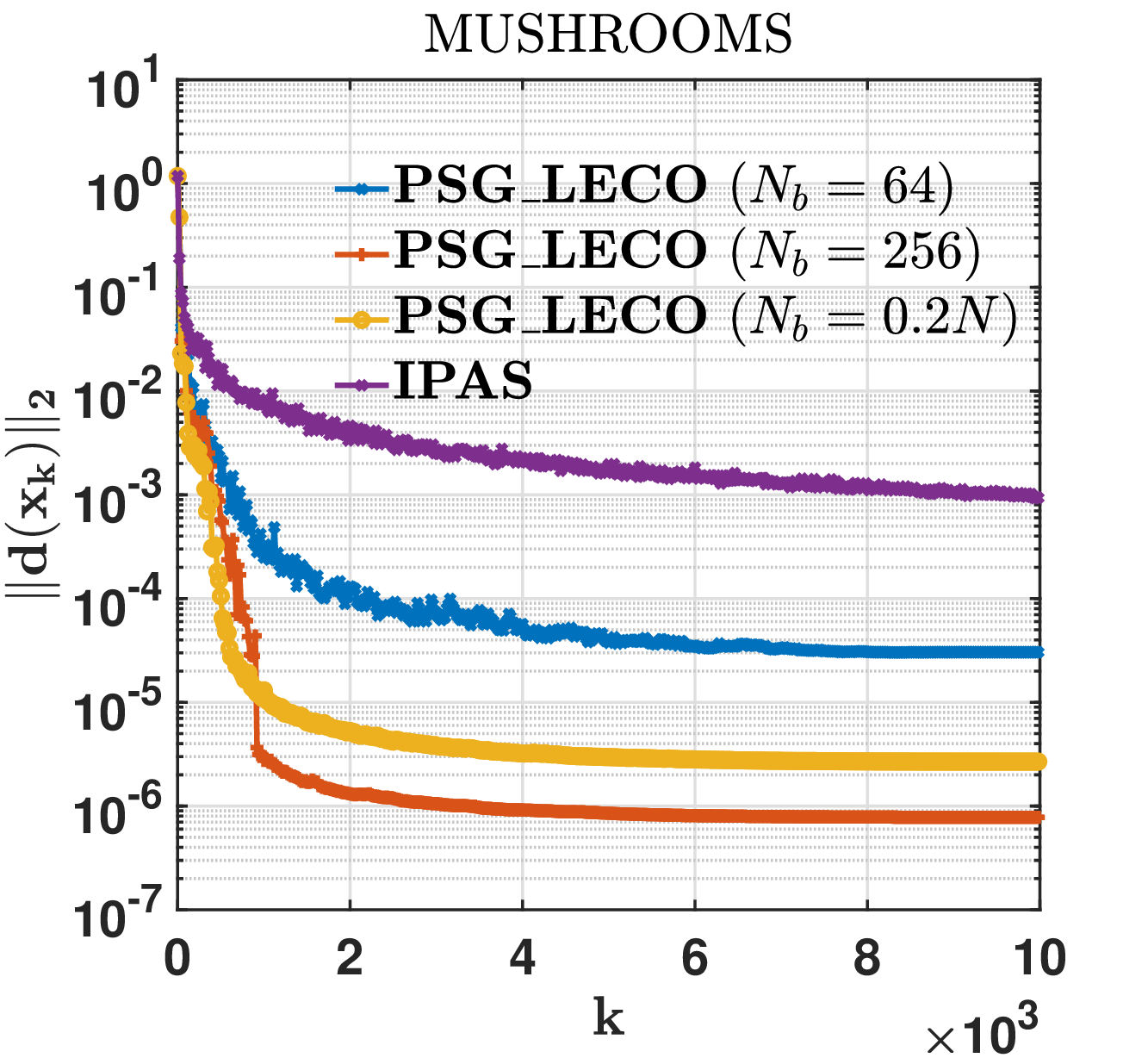}\hspace{-3mm} 
\includegraphics[scale=0.21]{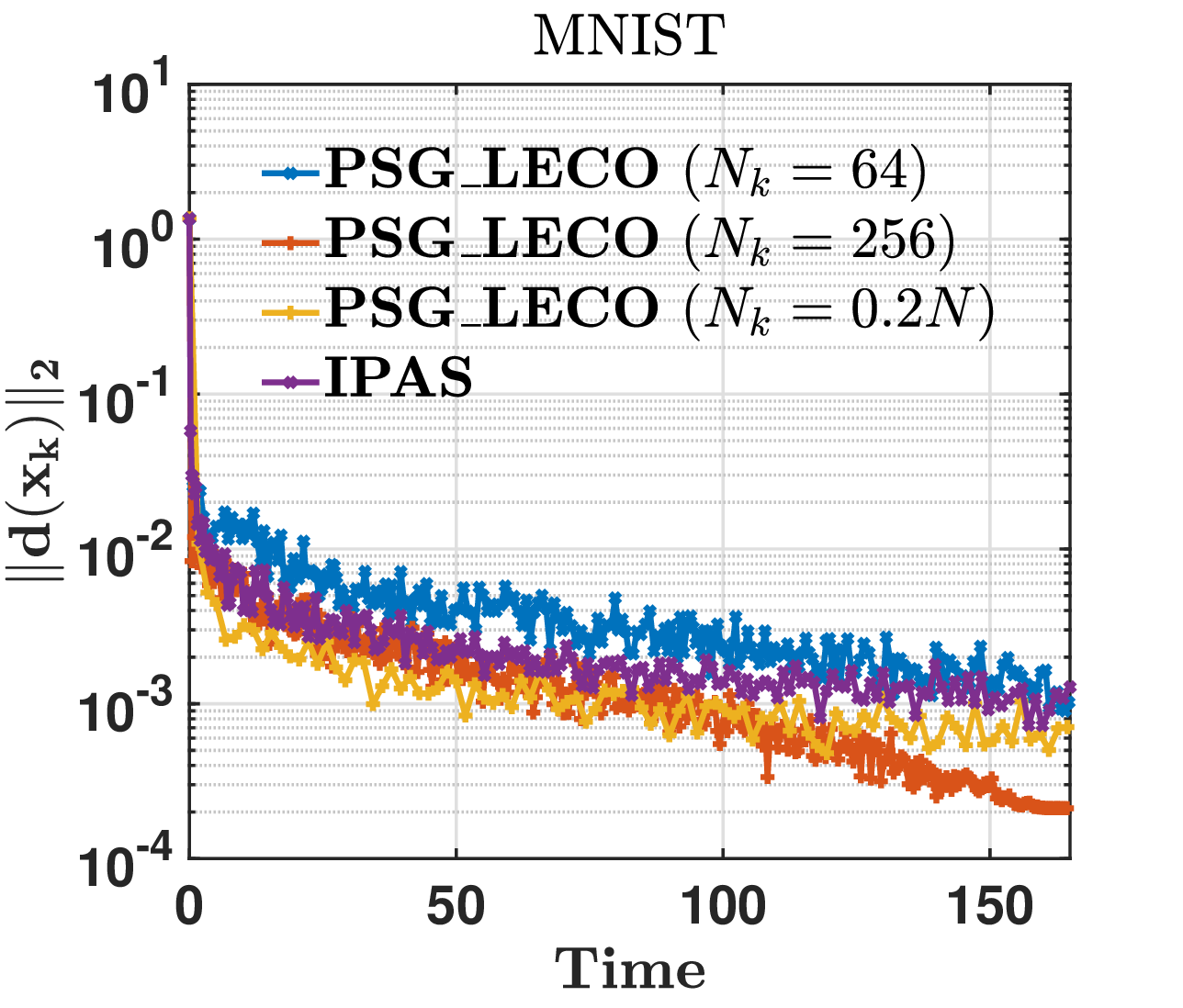}\hspace{-3mm}    
\includegraphics[scale=0.21]{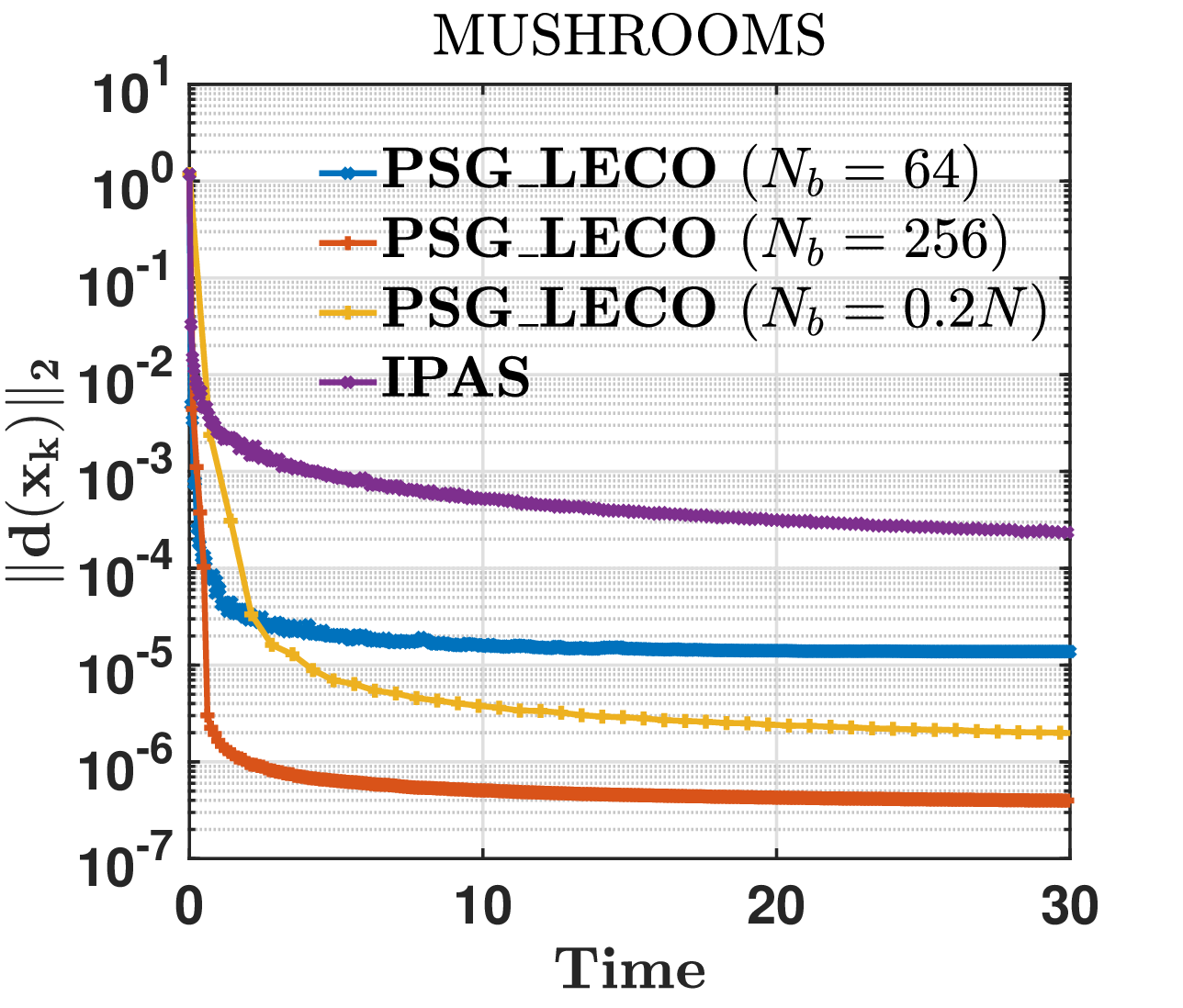}\hspace{-3mm}  
\end{subfigure} 
\medskip
\caption{\small Average values of $\|d(x_k)\|$ versus iterations (top) and CPU time (bottom) obtained by \ipas\ and \name\ using strategy \textbf{S2}.
}
\label{fig:PSG-IPS}
\end{figure}
The bottom panels show results obtained within a computational budget of 80 seconds for \textsc{Mnist} and 30 seconds for \textsc{Mushrooms}. For each $N_b$, ten seeds may result in different numbers of iterations within the same computational time. To enable a meaningful averaging across seeds,
for each run the computational time was recorded through the iterations. Then, a uniform time grid consisting of $10^4$  points was considered and 
the average value of the optimality at such points was computed  using  values obtained by linear interpolation. The plots confirm that
\names compares well with \ipas; namely, the performance of such a method does not suffer from the absence of a test for the acceptance of the iterates.
\section{Conclusions}\label{sec.con}
In this work, we proposed a projected stochastic gradient method for minimizing a function subject to deterministic linear constraints. Our method involves a projection map that can be evaluated either exactly or inexactly. Theoretical properties depend on the choice of a step-size related sequence and are equivalent to those established in the unconstrained setting. Numerical illustration of our procedure was presented to show its effectiveness.

\section*{Acknowledgments}
Natasa Krklec Jerinki{\'c} was supported by the Science Fund of the Republic of Serbia, GRANT No 7359, Project LASCADO. 
\\
Benedetta Morini and Mahsa Yousefi are members of the INdAM Research Group GNCS. The research that led to the present paper was partially supported by INDAM-GNCS through Progetti di Ricerca 2026.\\
The authors wish to thank Qi Wang and Yulang Zhu for providing the code for \sqp, and Luka Rute{\v{s}}i{\'c} for providing the code for \ipas.
\section*{Data availability }
The datasets utilized in this research are publicly accessible and commonly employed benchmarks in the field of machine learning and numerical optimization, see \cite{gratton2025s2mpj,chang2011libsvm}.
\section*{Declarations}

\textbf{Conflict of interest}

The authors have no relevant financial or non-financial interests to disclose.


\end{document}